\documentclass[a4paper]{amsart}

\usepackage[a4paper,width={16cm},left=2.5cm,bottom=3cm, top=3cm]{geometry}
\usepackage{enumerate}

\usepackage[english]{babel}

\usepackage{a4wide}
\usepackage[dvipsnames]{xcolor}
\usepackage{amsmath}
\usepackage{amssymb}
\usepackage{amsfonts}
\usepackage{amsthm,graphicx}
\usepackage{comment}
\usepackage{csquotes}

\usepackage{hyperref}
\hypersetup{
	linktocpage=true,
	colorlinks=true,
	linkcolor=blue!70!black,
	citecolor=orange!40!red,
	urlcolor=magenta!80!black,
}

\usepackage{mathtools}
\usepackage{cleveref}

\numberwithin{equation}{section}
\newtheorem{theorem}{Theorem}[section]
\newtheorem{lemma}[theorem]{Lemma}

\newtheorem{proposition}[theorem]{Proposition}
\theoremstyle{definition}

\newtheorem{remark}[theorem]{Remark}

\DeclareRobustCommand{\rchi}{{\mathpalette\irchi\relax}}
\renewcommand{\Bbb}{\mathbb}
\newcommand{\irchi}[2]{\raisebox{\depth}{$#1\chi$}} 
\newcommand{\mc}{\mathcal}
\newcommand{\mb}{\mathbb}

\newcommand{\la}{\lambda}
\newcommand{\norm}[1]{\left\lVert#1\right\rVert}
\newcommand{\pd}[2]{\frac{\partial#1}{\partial#2}}

\newcommand{\R}{\mb{R}}
\newcommand{\C}{\mb{C}}
\newcommand{\N}{\mb{N}}
\newcommand{\Z}{\mb{Z}}

\newcommand{\e}{\varepsilon}

\newcommand{\ps}[3]{\left( #2, #3 \right)_{#1}}

\newcommand{\colg}[1]{{\color{teal} #1}}

\usepackage{subfig}
\usepackage{tikz}
\usepackage{xcolor}
\usepackage{pgfplots}
\pgfplotsset{compat=1.18}
\usepgfplotslibrary{fillbetween}
\usepackage{mathrsfs}
\usetikzlibrary{arrows}
\usetikzlibrary{shadings}
\makeatletter

\begin{document}

\title[On Aharonov-Bohm operators with many coalescing poles]{Quantitative spectral stability for Aharonov-Bohm operators with many coalescing poles}

\author[V. Felli, B. Noris, R. Ognibene, and G. Siclari]{Veronica Felli, Benedetta Noris, Roberto Ognibene, and Giovanni Siclari}

\address{Veronica Felli and Giovanni Siclari 
  \newline \indent Dipartimento di Matematica e Applicazioni
  \newline \indent
Universit\`a degli Studi di Milano–Bicocca
\newline\indent Via Cozzi 55, 20125 Milano, Italy}
\email{veronica.felli@unimib.it,  g.siclari2@campus.unimib.it}

\address{Benedetta Noris
  \newline \indent Dipartimento di Matematica
  \newline \indent Politecnico di Milano
  \newline\indent Piazza Leonardo da Vinci 32, 20133 Milano, Italy}
\email{benedetta.noris@polimi.it}

\address{Roberto Ognibene
	\newline \indent Dipartimento di Matematica
	\newline \indent Universit\`a di Pisa
	\newline\indent  Largo Bruno Pontecorvo, 5, 56127 Pisa, Italy}
\email{roberto.ognibene@dm.unipi.it}

\date{Revised version, December 7, 2023.}

\begin{abstract}
  The behavior of simple eigenvalues of Aharonov-Bohm operators with
  many coalescing poles is discussed.  In the case of half-integer
  circulation, a gauge transformation makes the problem equivalent to
  an eigenvalue problem for the Laplacian in a domain with straight
  cracks, laying along the moving directions of poles. For this
  problem, we obtain an asymptotic expansion for eigenvalues, in which
  the dominant term is related to the minimum of an energy functional
  associated with the configuration of poles and defined on a space of
  functions suitably jumping through the cracks.
    
  Concerning configurations with an odd number of poles, an accurate
  blow-up analysis identifies the exact asymptotic behaviour of
  eigenvalues and the sign of the variation in some cases. An
  application to the special case of two poles is also discussed.
\end{abstract}

\maketitle


\medskip 


\section{Introduction}\label{sec_introduction}
This paper deals with asymptotic expansions of the eigenvalue
variation for Aharonov-Bohm operators with many coalescing poles with
half-integer circulation, under homogeneous Dirichlet boundary conditions on
a simply connected open bounded domain $\Omega\subset \R^2$.  More
precisely, we consider the case of any number $k$ of poles moving
along straight lines towards a collision point $P\in\Omega$, with
distances from $P$ vanishing with the same order. Without loss of
generality, we assume that $P=0\in \Omega$, so that the moving poles
can be written as multiples of $k$ fixed points
$\{a^j\}_{j=1,\dots,k}$ with the same multiplicative infinitesimal
parameter $\e>0$.

Since we are interested in the asymptotic behaviour of  eigenvalues
as $\e\to0^+$, it is not restrictive to assume that there exists $R<1$ such
that
\begin{equation*}
  \{a^j\}_{j=1,\dots,k} \subset D_R(0)\subset \Omega,
\end{equation*}
where, for every $r>0$ and $x \in \R^2$, we denote $D_r(x):=\{y \in
\R^2: |x-y|<r\}$. Henceforth, we denote $D_r(0)$ simply by
$D_r$.

We assume that, among the $k$ poles, there are $k_1$ poles that stand
alone on their own straight line through the origin, while the
remaining ones form $k_2$ pairs of poles staying on the same straight
line but on different sides with respect to the origin. Hence 
\begin{equation*}
k=k_1 +2k_2 \quad \text{with  } k_1,k_2 \in \mathbb{N},  \ (k_1,k_2) \neq (0,0),
\end{equation*}
and, for every $j=1,\dots, k$, there exist $r_j>0$ and $\alpha^j\in
(-\pi,\pi]$ such that $\alpha^{j}\neq\alpha^\ell$ if $j\neq\ell$ and
\begin{equation}\label{def_aj}
a^j= r_j(\cos(\alpha^j),\sin(\alpha^j)),
\end{equation}
where $\alpha^{j_1} \neq \alpha^{j_2}\pm \pi$ if $j_1 \neq j_2$ and
$j_1,j_2\in \{1,\dots, k_1\}$, while $\alpha^j\in (-\pi,0]$ and
$\alpha^{j+k_2}=\alpha^j+\pi$ for every
$j \in \{k_1+1, \dots, k_1+k_2\}$.  For the sake of simplicity, we
treat in detail configurations of the type described above, see
\Cref{fig:f1}; in \Cref{sec:general} we explain how our methods and
results can be extended to more general configurations of poles.
\begin{figure}[ht]
	\centering
	\begin{tikzpicture}[scale=0.5,line cap=round,line join=round,>=triangle 45,x=1.0cm,y=1.0cm]
  \clip(-6,-5.5) rectangle (6,4);
  \filldraw[fill=gray, opacity=0.2]  plot [smooth cycle] coordinates {(5.25,1.5) (5.25,3)
  (3,3.6) (0,3) 
  (-4.2,3.9) (-5.25,3)  (-4.5,1.5) (-3,0)
  (-1.5,-5.1) (0,-4) (1.5,-3.5)};
  \draw [line width=0.5pt,dashed] (-1.,1.)-- (2.45,-2.4);
\draw [line width=0.5pt, dashed] (-1.,-0.6)-- (1.42,0.82);
\draw [line width=0.5pt, dashed] (0.48,-2.)-- (-0.65,3.1);
\draw [line width=0.5pt, dashed] (1.,2.)-- (-1.94,-4.05);
\draw [fill=black] (0.03,-0.01) circle (2pt);
\begin{scriptsize}
\draw [fill=gray] (-1.,1.) circle (2.5pt);
\draw[color=black] (-1,1.45) node {$a^1$};
\draw [fill=gray] (-1.,-0.6) circle (2.5pt);
\draw[color=black] (-1.4,-1) node {$a^{j}$};
\draw [fill=gray] (1.42,0.82) circle (2.5pt);
\draw[color=black] (1.54,1.25) node {$a^{j+k_2}$};
\draw [fill=gray] (0.48,-2.) circle (2.5pt);
\draw[color=black] (0.60,-2.4) node {$a^{k_1}$};
\draw [fill=gray] (1.,2.) circle (2.5pt);
\draw[color=black] (1.12,2.45) node {$a^2$};
\end{scriptsize}
\end{tikzpicture}
\caption{Configuration of poles $(k_1+1\leq j \leq k_1+k_2)$.}
\label{fig:f1}
\end{figure}
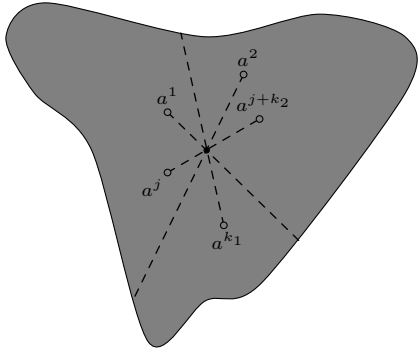

For every  $j=1,\dots, k$ and $\e \in (0,1]$, we define 
\begin{equation*}
a_\e^j:=\e a^j.
\end{equation*}
For every $b=(b_1,b_2) \in \R^2$, the Aharonov-Bohm vector potential
with pole $b$ and circulation $\rho\in\R$ is defined as 
\begin{equation*}
  A_b^\rho(x_1,x_2):=\rho\left(\frac{-(x_2-b_2)}{(x_1-b_1)^2+(x_2-b_2)^2},
    \frac{x_1-b_1}{(x_1-b_1)^2+(x_2-b_2)^2}\right), \quad (x_1,x_2)\in
  \R^2\setminus	\{b\}.
\end{equation*}
In this paper, we consider the case of half-integer circulations
$\rho\in \frac12+\Z$, which
is of particular interest from the mathematical point of view due to
applications to the problem of 
spectral minimal partitions, see \cite{BNHHO, NT}. For $\rho=\frac12$ we denote
\begin{equation}\label{def_A_a^j}
  A_b:=A_b^{1/2}.
\end{equation}
We are interested in the multi-singular vector potential 
\begin{equation*}
  \quad \mathcal
  A_\e^{(n_1,n_2,\dots,n_k)}:=\sum_{j=1}^{k}A_{a^j_\e}^{n_j+\frac12}=
  \sum_{j=1}^{k}(2n_j+1)A_{a^j_\e},
\end{equation*}
having at each pole $a^j_\e$ half-integer circulation $n_j+\frac12$
with $n_j\in\Z$, 
and in the corresponding eigenvalue problem 
\begin{equation}\label{prob_Aharonov-Bohm_multipole-nj}
\begin{cases}
\left(i\nabla +\mathcal A_\e^{(n_1,n_2,\dots,n_k)}\right)^2 u= \la \,u, &\text{ in } \Omega,\\
u=0, &\text{ on } \partial \Omega,
\end{cases}
\end{equation} 
where the magnetic Schr\"odinger operator
$\big(i\nabla + \mathcal A_\e^{(n_1,n_2,\dots,n_k)}\big)^2$ acts as
\begin{equation*}
  \left(i\nabla +\mathcal A_\e^{(n_1,n_2,\dots,n_k)}\right)^2u:=-\Delta u +2 i \,\mathcal
  A_\e^{(n_1,n_2,\dots,n_k)}\cdot \nabla u
  +\big|\mathcal A_\e^{(n_1,n_2,\dots,n_k)}\big|^2 u.
\end{equation*}
Since $n_j\in\Z$, 
$\mathcal A_\e^{(n_1,n_2,\dots,n_k)}$ is gauge equivalent to the vector potential
\begin{equation*}
\quad \mathcal A_\e:=\sum_{j=1}^{k}(-1)^{j+1}A_{a^j_\e}.
\end{equation*}
Therefore the operators $\big(i\nabla + \mathcal
A_\e^{(n_1,n_2,\dots,n_k)}\big)^2$ and $(i\nabla +\mathcal A_\e)^2$
are unitarily equivalent (see \cite[Theorem 1.2]{leinfelder} and
\cite[Proposition 2.2]{L2015}), and consequently the 
spectrum of \eqref{prob_Aharonov-Bohm_multipole-nj} coincides with that of
\begin{equation}\label{prob_Aharonov-Bohm_multipole}
\begin{cases}
(i\nabla +\mathcal A_\e)^2 u= \la \,u, &\text{ in } \Omega,\\
u=0, &\text{ on } \partial \Omega.
\end{cases}
\end{equation}
Hence, to study the behaviour as $\e\to0^+$ of the spectrum of
\eqref{prob_Aharonov-Bohm_multipole-nj}, it is not restrictive to
consider problem \eqref{prob_Aharonov-Bohm_multipole}.  We refer to
\eqref{eq_eigenfunctions_Aharonov_Bohm_multipole} for the variational
formulation of \eqref{prob_Aharonov-Bohm_multipole}. From classical
spectral theory, problem \eqref{prob_Aharonov-Bohm_multipole} has a
diverging sequence of real positive eigenvalues
$\{\la_{\e,n}\}_{n\in \N\setminus\{0\}}$; in the sequence
$\{\la_{\e,n}\}_{n\in \N\setminus\{0\}}$ we repeat each eigenvalue
according to its multiplicity. Moreover, the eigenspace associated to
each eigenvalue has finite dimension.

As $\e\to0^+$, the following limit eigenvalue problem
comes into play: 
\begin{equation}\label{prob_Aharonov-Bohm_0}
\begin{cases}
\left(i\nabla +\frac{1+(-1)^{k+1}}{2} A_{0}\right)^2u=\la u, &\text{ in } \Omega,\\
u=0, &\text{ on } \partial \Omega,
\end{cases}
\end{equation} 
with $A_0$ defined as in \eqref{def_A_a^j} with $b=0$. If $k$ is odd, the operator
in \eqref{prob_Aharonov-Bohm_0} is the Aharonov-Bohm operator with
one pole in $0$ and circulation $\frac12$; 
as above, the classical Spectral Theorem applies and provides a
diverging sequence of real positive eigenvalues
$\{\la_{0,n}\}_{n\in \N\setminus\{0\}}$ with finite multiplicity.
Furthermore, it is well-known that, in this case, eigenfunctions
  vanish in $0$ with order $\frac m2$, for some odd $m\in\N$, and
  have exactly $m$ nodal lines meeting at $0$ and dividing the whole
  $2\pi$-angle into $m$ equal parts; see \cite[Theorem 1.3,
Section 7]{FFT}  and \eqref{eq:asyu}--\eqref{eq:asygrad} for a 
description of the asymptotic behaviour at $0$ of eigenfunctions of \eqref{prob_Aharonov-Bohm_0}.

If $k$ is even the nature of the limit eigenvalue
  problem undergoes a significant mutation. Indeed, for $k$ even, the
  operator in \eqref{prob_Aharonov-Bohm_0} is the classical Dirichlet
  Laplacian and the eigenvalue problem \eqref{prob_Aharonov-Bohm_0} can
  be rewritten as 
\begin{equation}\label{prob_eigenvalue_keven_regular}
\begin{cases}
-\Delta u=\la u, &\text{ in } \Omega,\\
u=0, &\text{ on } \partial \Omega.
\end{cases}
\end{equation}
We conclude that, for every $k \in \mathbb{N} \setminus \{0\}$, the
spectrum of \eqref{prob_Aharonov-Bohm_0} is a diverging sequence
$\{\la_{0,n}\}_{n \in \mathbb{N}\setminus\{0\}}$ of positive real
eigenvalues.

We recall   from \cite[Theorem 1.2]{L2015} that, whatever the number $k$ of
poles is, 
\begin{equation*}
\text{the function}\quad\e \mapsto \la_{\e,n} \quad \text{is continuous on } [0,1], 
\end{equation*}
so that, in particular, 
\begin{equation}\label{limit_lae_la0}
\lim_{\e\to0^+}\la_{\e,n} = \la_{0,n}
\end{equation}
for every $n \in \mathbb{N}\setminus\{0\}$. The present paper
  aims at giving a sharp
  asymptotic expansion for the variation $\la_{\e,n} -
  \la_{0,n}$ of simple eigenvalues  with respect to the moving
  configuration of poles. 

  In the case of one moving pole, \cite{BNNNT} establishes a first
  relation between the rate of convergence \eqref{limit_lae_la0} and
  the number of the nodal lines of the corresponding
  eigenfunction. Sharper asymptotic expansions for simple eigenvalues
  are obtained in \cite{AFaharonov}, in the case of one pole moving
  along the tangent to a nodal line of the limit eigenfunction, and in
  \cite{AF-SIAM}, in the case of one pole moving along any direction.
  The case of one pole approaching the boundary is treated in
  \cite{AFNN-2017} and \cite{NNT-2015}. The methods developed in
  \cite{AFaharonov}, \cite{AFNN-2017}, and \cite{NNT-2015} are based
  on an Almgren type frequency formula, which provides local energy
  bounds for eigenfunctions. These are used to estimate the Rayleigh
  quotient, whose minimax levels characterize the eigenvalues, and to
  prove the convergence of a family of blown-up eigenfunctions to some
  non trivial limit profile. In particular, using the notation
  introduced above, in \cite{AFaharonov} it is proved that, for
  $k=k_1=1$ and
  $a^1_\e=\e a^1 =\e\, r_1(\cos(\alpha^1),\sin(\alpha^1))$ moving
  along the tangent to one of the $m$ nodal lines of the limit
  eigenfunction $u_0$, if $\lambda_{0,n}$ is a simple, then
\begin{equation}\label{eq:AF}
\lambda_{\e,n}-\lambda_{0,n}=
4 \,r_1^m(|\beta_1|^2+|\beta_2|^2)\,{\mathfrak M}\,\e^m +o(\e^m )\quad\text{as }\e\to0^+.
\end{equation}
In \eqref{eq:AF}  $(\beta_1,\beta_2)\neq(0,0)$ is such that 
\begin{equation*}
    \lim_{r\to0^+}r^{-\frac m2} u_0(r\cos t,r\sin t) =\beta_1 
e^{i\frac t2}\cos\big(\tfrac m2 t\big)+\beta_2  e^{i\frac t2}
\sin\big(\tfrac m2 t\big),
\end{equation*} see \eqref{eq:asyu},  and 
  ${\mathfrak M}<0$ is a negative constant depending only on $m$, which has the 
  following variational characterization:
  \begin{equation}\label{eq:m}
     {\mathfrak M}=\min\left\{\frac12 \int_{\R^2_+} |\nabla u(x)|^2 \,dx-\frac m2
 \int_0^1 x_1^{\frac m2 -1}u(x_1,0)\,dx_1:u\in {\mathcal D}^{1,2}_s(\R^2_+)\right\}, 
  \end{equation}
  where $s:=\{(x_1,x_2)\in\R^2: x_2=0\text{ and }x_1\geq 1\}$,
  $\R^2_+=\{(x_1,x_2)\in\R^2:x_2>0)\}$, and
  ${\mathcal D}^{1,2}_s(\R^2_+)$ is the completion of
  $C^\infty_{\rm c}(\overline{\R^2_+} \setminus s)$ with respect to
  the norm $( \int_{\R^2_+} |\nabla u|^2\,dx )^{1/2}$.  For an
  explicit formula for $\mathfrak M$ we refer to \cite[Theorem
  2.3]{AFLeigenvar}.  The quantity appearing in \eqref{eq:m} can be
  interpreted as a weighted \emph{torsional rigidity} of the segment
  along which the pole is moving. Concerning the classical notion of
  torsional rigidity of a set, the literature is vast; among many
  others, we cite the classical books \cite{PZ1951,henrot2018} for the
  basic definitions and some possible application in shape
  optimization and \cite{BBV,BM2020,brasco2022steklov} for more recent
  investigations in the field. We also point out \cite{AO-2023}, where
  a notion of \emph{thin torsional rigidity} is exploited in the study
  of spectral stability for some singularly perturbed problems.

In the case of one single pole, the study of Aharonov-Bohm eigenvalues benefits 
from the known regularity of the eigenvalue as a function of the pole 
position. Indeed, in \cite{L2015} it is proved that, in the case of one 
moving pole, eigenvalues are analytic as functions of the pole, so 
that the eigenvalue variation admits a Taylor expansion. The sharp 
asymptotics  on nodal lines \eqref{eq:AF} obtained in \cite{AFaharonov} 
is used in \cite{AF-SIAM} to compute the leading term of such Taylor 
expansion, exploiting symmetry and periodicity properties of the 
Fourier coefficients of the blow-up profile with respect to the
moving direction. In the case of many poles, the analyticity property  
is maintained as long as the poles are away from each other (see 
again \cite{L2015}), but is lost in the case of a collision; indeed 
in \cite{AFL} (and in \cite{AFHL} for symmetric domains) it is proved that,
in the case of two poles colliding 
at a point outside the nodal set of the limit eigenfunction, the 
eigenvalue variation is asymptotic to the logarithm of the distance. 

From the above discussion it therefore emerges that the case of 
multiple colliding poles presents additional significant
difficulties. So far, up to our knowledge, in the literature only the case of
two coalescing poles has been addressed  with the aim of deriving precise
asymptotic estimates in terms of the distance between the two poles. 
The paper \cite{AFHL} derives the asymptotic behaviour of eigenvalues of 
Aharonov–Bohm operators with two colliding poles moving on an axis of symmetry 
of the domain, which is assumed not to be tangent to any nodal line of the 
limit eigenfunction. 
The argument used in \cite{AFHL} is based on isospectrality  with the Dirichlet
Laplacian on the domain with a small segment removed, for which an asymptotic 
expansion of the eigenvalue variation is obtained by a capacity argument, in the
spirit of \cite{courtois}. The complementary case of two colliding poles, which move
on an axis of symmetry coinciding with a nodal line of the
limit eigenfunction, is treated in \cite{AFLeigenvar}, exploiting an 
isospectrality result and  a monotonicity formula in the spirit of 
\cite{AFaharonov}. The assumption of symmetry of the domain is removed in 
\cite{AFL}, in the case of two poles collapsing at an interior point out of nodal
lines of the limit eigenfunction; this is possible thanks to an estimate of the diameter of 
the nodal set of magnetic eigenfunctions close to the collision point. 

In the present paper we develop a new approach that provides  
asymptotic expansions of the eigenvalue variation in the most general 
case of any number of poles moving towards a collision point. We propose a 
method which combines the idea of torsional rigidity, naturally appearing in 
\cite{AFaharonov} (see also \cite[Theorem 2.2]{AFNN-2017}) to variationally 
characterize the coefficient of the leading term as in \eqref{eq:m}, 
with that of capacity, 
which \cite{courtois} and \cite{AFHL} show to be the good small  parameter
in a spectral perturbation theory in domains with small holes.
  
  Let us assume that there exists
  $n_0\in\N\setminus\{0\}$ such that
\begin{equation}\label{eq:simple}
 \lambda_{0,n_0}\text{ is a simple eigenvalue of \eqref{prob_Aharonov-Bohm_0}}.
\end{equation}
In view of \eqref{limit_lae_la0}, assumption \eqref{eq:simple} implies
that also $\lambda_{\e,n_0}$ is simple as an eigenvalue of
\eqref{prob_Aharonov-Bohm_multipole}, provided $\e$ is sufficiently
small. Simplicity of the spectrum is a \emph{generic} property for
many differential operators. We refer e.g. to \cite{Teytel}, where the
author exhibits sufficient conditions for genericity of simplicity of
the spectrum for various families of differential operators (including
Aharonov-Bohm operators with a single pole). See also
\cite{Abatangelo-multip} for a focus on the particular case of
Aharonov-Bohm operators.

The first step in our approach is to perform some gauge transformation, making  the
magnetic eigenvalue problem \eqref{prob_Aharonov-Bohm_multipole}, and its 
corresponding limit one \eqref{prob_Aharonov-Bohm_0}, equivalent to eigenvalue 
problems for the Laplacian in domains with straight cracks, laying along the moving 
directions of poles, see \eqref{prob_eigenvalue_gauged_multipole} and
\eqref{prob_eigenvlaue_guged0}. Fixing a $L^2$-normalized eigenfunction $v_0$ of 
the equivalent limit eigenvalue problem \eqref{prob_eigenvlaue_guged0} 
associated to the eigenvalue 
$\lambda_{0,n_0}$, we prove in Theorem \ref{t:main1} the following asymptotic 
expansion:
\begin{equation}\label{eq:exp-la}
  \la_{\e, n_0} -\la_{0, n_0}=2\big(\mc{E}_\e-L_\e(v_0)\big)
 +o\big(\|\nabla V_\e\|^2_{L^2(\Omega)}\big) \quad \text{as  }\e \to 0^+,
\end{equation}
where $L_\e$ is the  linear functional defined in \eqref{def_Le},  $\mathcal E_\e$ 
is the minimum of an energy functional associated with the configuration
of poles and defined on a space of functions suitably jumping through the cracks, 
see \eqref{eq:defEe}, and $V_\e$ is the potential attaining such a minimum.
We observe that $\mathcal E_\e$ is a kind of intermediate quantity
between torsional rigidity and capacity of the set obtained as the union of the 
segments connecting the poles to the origin. 
Indeed, the capacity of a set is defined by
minimizing the $L^2$-norm of the gradient among functions which are prescribed 
on the set; the torsional rigidity, instead, is constructed by minimizing an energy
functional, which contains a linear term involving an integral on the set, without
prescribing any condition. In the definition of  $\mathcal E_\e$ given in 
\eqref{eq:defEe}, we minimize an energy functional over a family of functions which 
are only partially prescribed on the cracks, in the sense that we impose a jump 
condition on the functions across the segments, obtaining a jump of the normal 
derivatives as a consequent natural condition. 
The development of such an intermediate notion provides a unified approach, which does not 
require an a priori relation between the configuration of poles and the orientation 
of the nodal set of the limit eigenfunction. 
We mention that elliptic problems in
cracked domains, with jumps of the unknown function and its normal derivative 
prescribed on the cracks, are studied in \cite{Medkova-Krutitskii}.  

For $k$ odd, a blow-up analysis allows us to identify the exact
asymptotic behaviour of the quantities appearing in the right hand
side of \eqref{eq:exp-la}. In Theorem \ref{t:exp-odd} we prove that
$\lim_{\e\to 0^+}\e^{-m}\mathcal E_\e =\mathcal E$, where $m$ is the
vanishing order of $v_0$ at $0$ and $\mathcal E$ is the minimum of the
energy functional defined in \eqref{def_J} over a space of suitably
jumping functions, see \eqref{def_E_limit}. Thus we generalize
\eqref{eq:AF} in the multipolar case, obtaining the following explicit
expansion
\begin{equation}\label{eq:exp-la2}
    \la_{\e, n_0} -\la_{0, n_0}=2\,\e^m\big(\mathcal E-L(\Psi_0)\big)+o(\e^m)
\end{equation}
as $\e \to 0^+$, where $L$ is the linear functional defined in
\eqref{def_L} and $\Psi_0$ is the $\frac m2$-homogeneous harmonic
function introduced in \eqref{def_Psi}. We note that the assumption
that $k$ is odd is crucial in the blow-up analysis, since it
guarantees the validity of the Hardy-type inequality proved in
Proposition \ref{prop_hardy}, needed to characterize the functional
space containing the limiting blow-up profile.  In the particular case
of all poles moving either along the tangents to nodal lines or along
the bisectors between nodal lines of the limit eigenfunction, we can
prove that the quantity $\mathcal E-L(\Psi_0)$, appearing as the
coefficient of the leading term of the asymptotic expansion
\eqref{eq:exp-la2}, does not vanish, see Proposition
\ref{p:two-cases}; this shows that $m$ is exactly the vanishing order
of the eigenvalue variation. On the other hand, the study of the
continuity properties of the coefficients appearing in
\eqref{eq:exp-la2}, see Theorem \ref{theorem_continuity}, allows us to
prove the existence of configurations of poles for which
$\mathcal E-L(\Psi_0)=0$ and hence $\la_{\e, n_0} -\la_{0, n_0}$ is an
infinitesimal of higher order than $m$.

If $k$ is even, a Hardy type inequality is no more available, and therefore 
the blow-up analysis meets the technical difficulty of identifying the limiting
profile in an appropriate functional space. In spite of that, in the case 
of two poles colliding in a point of the nodal set of the limit eigenfunction 
and  moving either along the tangents to its nodal lines or along its bisectors, 
in Theorems \ref{theorem_2poles_nodal} and \ref{theorem_2poles_bisector} we are
able to derive the exact asymptotic behavior of $\mc{E}_\e-L_\e(v_0)$, and
consequently   of $\la_{\e, n_0} -\la_{0, n_0}$ thanks to the use of elliptic
coordinates; in this way we generalize the results of \cite{AFHL} and 
\cite{AFLeigenvar}, which require an axial symmetry of the domain as a further 
hypothesis. 

In the next section we state the main results of the paper, after
having introduced the necessary notations.

\section{Statement of the main results}\label{sec_main_results}

To give a variational formulation of problem \eqref{prob_Aharonov-Bohm_multipole}, we
introduce the functional space
$H^{1,\e}(\Omega,\C)$, defined as the completion of
\begin{equation*}
  \{\phi \in H^1(\Omega,\C) \cap C^\infty(\Omega,\C): \phi\equiv 0
  \text{ in  a neighbourhood of } a^j_\e \text{ for all } j=1,\dots,k\}
\end{equation*}
with respect to the norm 
\begin{equation}\label{def_norm_H1a}
  \norm{w}_{H^{1,\e}(\Omega,\C)}:=\bigg(\norm{w}_{L^2(\Omega,\C)}^2+
    \norm{\nabla w}_{L^2(\Omega,\C^2)}^2+\sum_{j=1}^{k}
    \Big\|\tfrac{w}{|\cdot-a^j_\e|}\Big\|_{L^2(\Omega,\C)}^2\bigg)^{\!\!1/2}.
\end{equation}
We observe that $H^{1,\e}(\Omega,\C)=\left\{u\in
  H^1(\Omega,\C):\frac{u}{|\cdot-a^j_\e|}\in L^2(\Omega,\C)\text{ for
  all }j=1,\dots,k\right\}$.

In \cite{LW} (see also \cite{AFJT} and \cite[Lemma 3.1,
Remark 3.2]{FFT}), the following local magnetic Hardy-type inequality
\begin{equation*}
  \int_{D_r(b)}|i \nabla w + A_b^\rho w|^2 \, dx \ge
  \Big(\min_{j \in \mb{Z}}|j-\rho|\Big)^2\int_{D_r(b)}\frac{|w(x)|^2}{|x-b|^2}\, dx
\end{equation*}
is proved for every $ b \in \R^2$ and
$w \in C^\infty_{\rm c}(\overline{D_r(b)}\setminus\{b\},\C)$.  It follows that the norm
\eqref{def_norm_H1a} is equivalent to the norm
\begin{equation*}
  \left(\norm{\left(i\nabla +\mathcal
        A_\e\right)u}_{L^2(\Omega,\C^2)}^2+
    \norm{u}_{L^2(\Omega,\C)}^2\right)^{\!1/2}.
\end{equation*}
To deal with homogeneous Dirichlet boundary conditions, we also
introduce the space $H_0^{1,\e}(\Omega,\C)$ defined as the
closure of $C^\infty_{\rm c}(\Omega \setminus \{a_\e^1,\dots,a_\e^k\})$ in
$H^{1,\e}(\Omega,\C)$. The space
$H_0^{1,\e}(\Omega,\C)$ can be explicitly characterized as
\begin{equation*}
  H_0^{1,\e}(\Omega,\C)=\left\{w \in  H_0^1(\Omega,\C):
    \tfrac{w}{|\cdot-a^j_\e|}\in L^2(\Omega,\C) \text{ for all } j=1,\dots,k\right\}.
\end{equation*}
We say that $\lambda \in\R$ is an \emph{eigenvalue} of 
\eqref{prob_Aharonov-Bohm_multipole} if there exists
$u \in H_0^{1,\e}(\Omega,\C)\setminus\{0\}$ (called \emph{eigenfunction})
such that
\begin{equation}\label{eq_eigenfunctions_Aharonov_Bohm_multipole}
\int_\Omega \left(i\nabla +\mathcal A_\e\right)u\cdot \overline{\left(i\nabla
    +\mathcal A_\e\right)w} \, dx =\la \int_\Omega u \overline{w} \, dx
\quad \text{for all } w \in H_0^{1,\e}(\Omega,\C).
\end{equation}
We recall from the introduction that the eigenvalue problem 
\eqref{prob_Aharonov-Bohm_multipole} (and hence
\eqref{eq_eigenfunctions_Aharonov_Bohm_multipole}) admits a diverging sequence of real
positive eigenvalues 
\begin{equation*}
    \la_{\e,1}\leq \la_{\e,2}\leq \la_{\e,3}\leq\cdots,
\end{equation*} repeated 
in the enumeration  according to their multiplicity.

In a similar way, the variational formulation of \eqref{prob_Aharonov-Bohm_0} in the
case $k$ odd (corresponding to a problem of type \eqref{prob_Aharonov-Bohm_multipole}
with only one pole located at $0$) can be be given in the functional space
$\{w \in H_0^1(\Omega,\C): \tfrac{w}{|x|}\in L^2(\Omega,\C)\}$.  In the case $k$ even,
instead, \eqref{prob_Aharonov-Bohm_0} takes the form of the classical eigenvalue 
problem for the Dirichlet Laplacian, whose variational formulation is well known.
In both cases, \eqref{prob_Aharonov-Bohm_0} admits a diverging sequence of real
positive eigenvalues 
\begin{equation*}
    \la_{0,1}\leq \la_{0,2}\leq \la_{0,3}\leq\cdots,
\end{equation*} 
repeated according to their multiplicity.

A suitable gauge
transformation allows us to obtain equivalent formulations of
\eqref{prob_Aharonov-Bohm_multipole} and \eqref{prob_Aharonov-Bohm_0} as 
eigenvalue problems for the Laplacian in domains with straight cracks.
For every $\e \in [0,1]$ we define 
\begin{align*}
  &\Sigma^j:=\{ta^j: t \in \R\} \quad \text{for all  } j=1,\dots,k_1+k_2, \\
  &\Gamma^j_\e:=\{ta^j: t \in (-\infty,\e]\}, \quad S_\e^j:=\{ta^j: t
    \in [0,\e]\}
    \quad  \text{for all  } j=1,\dots,k_1, \\
  &S_\e^j:=\{t a^{j}+(\e-t)a^{j+k_2}: t \in [0,\e]\}
    \quad \text{for all  }
    j=k_1+1,\dots,k_1+k_2,  \\
  &\Gamma_\e:=\bigg(\bigcup_{j=1}^{k_1} \Gamma_\e^j \bigg)\cup\bigg(
    \bigcup_{j=k_1+1}^{k_1+k_2} S_\e^j\bigg),
\end{align*}
see \Cref{3figs}.
\begin{figure}%
\centering
\subfloat[\scriptsize The set $\Gamma_\e$.]{\begin{tikzpicture}[scale=0.6,line cap=round,line join=round,>=triangle 45,x=1.0cm,y=1.0cm]
  \clip(-4,-4) rectangle (3.5,3);
  \draw [line width=1pt,rotate=110,color=red] (-3,0)-- (2,0);
  \draw [line width=1pt,rotate=110,dotted,color=red] (-4,0)-- (-3,0);
\draw [fill=gray,rotate=110] (2,0) circle (2.5pt);
\draw[color=black,rotate=110] (2.5,0) node {\tiny $a^{1}_\e$};
 \draw [line width=1pt,rotate=70,color=red] (-3,0)-- (1,0);
  \draw [line width=1pt,rotate=70,dotted,color=red] (-4,0)-- (-3,0);
\draw [fill=gray,rotate=70] (1,0) circle (2.5pt);
\draw[color=black,rotate=70] (1.5,0) node {\tiny $a^{2}_\e$};
\draw [line width=1pt,rotate=10,color=red] (-3,0)-- (2.5,0);
  \draw [line width=1pt,rotate=10,dotted,color=red] (-4,0)-- (-3,0);
\draw [fill=gray,rotate=10] (2.5,0) circle (2.5pt);
\draw[color=black,rotate=10] (3,0) node {\tiny $a^{3}_\e$};
\draw [line width=1pt,rotate=-30,color=red] (0,0)-- (2,0);
  \draw [fill=gray,rotate=-30] (2,0) circle (2.5pt);
\draw[color=black,rotate=-30] (2.5,-0.2) node {\tiny$a^{k_1+1}_\e$};
\draw [line width=1pt,rotate=150,color=red] (0,0)-- (1,0);
  \draw [fill=gray,rotate=150] (1,0) circle (2.5pt);
\draw[color=black,rotate=150] (1.5,0) node {\tiny$a^{k_1+1+k_2}_\e$};
\draw [line width=1pt,rotate=45,color=red] (0,0)-- (3,0);
  \draw [fill=gray,rotate=45] (3,0) circle (2.5pt);
\draw[color=black,rotate=45] (3.5,0) node {\tiny$a^{k_1+2k_2}_\e$};
\draw [line width=1pt,rotate=225,color=red] (0,0)-- (1.5,0);
  \draw [fill=gray,rotate=225] (1.5,0) circle (2.5pt);
\draw[color=black,rotate=225] (2.2,0) node {\tiny$a^{k_1+k_2}_\e$};
\end{tikzpicture}}\quad
\subfloat[\scriptsize The set $\Gamma_0$.]{\begin{tikzpicture}[scale=0.6,line cap=round,line join=round,>=triangle 45,x=1.0cm,y=1.0cm]
  \clip(-4,-4) rectangle (3.5,3);
   \draw [line width=0.2pt,rotate=110,dashed] (0,0)-- (2,0);
  \draw [line width=1pt,rotate=110,color=red] (-3,0)-- (0,0);
  \draw [line width=1pt,rotate=110,dotted,color=red] (-4,0)-- (-3,0);
\draw [fill=gray,rotate=110] (2,0) circle (2.5pt);
\draw[color=black,rotate=110] (2.5,0) node {\tiny $a^{1}_\e$};
 \draw [line width=1pt,rotate=70,color=red] (-3,0)-- (0,0);
 \draw [line width=0.2pt,rotate=70,dashed] (0,0)-- (1,0);
  \draw [line width=1pt,rotate=70,dotted,color=red] (-4,0)-- (-3,0);
\draw [fill=gray,rotate=70] (1,0) circle (2.5pt);
\draw[color=black,rotate=70] (1.5,0) node {\tiny $a^{2}_\e$};
\draw [line width=1pt,rotate=10,color=red] (-3,0)-- (0,0);
\draw [line width=0.2pt,rotate=10,dashed] (0,0)-- (2.5,0);
  \draw [line width=1pt,rotate=10,dotted,color=red] (-4,0)-- (-3,0);
\draw [fill=gray,rotate=10] (2.5,0) circle (2.5pt);
\draw[color=black,rotate=10] (3,0) node {\tiny $a^{3}_\e$};
\draw [line width=0.2pt,rotate=-30,dashed] (0,0)-- (2,0);
  \draw [fill=gray,rotate=-30] (2,0) circle (2.5pt);
\draw[color=black,rotate=-30] (2.5,-0.2) node {\tiny$a^{k_1+1}_\e$};
\draw [line width=0.2pt,rotate=150,dashed] (0,0)-- (1,0);
  \draw [fill=gray,rotate=150] (1,0) circle (2.5pt);
\draw[color=black,rotate=150] (1.5,0) node {\tiny$a^{k_1+1+k_2}_\e$};
\draw [line width=0.2pt,rotate=45,dashed] (0,0)-- (3,0);
  \draw [fill=gray,rotate=45] (3,0) circle (2.5pt);
\draw[color=black,rotate=45] (3.5,0) node {\tiny$a^{k_1+2k_2}_\e$};
\draw [line width=0.2pt,rotate=225,dashed] (0,0)-- (1.5,0);
  \draw [fill=gray,rotate=225] (1.5,0) circle (2.5pt);
\draw[color=black,rotate=225] (2.2,0) node {\tiny$a^{k_1+k_2}_\e$};
\end{tikzpicture}}\quad
\subfloat[\scriptsize The sets $S^j_\e$.]{\begin{tikzpicture}[scale=0.6,line cap=round,line join=round,>=triangle 45,x=1.0cm,y=1.0cm]
  \clip(-4,-4) rectangle (3.5,3);
   \draw [line width=0.2pt,rotate=110,dashed] (-3,0)-- (0,0);
\draw [line width=1pt,rotate=110,color=red] (0,0)-- (2,0);
\draw [line width=0.2pt,rotate=110,dotted] (-4,0)-- (-3,0);
\draw [fill=gray,rotate=110] (2,0) circle (2.5pt);
\draw[color=black,rotate=110] (2.5,0) node {\tiny $a^{1}_\e$};
 \draw [line width=0.2pt,rotate=70,dashed] (-3,0)-- (0,0);
 \draw [line width=1pt,rotate=70,color=red] (0,0)-- (1,0);
  \draw [line width=0.2pt,rotate=70,dotted] (-4,0)-- (-3,0);
\draw [fill=gray,rotate=70] (1,0) circle (2.5pt);
\draw[color=black,rotate=70] (1.5,0) node {\tiny $a^{2}_\e$};
\draw [line width=0.2pt,rotate=10,dashed] (-3,0)-- (0,0);
\draw [line width=1pt,rotate=10,color=red] (0,0)-- (2.5,0);
  \draw [line width=0.2pt,rotate=10,dotted] (-4,0)-- (-3,0);
\draw [fill=gray,rotate=10] (2.5,0) circle (2.5pt);
\draw[color=black,rotate=10] (3,0) node {\tiny $a^{3}_\e$};
\draw [line width=1pt,rotate=-30,color=red] (0,0)-- (2,0);
  \draw [fill=gray,rotate=-30] (2,0) circle (2.5pt);
\draw[color=black,rotate=-30] (2.5,-0.2) node {\tiny$a^{k_1+1}_\e$};
\draw [line width=1pt,rotate=150,color=red] (0,0)-- (1,0);
  \draw [fill=gray,rotate=150] (1,0) circle (2.5pt);
\draw[color=black,rotate=150] (1.5,0) node {\tiny$a^{k_1+1+k_2}_\e$};
\draw [line width=1pt,rotate=45,color=red] (0,0)-- (3,0);
  \draw [fill=gray,rotate=45] (3,0) circle (2.5pt);
\draw[color=black,rotate=45] (3.5,0) node {\tiny$a^{k_1+2k_2}_\e$};
\draw [line width=1pt,rotate=225,color=red] (0,0)-- (1.5,0);
  \draw [fill=gray,rotate=225] (1.5,0) circle (2.5pt);
\draw[color=black,rotate=225] (2.2,0) node {\tiny$a^{k_1+k_2}_\e$};
\end{tikzpicture}}%
\caption{The sets $\Gamma_\e$, $\Gamma_0$, $S^j_\e$ ($1\leq j\leq k_1+k_2$).}
\label{3figs}
\end{figure}
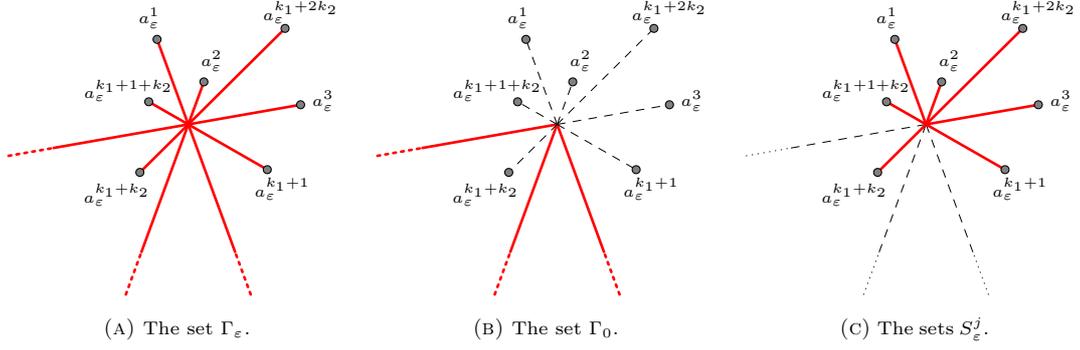
We note that, for every $j=1,\ldots,k_1$,  $\Gamma^j_0=\Gamma^j_\e\setminus S_\e^j$
is the straight half-line starting at $0$ with slope $\alpha_j+\pi$. For
every $\e \in [0,1]$, we consider the functional space
$\mathcal H_\e$ defined as the closure of
\begin{equation*}
\left\{w \in H^1(\Omega \setminus \Gamma_\e)=H^1(\Omega \setminus \Gamma_\e,\R): w=0 \text{ on a
    neighbourhood of } \partial \Omega\right\}
\end{equation*}
in $H^1(\Omega \setminus
  \Gamma_\e)$ endowed with the norm $\norm{w}_{H^1(\Omega \setminus
  \Gamma_\e)}=\|\nabla w\|_{L^2(\Omega
      \setminus \Gamma_\e)}+\|w\|_{L^2(\Omega)}$.
From the Poincar\'e type inequality stated in Proposition
\ref{prop_poincare}, it follows that 
\begin{equation*}
\norm{w}_{\mc{H}_\e}:= \left(\int_{\Omega \setminus \Gamma_\e} |\nabla
  w|^2 \, dx\right)^{\!\!1/2}
\end{equation*}
is a norm on $\mc{H}_{\e}$ equivalent to $\norm{w}_{H^1(\Omega \setminus  \Gamma_\e)}$.
The corresponding scalar product is denoted as $(\cdot,\cdot)_{\mathcal H_\e}$.

For every $j=1,\dots, k_1+k_2$, with the notation
$\nu^j:=\big(-\sin (\alpha^j), \cos (\alpha^j)\big)$ we consider the
half-planes
\begin{equation*}
  \pi_+^j:=\{x \in \R^2: x\cdot \nu^j>0\}\quad\text{and}\quad
  \pi_-^j:=\{x \in \R^2: x\cdot \nu^j<0\}.
\end{equation*}
We observe that $\nu^j$ is the unit outer normal vector to $\pi_-^j$ on $\partial \pi_-^j$.
In view of classical trace results and embedding theorems for
fractional Sobolev spaces in dimension $1$, for every $j=1,\dots,k_1+k_2$ and
$p \in [2,+\infty)$  there exist continuous
trace operators
\begin{equation}\label{def_traces}
  \gamma_+^j:H^1(\pi_+^j\setminus  \Gamma_1 ) \to L^p(\Sigma^j)  \quad
  \text{and} \quad
  \gamma_-^j:H^1(\pi_-^j\setminus  \Gamma_1 ) \to L^p(\Sigma^j).
\end{equation} 
We also define the  trace operators
\begin{equation}\label{eq:trTj}
  T^j:H^1(\R^2\setminus \Gamma_1) \to L^p(\Sigma^j), \quad
  T^j(w):=\gamma_+^j(w\vert_{\pi_+^j})
  +\gamma_-^j(w\vert_{\pi_-^j}),
\end{equation}
for every $j=1,\dots,k_1+k_2$ and $p\in[2,+\infty)$.  For every
$\e\in[0,1]$, the restrictions to $\mathcal H_\e$ of the operators
$\gamma_+^j,\gamma_-^j$ and $T^j$ are linear and continuous, since any
element of $\mc{H}_\e$ can be trivially extended by $0$ to an element
of $H^1(\R^2 \setminus \Gamma_1)$; furthermore, due to the
  boundedness of $\Omega$, such restrictions are continuous and
  compact from $\mathcal H_\e$ into $L^p(\Sigma^j\cap\Omega)$ for all
  $p\in[1,+\infty)$.

For every $\e \in (0,1]$, we define the space
\begin{equation}\label{def_tilde_H_e}
  \widetilde{\mc{H}}_\e:= \left\{
    \begin{array}{ll}
      \!\!w \in \mc{H}_\e: &T^j(w)=0 \text { on }\Gamma^j_\e \text{ for all } j=1,\dots,k_1,\\[3pt]
                           &T^j(w)=0 \text { on } S^j_\e \text{ for all } j=k_1+1,\dots,k_1+ k_2\!\!
    \end{array}
\right\},
\end{equation}
and, for $\e=0$,
\begin{equation}\label{def_tilde_H_0}
  \widetilde{\mc{H}}_0:= \big\{
    w \in \mc{H}_0: T^j(w)=0 \text { on }\Gamma^j_0 \text{ for all } j=1,\dots,k_1\big\}.
  \end{equation}
In Section \ref{sec_equivalent_Aharonov_Bohm} we construct a function 
\begin{equation}\label{eq:Theta-e-intro}
    \Theta_\e:\R^2\setminus\{a^j_\e:j=1,\dots,k\}\to\R
\end{equation}
    such that 
    \begin{equation}\label{eq:proprieta_Theta-eps}
\begin{cases}
  \Theta_\e\in C^\infty(\R^2\setminus \Gamma_\e)\\
  \text{$\nabla \Theta_\e$ can be extended to be in
    $C^\infty(\R^2\setminus\{a^j_\e:j=1,\dots,k\})$ with
    $\nabla\Theta_\e=\mathcal A_\e$},
\end{cases}        
    \end{equation}
     see \eqref{def_Thetae} for the definition of $\Theta_\e$. The phase multiplication
\begin{equation}\label{eq:ph.transf}
    u(x)\mapsto v(x):=e^{-i\Theta_\e(x)} u(x), \quad  x \in  \Omega\setminus\Gamma_\e,
\end{equation}
transforms any solution $u$  to problem \eqref{prob_Aharonov-Bohm_multipole} into 
a solution $v$  to 
\begin{equation}\label{prob_eigenvalue_gauged_multipole}
\begin{cases}
-\Delta v= \la  v,  &\text{ in } \Omega \setminus \Gamma_\e,\\
v=0, &\text{on } \partial \Omega,\\
T^j(v)=0, &\text { on }\Gamma^j_\e \text{ for all } j=1,\dots,k_1,\\
T^j(\nabla v\cdot \nu^j)=0, &\text {on }\Gamma^j_\e \text{ for all } j=1,\dots,k_1,\\
T^j(v)=0, &\text {on } S^j_\e \text{ for all } j=k_1+1,\dots,k_1+k_2,\\
T^j(\nabla v\cdot \nu^j)=0, &\text {on }S^j_\e \text{ for all } j=k_1+1,\dots,k_1+k_2,
\end{cases}
\end{equation}
In \eqref{def_Theta0} we also define a function 
\begin{equation}\label{eq:Theta-0-intro}
    \Theta_0:\R^2\setminus\{0\}\to\R
\end{equation}
    satisfying
    \begin{equation}\label{eq:proprieta_Theta-0}
\begin{cases}
  \Theta_0\in C^\infty(\R^2\setminus \Gamma_0)\\
  \text{$\nabla \Theta_0$ can be extended to be in
    $C^\infty(\R^2\setminus\{0\})$ with
    $\nabla\Theta_0=\frac{1+(-1)^{k+1}}{2}A_0$}.
\end{cases}        
    \end{equation}
The gauge transformation 
\begin{equation}\label{eq:gauge0}
    u(x)\mapsto v(x):=e^{-i\Theta_0(x)} u(x), \quad  x \in  \Omega\setminus\Gamma_0,
\end{equation}
shows that the limit eigenvalue
problem \eqref{prob_Aharonov-Bohm_0} is equivalent to  
\begin{equation}\label{prob_eigenvlaue_guged0}
\begin{cases}
-\Delta v= \la  v,  &\text{in } \Omega \setminus \Gamma_0,\\
v=0, &\text{on } \partial \Omega,\\
T^j(v)=0, &\text {on }\Gamma^j_0 \text{ for all } j=1,\dots,k_1,\\
T^j(\nabla v\cdot \nu^j)=0, &\text {on }\Gamma^j_0 \text{ for all } j=1,\dots,k_1,
\end{cases}
\end{equation} 
in the sense that the two problems have the same eigenvalues and their
eigenfunctions match each other via the phase multiplication \eqref{eq:gauge0}, see Section
\ref{sec_equivalent_Aharonov_Bohm} for details.  Therefore, under
assumption \eqref{eq:simple},  $\lambda_{0,n_0}$ is also a
simple eigenvalue of \eqref{prob_eigenvlaue_guged0}. Let
\begin{equation}\label{eq:v0}
v_0\text{ be an eigenfunction of \eqref{prob_eigenvlaue_guged0} associated to
$\lambda_{0,n_0}$ such that
$\norm{v_0}_{L^2(\Omega)}=1$};
\end{equation}
it is not restrictive to assume that
$v_0$ is real-valued, see Remark \ref{remark_va_real}. 
Once $v_0$ is fixed as above, for every
  $\e\in(0,1]$ we define
\begin{equation}\label{def_Le}
  L_\e: \mc{H}_1 \to\R,\quad L_\e(w):= 2 \sum_{j=1}^{k_1+k_2} \int_{S^j_\e}\nabla v_0 \cdot \nu^j
  \gamma_+^j (w)\, dS
\end{equation}
and 
\begin{equation}\label{def_Je}
  J_\e: \mc{H}_\e\to\R,\quad
  J_\e(w):=\frac{1}{2}\int_{\Omega \setminus \Gamma_\e} |\nabla w|^2 \, dx +L_\e(w).
\end{equation}
As proved in Proposition \ref{prop_Ve_existence}, for every
$\e\in(0,1]$ there exists a
unique $V_\e\in\mathcal H_\e$ such that 
\begin{equation}\label{eq:Vepsin}
  V_\e-v_0\in
  \widetilde{\mathcal H}_\e\quad\text{and}\quad
  J_\e(V_\e)=  \min\left\{J_\e(w): w\in \mathcal H_\e \text{ and }w -v_0  \in  \widetilde{\mc{H}}_\e \right\}.
\end{equation}
Our first main result is the following expansion of the eigenvalue
variation $\la_{\e,n_0} - \la_{0, n_0}$ in terms~of 
\begin{equation}\label{eq:defEe}
    \mathcal E_\e=J_\e(V_\e)
\end{equation} and
$L_\e(v_0)$. 
\begin{theorem}\label{t:main1}
  Under assumption \eqref{eq:simple}, let $v_0$ be as in
  \eqref{eq:v0}. Then
\begin{equation}\label{eq_asymptotic_eigenvlaues-intro}
  \la_{\e, n_0} -\la_{0, n_0}=2(\mc{E}_\e-L_\e(v_0))+o\big(\|V_\e\|^2_{\mc{H}_\e}\big) \quad \text{as  }\e \to 0^+,
\end{equation}
where $\mathcal E_\e$ and $V_\e$ are defined in \eqref{eq:defEe} and \eqref{eq:Vepsin},
respectively.
\end{theorem}
\subsection{The case \texorpdfstring{$k$}{k} odd}
For $k$ odd, the asymptotic behaviour of $\mathcal E_\e$ as $\e\to0^+$
can be quantified in terms of the vanishing order of $v_0$ at the
collision point $0$. Indeed, as detailed in Proposition
\ref{prop_v0_asympotic_odd}, if $k$ is odd, there exists
$\beta\in \R\setminus\{0\}$ such that, as $\e \to 0^+$,
  \begin{equation}\label{asy-v0}
  \e^{-\frac{m}{2}} v_0\big(\e\cos t,\e\sin t\big) \to \beta \,f(t)\, 
\sin\left(\tfrac{m}{2}(t-\alpha_0)\right) 
\end{equation}
in
$C^{1,\tau}\big([0,2\pi]\setminus\{\alpha^j+\pi\}_{j=1}^{k_1},\R\big)$
for all $\tau \in (0,1)$, where $m\in\N$ is odd and corresponds to the
number of nodal lines of $v_0$ meeting at $0$ (which equals the number
of nodal lines of eigenfunctions of \eqref{prob_Aharonov-Bohm_0}
associated to $\lambda_{0,n_0}$),
$\alpha_0\in \big[0,\frac{2\pi}m\big)$ is the minimal slope of such
nodal lines, and
\begin{equation}\label{def_f}
f:[0,2\pi]\to\{-1,1\},\quad 
f(t):= \prod_{j=1}^{k_1}(-1)^{\rchi_{[\alpha^j+\pi,2 \pi)}(t)},
\end{equation}
where 
\begin{equation}\label{eq:chi_step_func}
  \rchi_{[\alpha^j+\pi,2 \pi)}(t):=
\begin{cases}
  0, & \text {if } t \in [0,\alpha^j+\pi),\\
  1, & \text {if } t \in [\alpha^j+\pi,2 \pi).
\end{cases}
\end{equation}
From \eqref{asy-v0} we realize \colg{that} the $m$ nodal lines of $v_0$ which meet at $0$
are tangent to the $m$ straight half-lines 
\begin{equation*}
    \mathcal R_j=\Big\{\big(\cos\big(\alpha_0+j\tfrac{2\pi}{m}\big),
    \sin\big(\alpha_0+j\tfrac{2\pi}{m}\big)\big)r:r\geq0\Big\},
    \quad j=0,1,\dots,m-1,
\end{equation*}
which divide the whole $2\pi$-angle into $m$ equal sectors.
We define the functional space
\begin{equation}\label{eq:defX}
  \widetilde{\mathcal X}:=
  \left\{\hskip-3pt
\begin{array}{ll}
  w \in L^1_{\rm loc}(\R^2):\hskip-5pt &w\in
                    H^1(D_r\setminus \Gamma_1) \text{ for all } r>0,\\[3pt]
  &\nabla w \in L^2(\R^2\setminus \Gamma_1,\R^2), \
  T^j(w)=0 \text{ on } \Gamma_0^j \text{ for all } j=1,\dots,k_1
\end{array}
\hskip-3pt\right\},
\end{equation}
and consider its closed subspace
\begin{equation}\label{eq:deftildeH}
  \widetilde{\mc{H}}:=\{w \in \widetilde{\mathcal X}: T^j(w)=0
  \text{ on } S_1^j \text{ for any } j=1,\dots,k_1+ k_2\}.
\end{equation}
Letting
\begin{equation}\label{def_Psi}
  \Psi_0(x)=\Psi_0(r \cos t,r\sin t)=\beta  \, r^{\frac{m}{2}}
\,f(t)\, 
\sin\left(\tfrac{m}{2}(t-\alpha_0)\right) 
\end{equation}
with $f$, $m$, $\beta$, and $\alpha_0$ as in \eqref{asy-v0}, we observe that the
nodal set of $\Psi_0$ is given by $\bigcup_{j=0}^{m-1}\mathcal R_j$. We define 
\begin{equation}\label{def_L}
  L: \widetilde{\mathcal X}\to\R,\quad
  L(w):= 2 \sum_{j=1}^{k_1+k_2} \int_{S^j_1}\nabla \Psi_0 \cdot \nu^j
  \gamma_+^j (w)\, dS
\end{equation}
and 
\begin{equation}\label{def_J}
  J: \widetilde{\mathcal X}\to\R,\quad
  J(w):=\frac{1}{2}\int_{\R^2 \setminus \Gamma_1} |\nabla w|^2 \, dx +L(w).
\end{equation}
We observe that $L(w)$ is well-defined also for any function $
w\in H^1(D_1\setminus{\Gamma_1})$.

Let $\eta\in C^\infty_{\rm c}(\R^2)$ be a radial cut-off function 
such that
\begin{equation}\label{eq:def_eta}
        \begin{cases} 
        0\le \eta(x) \le 1 \text{ for all $x \in \R^2$},\\ 
    \eta(x)=1 \text{ if $x\in D_{1}$},\quad \eta(x)=0 \text{ if $x\in \R^2 \setminus D_{2}$},\\
    |\nabla \eta|\leq 2 \text{ in $D_{2}\setminus D_{1}$}.
\end{cases}
\end{equation}
As proved in Proposition \ref{prop_Vtilde_existence}, there exists a
unique $\widetilde V\in \widetilde{\mathcal X}$ such that
\begin{equation}\label{eq:minJ}
  \widetilde{V}-\eta \Psi_0\in \widetilde{\mathcal H}\quad\text{and}\quad
  J(\widetilde V)=  \min\left\{J(w): w\in \widetilde{\mathcal X}
    \text{ and }w
    -\eta \Psi_0\in \widetilde{\mathcal H}\right\}.
\end{equation}

\begin{theorem}\label{t:exp-odd}
     Let $k$ be odd. Under assumption \eqref{eq:simple}, let $v_0$ be as in
  \eqref{eq:v0}. Then
  \begin{enumerate}[\rm (i)]
      \item $\lim_{\e\to 0^+}\e^{-m}\mathcal E_\e =\mathcal E$,
where $m$ is the vanishing order of $v_0$ at $0$ as in \eqref{asy-v0} and 
\begin{equation}\label{def_E_limit}
     \mathcal E=J(\widetilde V)=\min_{\eta\Psi_0+\widetilde{\mathcal H}}J;
\end{equation}
    \item $\la_{\e, n_0} -\la_{0, n_0}=2\,\e^m\big(\mathcal E-L(\Psi_0)\big)+o(\e^m)$ 
    as  $\e \to 0^+$.
\end{enumerate}
\end{theorem}

The expansion proved in Theorem \ref{t:exp-odd}-(ii) identifies the sharp
asymptotic behaviour of the eigenvalue variation $\la_{\e, n_0} -\la_{0, n_0}$ if 
$\mathcal E-L(\Psi_0)\neq0$; 
if instead $\mathcal E-L(\Psi_0)=0$, Theorem \ref{t:exp-odd}-(ii) only provides the 
information that $\la_{\e, n_0} -\la_{0, n_0}$ is an infinitesimal of higher 
order than $m$. It is therefore natural to ask whether there are 
configurations of poles $\{a^j\}$ for which the quantity 
$\mathcal E-L(\Psi_0)$ does or does not vanish.
The following proposition gives an answer in this sense, also providing 
precise information on the sign of the eigenvalue variation in two remarkable 
cases: the case in which each pole moves along the tangent to a nodal line of 
the limit eigenfunction and the case in which each pole moves along the 
bisector between two nodal lines. 
\begin{proposition}\label{p:two-cases}
    Let $k=k_1\leq m$ be odd and $k_2=0$. Under assumption \eqref{eq:simple}, 
    let $v_0$ be as in \eqref{eq:v0} and $\alpha_0$ as in \eqref{asy-v0}.
    For every $j\in \{1,\dots,k_1\}$ let $\alpha^j$ be as in \eqref{def_aj}.

\smallskip
\begin{enumerate}[\rm (i)]\setlength\itemsep{0.8em}
\item If
  $\alpha^j\in \{\alpha_0+\ell\tfrac{2\pi}{m}:\ell=0,1,2,\dots,m-1\}$
  for all $j\in\{1,\dots,k_1\}$, then
\begin{equation*}
\mathcal E<0 \quad\text{and}\quad L(\Psi_0)=0;
\end{equation*}
furthermore, $\la_{\e, n_0} <\la_{0, n_0}$ provided that $\e>0$ is sufficiently small.
\item If
  $\alpha^j\in \{\alpha_0+(1+2\ell)
  \tfrac{\pi}{m}:\ell=0,1,2,\dots,m-1\}$ for all
  $j\in\{1,\dots,k_1\}$, then
\begin{equation*}
\mathcal E>0 \quad\text{and}\quad L(\Psi_0)=0;
\end{equation*}
furthermore, $\la_{\e, n_0} >\la_{0, n_0}$ provided that $\e>0$ is sufficiently small.
\item There exists a choice of $\{\alpha^j:j=1,\dots,k\}$ such that 
$\mathcal E-L(\Psi_0)=0$ and  $\la_{\e,n_0}-\la_{0,n_0}=o(\e^{m})$ as $\e\to0^+$.
\end{enumerate}
\end{proposition}

The proof of claim (iii) in Proposition \ref{p:two-cases} is based on
a continuity argument. Indeed, the function $\mathcal E-L(\Psi_0)$
varies continuously under rotations of the configuration of poles, see
Theorem \ref{theorem_continuity}. Hence (i) and (ii), together with
Bolzano's Theorem, guarantee the existence of intermediate
configurations for which $\mathcal E-L(\Psi_0)$ vanishes.  The proof
of claims (i) and (ii) highlights the fact that, analogously to
$\mathcal{E}_\e$, also $\mathcal{E}$ represents an intermediate notion
between the capacity and the torsional rigidity of the set
$\cup_{j=1}^{k_1}S_1^j$. Indeed, in case (i) it occurs that
\begin{equation*} 
  \mathcal{E}=\min_{w\in \widetilde{\mathcal{H}}}
  \left\{ \frac{1}{2}\int_{\R^2\setminus\Gamma_1}|\nabla w|^2\,dx+L(w)\right\}<0,
\end{equation*}
see \eqref{eq:231}, i.e. $\mathcal{E}$ is the minimum of a functional
containing a (quadratic) energy term and a linear one, over a linear
space: this makes it somehow behaving like a torsional rigidity of the
set $\cup_{j=1}^{k_1}S_1^j$. On the other hand, in case (ii) we have
the characterization
\begin{equation*}
  \mathcal{E}=\min \left\{ \frac{1}{2}\int_{\R^2\setminus\Gamma_1}
    |\nabla w|^2\,dx\colon w-\eta \Psi_0\in\widetilde{\mathcal{H}} \right\}>0,
\end{equation*}
see \eqref{eq:232}, which yields a notion resembling that of
$\Psi_0$-capacity of the set $\cup_{j=1}^{k_1}S_1^j$.

The proof of Theorem \ref{t:exp-odd} is based on a blow-up
analysis, which also provides the following result on the behavior
of eigenfunctions, characterizing their blow-up profile and
quantifying the convergence speed of the eigenfunctions of problem
\eqref{prob_Aharonov-Bohm_multipole}  towards the corresponding
eigenfunction of the limit problem \eqref{prob_Aharonov-Bohm_0}.

\begin{theorem}\label{t:autofunzioniAB}
  Let $k$ be odd and $n_0\in \N\setminus\{0\}$ be such that
  \eqref{eq:simple} is satisfied.  Let $u_0$ be an eigenfunction of
  \eqref{prob_Aharonov-Bohm_0} associated to $\lambda_{0,n_0}$ such
  that $\int_\Omega |u_0|^2\,dx=1$. For every $\e\in(0,1]$, let
  $u_\e\in H^{1,\e}_0(\Omega,\C)$ be the eigenfunction of
  \eqref{prob_Aharonov-Bohm_multipole} associated to the eigenvalue
  $\lambda_{n_0,\e}$ such that
    \begin{equation}\label{eq:ipo-auto}
        \int_\Omega |u_\e|^2\,dx=1\quad\text{and}\quad 
        \int_\Omega e^{-i(\Theta_\e-\Theta_0)}u_\e\overline{u_0}\,dx
        \text{ is a positive real number},
    \end{equation}
    where $\Theta_\e$ and $\Theta_0$ are as in 
    \eqref{eq:Theta-e-intro}--\eqref{eq:proprieta_Theta-eps} and 
    \eqref{eq:Theta-0-intro}--\eqref{eq:proprieta_Theta-0}, respectively.
    Then
    \begin{equation}\label{eq:asy-ue1}
        \e^{-m/2}u_\e(\e \,\cdot)\to e^{i\Theta_1}(\Psi_0-\widetilde V)
        \quad\text{as $\e\to0^+$ }
    \end{equation}
    in $H^{1,1}(D_R,\C)$ for all $R>0$, where $\widetilde V$ and
    $\Psi_0$ are as in \eqref{eq:minJ} and \eqref{def_Psi},
    respectively. Moreover,
    \begin{equation}\label{eq:asy-ue2}
        \lim_{\e\to0^+}\e^{-m}\int_{\R^2\setminus\Gamma_1}\left|e^{-i(\Theta_\e-\Theta_0)}(i\nabla+\mathcal A_\e)u_\e-(i\nabla+A_0)u_0\right|^2\,dx=\|\nabla \widetilde V\|_{L^2(\R^2\setminus\Gamma_1)}^{2}.
    \end{equation}
    \end{theorem}
    We observe that condition \eqref{eq:ipo-auto} allows us to
    identify, among all the eigenfunctions of
    \eqref{prob_Aharonov-Bohm_multipole} associated to the eigenvalue
    $\lambda_{n_0,\e}$ (that are multiples of a given one due to the
    simplicity of $\lambda_{n_0,\e}$), the one that converges to $u_0$
    as $\e\to0^+$.

\subsection{The case of two opposite poles \texorpdfstring{$(k_1=0$, $k_2=1)$}{k1k2}}
\label{sec:two-opposite}
In the case of two opposite poles $a^1_\e$, $a^2_\e=-a^1_\e$ colliding to $0$
from the two sides of the same straight line, we can 
rewrite the terms appearing in \eqref{eq_asymptotic_eigenvlaues-intro} in 
elliptic coordinates in the spirit of \cite[Subsection 2.2]{AFHL}, 
thus determining the dominant term in the asymptotic expansion. This allows 
us to  generalize \cite[Theorem 2.6, Theorem 2.8]{AFLeigenvar}, 
see also \cite[Theorem 1.16]{AFHL}, removing any symmetry assumption on the 
domain $\Omega$. Let us assume that
\begin{equation*}
  \text{the $n_0$-th eigenvalue $\lambda_{n_0}$ of the Dirichlet Laplacian in $\Omega$ is simple.}
\end{equation*}
We recall that, since $k$ is even in this case,
$\lambda_{n_0}=\lambda_{0,n_0}$ coincides with the  $n_0$-th eigenvalue of 
the limit problem \eqref{prob_Aharonov-Bohm_0}. 
Let 
\begin{equation}\label{def_u_0}
  u_0 \text{ be an eigenfunction of \eqref{prob_eigenvalue_keven_regular} associated to $\lambda_{n_0}=\lambda_{0,n_0}$ such that  }\int_{\Omega} u_0^2 \, dx =1.
\end{equation}
If $u_0(0) \neq0$, then, for any bounded simply connected domain
$\Omega$, a sharp expansion of the variation
$\la_{\e, n_0} -\la_{0, n_0}$ has already been obtained in
\cite[Theorem 1.2]{AFL}, see Remark \ref{rem:m=0}. Hence we assume
that $u_0(0)=0$.  Up to a suitable choice of the coordinate system,
according to the notation introduced in \eqref{def_aj}, it is not
restrictive to consider the case $\alpha^1=0$, $\alpha^2=\pi$, so
that, for some $r_1\in(0,R)$, the configuration of the two opposite
poles is given by
\begin{equation}\label{eq:due-poli}
	a^1_\e=r_1(\e,0) \quad \text{and} 
    \quad a^2_\e=r_1(-\e,0),
\end{equation}
and
\begin{equation}\label{def_S_e}
	S_\e:=S_\e^1=[-r_1\e,r_1\e] \times \{0\},
\end{equation}
see \Cref{fig:f2}.
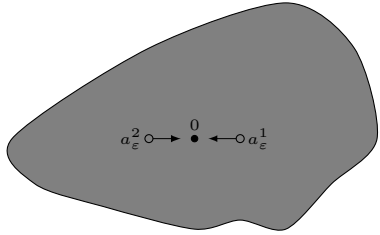
\begin{figure}[ht]
	\centering
	\begin{tikzpicture}[scale=0.6,line cap=round,line join=round,>=triangle 45,x=1.0cm,y=1.0cm]
  \clip(-4.5,-2.5) rectangle (4.5,3.5);
\filldraw[fill=gray, opacity=0.2]  plot [smooth cycle] coordinates {(-4,0) (-3.5,-1) (-2,-1.5) (0,-2) (1,-1.8) (2,-2) (3,-1) (4,0) (3.5,2) (2,3) (-1,2)};
  \draw [line width=0.3pt,color=gray] (-1,0)-- (1,0);
\draw [line width=0.3pt,-latex] (-1,0)-- (-0.3,0);
\draw [line width=0.2pt,latex-] (0.3,0)-- (1,0);
\draw [fill=black] (0,0) circle (2pt);
\draw [fill=gray] (-1,0) circle (2.5pt);
\draw[color=black] (0,0.3) node {\tiny$0$};
\draw[color=black] (1.4,0) node {\tiny$a^1_\e$};
\draw [fill=gray] (1,0) circle (2.5pt);
\draw[color=black] (-1.4,0) node {\tiny$a^{2}_\e$};
\end{tikzpicture}
\caption{Two opposite poles colliding at $0$  ($k_1=0$, $k_2=1$).}
\label{fig:f2}
\end{figure}
Furthermore, since $u_0(0) =0$, it is well known that there exists
$m \in \mathbb{N}\setminus\{0\}$, $\beta \in \R \setminus \{0\}$ and
$\alpha_0 \in [0,\frac {\pi}{m})$ such that
\begin{equation}\label{asymptotic_u0_even}
	r^{-m}u_0(r \cos t,r \sin t) \to \beta \sin(m(t-\alpha_0)) \quad \text{ in } C^{1,\tau}([0,2 \pi],\mb{C})\quad\text{as $r \to 0^+$},
\end{equation}
for any $\tau \in (0,1)$. In particular, the $2m$ half-lines 
with slopes $\alpha_0+j \frac{\pi}{m}$, $j=0,\dots, 2m-1$, are tangent to 
the nodal lines of $u_0$ meeting at $0$.

\begin{remark}\label{remark_beta}
  By standard regularity theory, $u_0$ is analytic in $\Omega$.  Let
  $T_m$ be the Taylor polynomial of $u_0$ centered at $0$ of order
  $m$, with $m \in \mb{N}\setminus\{0\}$ being as in
  \eqref{asymptotic_u0_even}.  Then, in view of
  \eqref{asymptotic_u0_even},
	\begin{equation}\label{def_T_m}
          T_m(x_1,x_2)=\sum_{j=0}^m \frac{1}{(m-j)!j!}
          \pd{^mu_0}{x^{m-j}_1\partial x_2^j}(0)\,x_1^{m-j}x_2^j.
	\end{equation}
	For every $t \in [0,2\pi]$, we have 
	\begin{align*}
		&T_m(\cos t,\sin t)= \beta \sin(m(t-\alpha_0)),\\
		&(\nabla T_m)(\cos t,\sin t)\cdot(-\sin t,\cos t)=m\beta \cos(m(t-\alpha_0)).
	\end{align*}
	Hence
 \begin{equation*}
		\frac{1}{m!}\pd{^m u_0}{x^{m}_1}(0)=-\beta
                \sin(m\alpha_0)
                \quad \text{ and } \quad \frac{1}{(m-1)!}
                \pd{^m u_0}{x^{m-1}_1\partial x^{1}_2}(0)=m\beta \cos(m\alpha_0),
	\end{equation*}
	so that, in particular, 
\begin{equation}\label{eq:beta2}
  \beta=\frac{(-1)^j}{m!}\pd{^m u_0}{x^{m-1}_1\partial x^{1}_2}(0)
  \quad
  \text{if } \alpha_0=\frac{j\pi}{m}\quad \text{for some $j\in\{0,1,\dots,2m-1\}$},
	\end{equation}	
 and
 \begin{equation}\label{eq:beta1}
   \beta =\frac{(-1)^{j+1}}{m!}\pd{^m u_0}{x^{m}_1}(0)
   \quad  \text{if } \alpha_0=\frac{\pi}{2m}+\frac{j\pi}{m}\quad
   \text{for some $j\in\{0,1,\dots,2m-1\}$}. 
  \end{equation}
\end{remark}
If the segment $S_\e$ is tangent to a nodal line of $u_0$, i.e. if
$\alpha_0=\frac{j\pi}{m}$ for some $j\in\{0,1,\dots,2m-1\}$, we have
the following result which generalizes \cite[Theorem 2.8]{AFLeigenvar}
dropping any symmetry assumption on $\Omega$.
\begin{theorem}\label{theorem_2poles_nodal}
  Let $\la_{n_0}$ be a simple eigenvalue of
  \eqref{prob_eigenvalue_keven_regular} and let $u_0$ be as in
  \eqref{def_u_0}.  Assume that $u_0(0)=0$ and let
  $m \in \mb{N} \setminus \{0\}$ and $\alpha_0$ be as in
  \eqref{asymptotic_u0_even}.  Let $k_1=0$ and $k_2=1$ with the
  configuration of poles as in assumption \eqref{eq:due-poli}.  If
  $\alpha_0=\frac{j\pi}{m}$ for some $j\in\{0,1,\dots,2m-1\}$, then
 \begin{equation*}
\la_{\e, n_0} -\la_{n_0}		
=-\frac{m\pi \beta^2
  r_1^{2m}}{4^{m-1}}\binom{m-1}{\lfloor\frac{m-1}{2}\rfloor}^{\!2}\e^{2m}
+o(\e^{2m})\quad\text{as }\e\to0^+,
	\end{equation*}
 with $\beta$ as in \eqref{eq:beta2}.
\end{theorem}
On the other hand, if $S_\e$ lays on the bisector of the angle between
the tangents to nodal lines, i.e. if
$\alpha_0=\frac{\pi}{2m}+\frac{j\pi}{m}$ for some
$j\in\{0,1,\dots,2m-1\}$, then we prove the following expansion.
\begin{theorem}\label{theorem_2poles_bisector}
Let $\la_{n_0}$ be a simple eigenvalue of 
\eqref{prob_eigenvalue_keven_regular} and let $u_0$ be as in \eqref{def_u_0}. 
Assume that $u_0(0)=0$ and let $m \in \mb{N} \setminus \{0\}$ and 
$\alpha_0$ be as in \eqref{asymptotic_u0_even}.	
Let $k_1=0$ and $k_2=1$  with the configuration of poles as in 
assumption \eqref{eq:due-poli}. If $\alpha_0=\frac{\pi}{2m}+\frac{j\pi}{m}$ 
for some $j\in\{0,1,\dots,2m-~\!\!1\}$, then
	\begin{equation*}
          \la_{\e, n_0} -\la_{n_0}=\frac{m\pi \beta^2r_1^{2m}}{4^{m-1}}
          \binom{m-1}{\lfloor\frac{m-1}{2}\rfloor}^{\!2}\e^{2m}+o(\e^{2m})\quad\text{as }\e\to0^+,
	\end{equation*}
 with $\beta$ as in \eqref{eq:beta1}.
\end{theorem}
We observe that Theorem \ref{theorem_2poles_bisector} is a
generalization of \cite[Theorem 2.6]{AFLeigenvar} and \cite[Theorem
1.16]{AFHL}.

\medskip The rest of the paper is organized as follows. In Section
\ref{sec_preliminaries} we collect some basic facts, such as the gauge
invariance property of the problem, useful features of the functional
spaces involved, and some known results that will be used in the rest
of the paper. In Section \ref{sec_capacity_torsion} we provide some
preliminary estimates on the quantities $\mc{E}_\e$ and $L_\e$ that
appear in formula \eqref{eq_asymptotic_eigenvlaues-intro}; such
estimates are used in Section \ref{sec_asymptotic_expansion}, where
the proof of Theorem \ref{t:main1} is completed.  In Section
\ref{sec_blow_up} we perform a blow-up analysis of the potential
$V_\e$ appearing in \eqref{eq:Vepsin}, in the case $k$ odd; this is
the key ingredient in the proof of Theorems \ref{t:exp-odd} and
\ref{t:autofunzioniAB}. In the same section we also complete the proof
of \Cref{p:two-cases}.  Finally, in Section \ref{sec_two_poles} we
consider the case of two poles colliding to $0$ from opposite sides of
the same straight line, thus proving Theorems
\ref{theorem_2poles_nodal} and \ref{theorem_2poles_bisector}.

\section{Preliminaries} \label{sec_preliminaries}
\subsection{Scalar potential functions for \texorpdfstring{$A_b$}{Ab} outside half-lines}\label{sec_agaugetransformation}
The construction of the gauge transformation, which makes problems
\eqref{prob_Aharonov-Bohm_multipole} and \eqref{prob_Aharonov-Bohm_0}
equivalent to eigenvalue problems for the Laplacian in domains with
straight cracks, is based on the remark that, since Aharonov-Bohm
vector fields are irrotational, they are gradients of some scalar
potential functions in simply connected domains, such as the
complement of straight half-lines starting at the pole.

For every $b=(b_1,b_2) \in \R^2$, let  $\theta_b: \R^2\setminus\{b\}
\to [0,2\pi)$ be defined as 
\begin{equation*}
\theta_b(x_1,x_2):=
\begin{cases}
  \arctan\big(\frac{x_2-b_2}{x_1-b_1}\big), &\text{if } x_1>b_1,\  x_2\ge b_2,\\
  \frac{\pi}{2}, &\text{if } x_1=b_1,\ x_2> b_2,\\
  \pi+\arctan\big(\frac{x_2-b_2}{x_1-b_1}\big),    &\text{if } x_1<b_1,\\
  \frac{3}{2}\pi, &\text{if } x_1=b_1,\ x_2< b_2,\\
  2\pi +\arctan\big(\frac{x_2-b_2}{x_1-b_1}\big), &\text{if }
  x_1>b_1,\ x_2< b_2,
\end{cases}
\end{equation*}
i.e.,
\begin{equation*}
  \theta_b\big(b+r(\cos t,\sin t)\big)=t\quad\text{for all $t \in
    [0,2\pi)$ and $r>0$}.
\end{equation*}
We observe that  $\theta_b \in
C^\infty(\R^2\setminus\{(x_1,b_2):x_1\geq b_1\})$ and  $\nabla
\theta_b$ can be extended to be in $C^\infty(\R^2\setminus\{b\})$, with  $\nabla
\big(\frac{\theta_b}2\big)= A_b$ in $\R^2\setminus\{b\}$.
For every $b\in\R^2$, $\alpha \in \R$, and
$x=(x_1,x_2)\in \R^2$, we define
\begin{equation}\label{def_R_b}
R_{b,\alpha}(x):=
\begin{bmatrix}
b_1\\
b_2\\
\end{bmatrix}
+M_\alpha
\begin{bmatrix}
x_1-b_1\\
x_2-b_2\\
\end{bmatrix},
\end{equation}
with
\begin{equation}\label{eq:def_Mzeta}
    M_\alpha:=\begin{bmatrix}
\cos\alpha&-\sin\alpha\\
\sin\alpha&\cos\alpha\\
\end{bmatrix},
\end{equation}
i.e., $R_{b,\alpha}$ is a rotation about $b$ by an angle $\alpha$.
Let 
\begin{equation}\label{def_theta_alphab}
\theta_{b,\alpha}:= \theta_b \circ  R_{b,\alpha},
\end{equation}
so that $\theta_{b,\alpha}(b+r(\cos t,\sin t))=\alpha+t$ for
  every $r>0$ and $t\in[-\alpha,-\alpha+2\pi)$.  We observe that
$\theta_{b,\alpha}$ is smooth in
  $\R^2\setminus\{b+r(\cos\alpha,-\sin\alpha):r\geq0\}$ and 
$\nabla \theta_{b,\alpha}$ can be extended to be in
$C^\infty(\R^2\setminus\{b\})$, with
$\nabla \big(\frac{\theta_{b,\alpha}}2\big)= A_b$, see \Cref{fig:theta_b_alpha}.

\begin{figure}[ht]
	\centering
	\begin{tikzpicture}[scale=0.5,line cap=round,line join=round,>=triangle 45,x=1.0cm,y=1.0cm]
  \clip(-3.5,-3.5) rectangle (3.5,3.5);
\draw [line width=0.5pt,thick] (0,0)-- (3.4641,-2);
\draw [line width=0.5pt, dashed] (0,0)-- (4,0);
\draw [line width=0.5pt, dashed] (0,0)-- (3.4641,2);
\draw [fill=gray] (0,0) circle (2.5pt);
\begin{scriptsize}
\draw (-0.3,0.3) node {$b$};
\draw[line width=0.5pt, thick] (1,0) arc (0:30:1);
\draw (1.4,0.4) node {$\alpha$};
\draw[line width=0.5pt,thick, ->] (3,0) arc (0:220:3);
\draw[line width=0.5pt,thick,] (3,0) arc (0:-30:3);
\draw (2,3) node {$\theta_{b,\alpha}$};
\end{scriptsize}
\end{tikzpicture}
\caption{$\theta_{b,\alpha}$ is smooth in
  $\R^2\setminus\{b+r(\cos\alpha,-\sin\alpha):r\geq0\}$.}
  \label{fig:theta_b_alpha}
\end{figure}
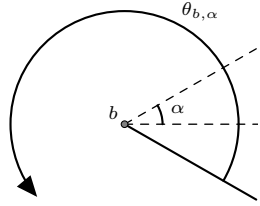

\subsection{Some remarks on functional spaces}\label{sec_notations}
In this subsection we describe some properties of the functional
  spaces $\mathcal H_\e$ introduced in Section \ref{sec_main_results}.

\begin{remark}\label{remark_embbeing_compact}
  The natural embedding $I: \mc{H}_\e \to L^2(\Omega)$ is
  compact. Indeed, we can cut $\Omega$ along the lines $\Sigma^j$ for
  $j=1 \dots ,k_1+k_2$, where $\Sigma^j$ are defined in Section
  \ref{sec_main_results}. Then we can use classical compact embedding
  results for each resulting subset, see for example \cite[Theorem
  12.30]{Leo}.
\end{remark}

Arguing as in \Cref{remark_embbeing_compact}, from the Poincaré
inequality for functions vanishing on a portion of the boundary we can
deduce the following Poincaré inequality in $\mc{H}_1$, and hence in
$\mc{H}_\e$ for any $\e \in [0,1]$.
\begin{proposition}\label{prop_poincare}
There exists  a constant $C_P>0$ such that, for every $\e\in [0,1]$
and $w \in \mc{H}_{\e}$,
\begin{equation*}
  \int_{\Omega} w^2 dx \le C_P \int_{\Omega \setminus \Gamma_\e} |\nabla w|^2 \, dx.
\end{equation*}
\end{proposition}

Since
$\Omega\setminus\Gamma_1\subseteq\Omega\setminus\Gamma_{\e_1}
\subseteq\Omega\setminus\Gamma_{\e_2}$, we have
$\mathcal H_{\e_2}\subseteq \mathcal H_{\e_1}\subseteq\mathcal H_1$
for all $0\leq \e_2\leq\e_1\leq1$. Proposition \ref{p:mosco} below
establishes a Mosco-type convergence result for the spaces
$\mathcal H_\e$ as $\e\to0^+$.

  \begin{proposition}\label{p:mosco}
    Let $\{\e_n\}\subset(0,1)$ be such that
    $\lim_{n\to\infty}\e_n=0$. If $\{v_n\}_n\subset\mathcal H_1$ and
    $v\in \mathcal H_1$ are such that $v_n\in \mathcal H_{\e_n}$ for
    all $n\in \N$ and $v_n\rightharpoonup v$ in $\mathcal H_1$ as
    $n\to\infty$, then $v\in \mathcal H_0$.
  \end{proposition}
  \begin{proof}
    For every $\e\in(0,1]$, there exists $n(\e)\in\N$ such that
    $v_n\in \mathcal H_\e$ for all $n>n(\e)$. The weak convergence
    $v_n\rightharpoonup v$ in $\mathcal H_1$ then implies that
    $v\in \mathcal H_\e$ for all $\e\in(0,1]$. It follows that there
    exists $\mathbf{f}\in L^2(\Omega,\R^N)$ such that
    $\nabla v=\mathbf{f}$ in $\mathcal D'(\Omega\setminus\Gamma_\e)$
    for all $\e\in(0,1]$.  Actually, $\nabla v=\mathbf{f}$ in
    $\mathcal D'(\Omega\setminus\Gamma_0)$, since, for every
    $\varphi\in C^\infty_{\rm c}(\Omega\setminus\Gamma_0)$,
    $\mathop{\rm supp}\varphi\subset\Omega\setminus\Gamma_\e$ for $\e$
    sufficiently small. Therefore,
    $v\in H^1(\Omega\setminus\Gamma_0)$. From the fact that
    $v\in \mathcal H_1\cap H^1(\Omega\setminus\Gamma_0)$ it follows
    that $v\in \mathcal H_0$.
  \end{proof}
Since the singleton $\{0\}$ has null capacity in $\Omega$,
  functions in  $\mathcal H_0$, respectively in $\widetilde{\mathcal
    H}_0$, can be approximated by functions vanishing  in a
  neighbourhood of $0$, as stated in Lemma \ref{l:approx-00}.
  \begin{lemma}\label{l:approx-00}\quad
    
    \begin{enumerate}[\rm (i)]
      \item The set $\mathcal H_{0,0}:=\{v\in \mathcal H_0:v\equiv0\text{ in a neighbourhood of
        }0\}$ is dense in $\mathcal H_0$. 
        \item The set $\widetilde{\mathcal H}_{0,0}:=\{v\in \widetilde{\mathcal H}_0:v\equiv0\text{ in a neighbourhood of
    }0\}$ is dense in $\widetilde{\mathcal H}_0$. 
  \end{enumerate}
\end{lemma}
\begin{proof}
  To prove (i) we first notice that, if $v\in \mathcal H_0$, then, defining $v_n$ as
  \begin{equation*}
    v_n(x)=
    \begin{cases}
      v(x),&\text{if }|v(x)|<n,\\
      -n ,&\text{if }v(x)<-n,\\
      n ,&\text{if }v(x)>n,
    \end{cases}
  \end{equation*}
  $v_n\in \mathcal H_0 \cap L^\infty(\Omega)$ for all $n\in\N$ and
  $v_n\to v$ in $\mathcal H_0$. Therefore it is enough to prove that
  $\mathcal H_{0,0}\cap L^\infty(\Omega)$ is dense in
  $\mathcal H_{0}\cap L^\infty(\Omega)$. To this aim, let us fix some
  $v\in\mathcal H_{0}\cap L^\infty(\Omega)$. For every $\e\in(0,1)$
 we consider the cut-off function $\omega_{\e}\in W^{1,\infty}(\R^2)$
  defined as  
  \begin{equation}\label{def_eta_even}
\omega_{\e}(x):=
\begin{cases}
1, &\text{if } x \in  D_{\e},\\
\frac{2\log|x|-\log\e}{\log\e}, &\text{if } x \in  D_{\sqrt{\e}} \setminus  D_{\e}, \\
0, &\text{if } x \in  \Omega \setminus D_{\sqrt{\e}}.
\end{cases}
\end{equation}
One may directly verify that
$(1-\omega_\e)v\in \mathcal H_{0,0}\cap L^\infty(\Omega)$ for all
$\e\in(0,1)$ and $(1-\omega_\e)v\to v$ in $\mathcal H_0$ as $\e\to
0$. The proof of (i) is thereby complete. We can proceed in a similar
way to obtain (ii).  
\end{proof}

  \subsection{An equivalent eigenvalue problem by gauge transformation }\label{sec_equivalent_Aharonov_Bohm}
For every $\e \in (0,1]$, using the notation introduced in
\eqref{def_theta_alphab}, we define
\begin{equation*}
  \theta_\e^j:=\begin{cases}
    \theta_{a^j_\e,\pi-\alpha^j}, &\text {if } j=1,\dots, k_1+k_2, \\[3pt]
  \theta_{a^j_\e, -\alpha^{j}},  &\text {if }
  j=k_1+k_2+1,\dots,k_1+2k_2,
\end{cases}
\end{equation*}
with  $\alpha^j$ as in \eqref{def_aj}, see \Cref{fig:theta_j_eps}, and 
\begin{equation}\label{def_Thetae}
   \Theta_\e:\R^2\setminus\{a^j_\e:j=1,\dots,k\}\to\R,\quad 
   \Theta_\e:=\frac12\sum_{j=1}^{k}(-1)^{j+1}\theta_\e^j.
\end{equation}
We observe that $\Theta_\e$ verifies \eqref{eq:proprieta_Theta-eps}.

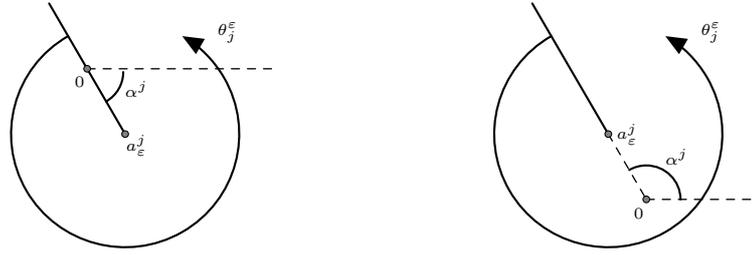
\begin{figure}
\centering
\subfloat[\scriptsize $\theta_j^\varepsilon$ for $j\leq k_1+k_2$.]
{\begin{tikzpicture}[scale=0.5,line cap=round,line join=round,>=triangle 45,x=1.0cm,y=1.0cm]
 \clip(-6,-5) rectangle (6,5);
\draw [line width=0.5pt,thick] (0,0)-- (-2,3.4641);
\draw [line width=0.5pt, dashed] (-1,1.732)-- (4,1.732);
\draw [fill=gray] (0,0) circle (2.5pt);
\draw (0.3,-0.3) node {\tiny $a^j_\varepsilon$};
\draw [fill=gray] (-1,1.732) circle (2.5pt);
\draw (-1.2,1.4) node {\tiny $0$};
\draw[line width=0.5pt,thick, ->] (-1.5,2.598) arc (120:420:3);
\draw (2.7,2.7) node {\tiny $\theta_j^\varepsilon$};
\draw[line width=0.5pt, thick] (-0.5,0.866) arc (300:360:0.9);
\draw (0.3,1.2) node {\tiny $\alpha^j$};
\end{tikzpicture}}
\quad
\subfloat[\scriptsize $\theta_j^\varepsilon$ for $j\geq k_1+k_2+1$.]
{\begin{tikzpicture}[scale=0.5,line cap=round,line join=round,>=triangle 45,x=1.0cm,y=1.0cm]
 \clip(-6,-5) rectangle (6,5);
\draw [line width=0.5pt,thick] (0,0)-- (-2,3.4641);
\draw [line width=0.5pt, dashed] (1,-1.732)-- (4,-1.732);
\draw [line width=0.5pt, dashed] (0,0)-- (1,-1.732);
\draw [fill=gray] (0,0) circle (2.5pt);
\draw (0.5,0) node {\tiny $a^j_\varepsilon$};
\draw [fill=gray] (1,-1.732) circle (2.5pt);
\draw (0.8,-2.1) node {\tiny $0$};
\draw[line width=0.5pt,thick, ->] (-1.5,2.598) arc (120:420:3);
\draw (2.7,2.7) node {\tiny $\theta_j^\varepsilon$};
\draw[line width=0.5pt, thick] (1.9,-1.732) arc (0:120:0.9);
\draw (1.8,-0.6) node {\tiny $\alpha^j$};
\end{tikzpicture}}
\caption{The angles $\theta_j^\varepsilon$ for $1\leq j\leq k_1+2k_2$. 
The half-lines represent the singular set of the function $\theta_j^\varepsilon$.
}
\label{fig:theta_j_eps}
\end{figure}

For any $\e\in(0,1]$, let $\lambda\in\R$ be an eigenvalue of problem
\eqref{prob_Aharonov-Bohm_multipole} associated to the eigenfunction
$u \in H_0^{1,\e}(\Omega,\C)\setminus\{0\}$. Then the function 
\begin{equation*}
v(x):=e^{-i\Theta_\e(x)} u(x), \quad  x \in  \Omega\setminus\Gamma_\e,
\end{equation*}
belongs to $\mathcal H_\e$ and weakly solves 
\eqref{prob_eigenvalue_gauged_multipole}, 
in the sense that $v \in \widetilde{\mc{H}}_\e$ and 
\begin{equation}\label{eq_eigenvalue_gauged_multipole}
\int_{\Omega \setminus \Gamma_\e} \nabla v \cdot \nabla w \, dx = \lambda\int_{\Omega} v w \, dx 
\quad  \text{for all } w \in \widetilde{\mc{H}}_\e,
\end{equation}
where $\widetilde{\mc{H}}_\e$ is defined in
\eqref{def_tilde_H_e}. On the other hand, if $v \in
  \widetilde{\mc{H}}_\e$ solves \eqref{eq_eigenvalue_gauged_multipole},
  then $u=e^{i\Theta_\e} v$
  solves \eqref{prob_Aharonov-Bohm_multipole}.
Therefore the eigenvalue problems \eqref{prob_Aharonov-Bohm_multipole} and 
\eqref{prob_eigenvalue_gauged_multipole}
have the same eigenvalues and their eigenfunctions match each other
via the phase $e^{-i\Theta_\e}$.

A similar gauge transformation can be made for solutions to \eqref{prob_Aharonov-Bohm_0}.
 For every $j=1,\dots, k_1$, let $\alpha^j$ be as in
\eqref{def_aj} and
\begin{equation*}
  \theta_0^j:=\theta_{0} \circ R_{0, \pi-\alpha^j}.
\end{equation*}
We define 
\begin{equation}\label{def_Theta0}
    \Theta_0:\R^2\setminus\{0\}\to\R,\quad \Theta_0(x)=\frac12
    \sum_{j=1}^{k_1}(-1)^{j+1}\theta^j_0(x).
\end{equation}
If $k_1=0$ we just take $\Theta_0\equiv0$.
We observe that $\Theta_0$ satisfies \eqref{eq:proprieta_Theta-0}. Furthermore, 
if $t\in[0,2\pi)$,
  \begin{align}
    \label{eq:theta0j}\theta_0^j(\cos t,\sin t)& =
    \begin{cases}
      t-\alpha^j+\pi,&\text{if }t\in[0,\alpha^j+\pi),\\
     t-\alpha^j-\pi,&\text{if }t\in[\alpha^j+\pi,2\pi),
   \end{cases}\\
    \notag&=-\alpha^j+t+\pi(1-2\rchi_{[\alpha^j+\pi,2\pi)}),
  \end{align}
where   $\rchi$ is defined in \eqref{eq:chi_step_func}.
We have that $u$ is an eigenfunction of problem \eqref{prob_Aharonov-Bohm_0},
associated to the eigenvalue $\la$, if and only if the function
\begin{equation}\label{def_vo_gauged}
v(x):=e^{-i\Theta_0(x)}u(x),
\quad  x \in \Omega\setminus\Gamma_0,
\end{equation}
is a non-zero weak solution of \eqref{prob_eigenvlaue_guged0}
in the sense that $v \in \widetilde{\mc{H}}_{0}$ and 
\begin{equation}\label{eq_eigenvalue_gauged0}
\int_{\Omega \setminus \Gamma_0} \nabla v \cdot \nabla w \, dx = \lambda\int_{\Omega} v w \, dx 
\quad  \text{for all } w \in \widetilde{\mc{H}}_{0},
\end{equation}
where $\widetilde{\mc{H}}_{0}$ is defined in \eqref{def_tilde_H_0}.  
We recall that, if $k_1$ is even, then, letting $v$ as in
  \eqref{def_vo_gauged}, the function $v
  e^{i\Theta_0}=u$ is an eigenfunction of the Dirichlet
  Laplacian in $\Omega$, hence it is smooth in $\Omega$.

\begin{remark}\label{remark_va_real}
  We may treat eigenfunctions of problems
  \eqref{prob_eigenvalue_gauged_multipole} and
  \eqref{prob_eigenvlaue_guged0} as real-valued functions (thus
  justifying the choice to consider $\mathcal H_\e$ as a space of real
  functions). Indeed, since all the coefficients in
  \eqref{prob_eigenvalue_gauged_multipole} and
  \eqref{prob_eigenvlaue_guged0} are real, both the real and the
  imaginary part of any eigenfunction are eigenfunctions, if not
  trivial. Hence, any eigenspace of
  \eqref{prob_eigenvalue_gauged_multipole} and
  \eqref{prob_eigenvlaue_guged0} admits a basis made of reals
  eigenfunctions. See also \cite[Subsection 2.3]{AFaharonov}.
\end{remark}

\subsection{Asymptotics of solutions to the limit eigenvalue problem}

Let $\{\alpha^j\}_{j=1}^{k_1}$ and $\rchi_{[\alpha^j+\pi,2 \pi)}$ be
as in \eqref{def_aj} and \eqref{eq:chi_step_func}, respectively.  Let
$f$ be the function defined in \eqref{def_f}.

\begin{proposition}\label{prop_v0_asympotic_odd}
  Let $k_1$ be odd. If $v$ is a non-trivial solution to
  \eqref{prob_eigenvlaue_guged0}, in the sense that
  $v\in\widetilde{\mathcal H}_0$ satisfies
  \eqref{eq_eigenvalue_gauged0}, then there exist an odd number
  $m \in \mathbb{N}$, $\beta\in \R\setminus\{0\}$, and
  $\alpha_0\in \big[0,\frac{2\pi}m\big)$ such that
  \begin{equation}\label{limit_v_0_odd}
  \e^{-\frac{m}{2}} v\big(\e\cos t,\e\sin t\big) \to \beta \,f(t)\, 
\sin\left(\tfrac{m}{2}(t-\alpha_0)\right) 
\end{equation}
in
$C^{1,\tau}\big([0,2\pi]\setminus\{\alpha^j+\pi\}_{j=1}^{k_1},\R\big)$
as $\e \to 0^+$, for all $\tau \in (0,1)$.  Moreover, there
exists a constant $C>0$ such that
\begin{equation}\label{ineq_v_0_pointiweise_bound_odd}
  |v(x)| \le C|x|^{\frac{m}{2}} 
  \quad \text{and} \quad |\nabla v(x)| \le C|x|^{\frac{m}{2}-1}
  \quad\text{for all } x \in \Omega\setminus\Gamma_0.
\end{equation}
Furthermore, letting
\begin{equation*}
  \Psi(x)=\Psi(r \cos t,r\sin t)=\beta  \, r^{\frac{m}{2}}
\,f(t)\, 
\sin\left(\tfrac{m}{2}(t-\alpha_0)\right), 
\end{equation*}
with $m$, $\beta$, and $\alpha_0$ as in \eqref{limit_v_0_odd} and $f$ as in 
\eqref{def_f}, we have that, as $\e\to0^+$,
\begin{equation}\label{eq:convH1}
  \e^{-\frac{m}{2}}v(\e \cdot)
\to \Psi \quad\text{in
  }H^1(D_\rho\setminus\Gamma_0)\text{ for all }\rho>0.
\end{equation}
\end{proposition}
\begin{proof}
  As observed above, the function
  $u:=e^{i\Theta_0}v$
  is an eigenfunction of \eqref{prob_Aharonov-Bohm_0} with $k$
  odd, i.e.
\begin{equation*}
\begin{cases}
\left(i\nabla + A_{0}\right)^2u=\la u, &\text{ in } \Omega,\\
u=0, &\text{ on } \partial \Omega,
\end{cases}
\end{equation*} 
with $A_0$ defined in \eqref{def_A_a^j}.  From \cite[Theorem 1.3,
Section 7]{FFT} it follows that there exist an odd $m \in \mathbb{N}$
and $\beta_1,\beta_2 \in \C $ such that $(\beta_1,\beta_2)\neq (0,0)$
and, as $\e \to 0^+$,
\begin{equation}\label{eq:asyu}
  \e^{-\frac{m}{2}} u(\e\cos t ,\e\sin t ) \to  e^{\frac{i}{2}t}
  \left(\beta_1 \cos\left(\tfrac{m}{2}t\right)+\beta_2
    \sin\left(\tfrac{m}{2}t\right)\right)
  \quad\text{in $C^{1,\tau}([0,2\pi],\C)$}
\end{equation}
and
\begin{multline}\label{eq:asygrad}
  \e^{1-\frac{m}{2}} \nabla u(\e\cos t ,\e\sin t ) \to  \tfrac m2 e^{\frac{i}{2}t}
  \left(\beta_1 \cos\left(\tfrac{m}{2}t\right)+\beta_2
    \sin\left(\tfrac{m}{2}t\right)\right)\boldsymbol{\theta}(t)\\
  +\tfrac{d}{dt}\left(e^{\frac{i}{2}t}
  \left(\beta_1 \cos\left(\tfrac{m}{2}t\right)+\beta_2
    \sin\left(\tfrac{m}{2}t\right)\right)\right) \boldsymbol{\tau}(t)
  \quad\text{in $C^{0,\tau}([0,2\pi],\C)$}
\end{multline}
for all $\tau \in (0,1)$, where $\boldsymbol{\theta}(t)=(\cos t, \sin
t)$ and $\boldsymbol{\tau}(t)=(-\sin t,\cos t)$. Furthermore, by
  \eqref{eq:theta0j}, for all $t\in[0,2\pi]$ we have
  \begin{equation}\label{eq:exp}
\sum_{j=1}^{k_1}(-1)^{j+1}\theta^j_0(\e\cos t ,\e\sin t )=t+\sum_{j=1}^{k_1}(-1)^j\alpha^j+\pi
      -2\pi\sum_{j=1}^{k_1}(-1)^{j+1}\rchi_{[\alpha^j+\pi,2 \pi)}(t).
  \end{equation}
From 
\eqref{eq:asyu}, the definition of $u$, and \eqref{eq:exp} it follows that
\begin{equation*}
  \e^{-\frac{m}{2}} v\big(\e\cos t ,\e\sin t \big) \to f(t)
e^{-\frac i2\left(\pi+\sum_{j=1}^{k_1}(-1)^j\alpha^j\right)}
\Big(\beta_1 \cos\left(\tfrac{m}{2}t\right)+\beta_2 \sin\left(\tfrac{m}{2}t\right)\Big) 
\end{equation*}
in
$C^{1,\tau}\big([0,2\pi]\setminus\{\alpha^j+\pi\}_{j=1}^{k_1},\C\big)$
as $\e \to 0^+$, for all $\tau \in (0,1)$. Then, 
since $v$ is real-valued (see Remark
\ref{remark_va_real}),
we have proved that there exist $c_1,c_2 \in \R $ such that
  $(c_1,c_2)\neq (0,0)$ and
\begin{equation}\label{limit_v_0_odd-old}
  \e^{-\frac{m}{2}} v\big(\e\cos t ,\e\sin t \big) \to f(t)
\Big(c_1 \cos\left(\tfrac{m}{2}t\right)+c_2 \sin\left(\tfrac{m}{2}t\right)\Big). 
\end{equation}
Letting
\begin{equation*}
  \alpha_0=
  \begin{cases}
    \frac2m\mathop{\rm arccot}\big(-\frac{c_2}{c_1}\big),&\text{if }c_1\neq0,\\
    0,&\text{if }c_1=0,
  \end{cases}
\end{equation*}
we can rewrite \eqref{limit_v_0_odd-old} as \eqref{limit_v_0_odd}. Estimate 
\eqref{ineq_v_0_pointiweise_bound_odd} is a consequence of
\eqref{eq:asyu} and \eqref{eq:asygrad}.

Finally, to prove
  \eqref{eq:convH1}, we define
     \begin{equation*}
\tilde{u}_{\e}(x):=\e^{-\frac{m}{2}}u(\e x), \quad 
\Phi(x)=\Phi(r \cos t ,r\sin t )=r^{\frac m2}
 e^{\frac{i}{2}t}
  \left(\beta_1 \cos\left(\tfrac{m}{2}t\right)+\beta_2
    \sin\left(\tfrac{m}{2}t\right)\right). 
  \end{equation*}
We observe that
  \eqref{eq:asyu}, \eqref{eq:asygrad}, and the Dominated Convergence
  Theorem imply that
  \begin{equation*}
    \nabla \tilde{u}_\e\to \nabla \Phi \quad\text{and}\quad
    \frac{\tilde{u}_\e}{|x|}\to \frac{\Phi}{|x|}\quad\text{in
    }L^2(D_\rho)\quad \text{for all }\rho>0,
  \end{equation*}
which easily provides \eqref{eq:convH1}.
\end{proof}

In the case $k$ even, solutions to \eqref{prob_eigenvlaue_guged0} are more regular.
\begin{proposition} \label{prop_v0_asympotic_even} Let $k_1$ be
  even. If $v$ is a non-trivial solution to
  \eqref{prob_eigenvlaue_guged0}, then there exist 
  $m \in \mathbb{N}$,  $\beta\in \R\setminus\{0\}$, and $\alpha_0\in
  \big[0,\frac{\pi}m\big)$ such that
\begin{equation}\label{limit_v_0_even}
  \e^{-m} v(\e\cos t ,\e \sin t ) \to \beta f(t)
 \sin(m(t-\alpha_0))
\end{equation}
in $C^{1,\tau}([0,2\pi]\setminus\{\pi+\alpha^j\}_{j=1}^{k_1},\R)$ as $\e
\to 0^+$,
 for all $\tau \in (0,1)$. Moreover, there exists a constant $C>0$ such that 
\begin{equation}\label{ineq_v_0_pointiweise_bound_even}
|v(x)| \le C|x|^m \quad\text{and}\quad |\nabla v(x)| \le
\begin{cases}
  C|x|^{m-1},&\text{if }m\geq1,\\
  C ,&\text{if }m=0,
\end{cases}
\end{equation}
for all $x \in \Omega\setminus\Gamma_0$.
\end{proposition}
\begin{proof}
  The function $u:=e^{i\Theta_0}v$ is an eigenfunction of \eqref{prob_Aharonov-Bohm_0} 
  with $k$ even, i.e. $u$ is an eigenfunction of the Dirichlet
  Laplacian. 
  From  \eqref{eq:theta0j} we deduce the analogue of \eqref{eq:exp} in the even case: 
  \begin{equation}\label{eq:exp-even}
\sum_{j=1}^{k_1}(-1)^{j+1}\theta^j_0(\e\cos t ,\e\sin t )=\sum_{j=1}^{k_1}(-1)^j\alpha^j
      -2\pi\sum_{j=1}^{k_1}(-1)^{j+1}\rchi_{[\alpha^j+\pi,2 \pi)}(t)
  \end{equation}
   for all $t\in[0,2\pi]$.
  Claims \eqref{limit_v_0_even} and \eqref{ineq_v_0_pointiweise_bound_even}
  follow from the fact that $u$ is analytic, the definition of $u$
  and \eqref{eq:exp-even}, observing that, since $k_1$ is even,
  $|\nabla u|=|\nabla v|$.
\end{proof}

\begin{remark}
  For the sake of simplicity, for any $w \in\mc{H}_0$ we simply
  write $w$ instead of $\gamma_+^j(w)$ on $S_\e^j$, since
  $\gamma_+^j(w)= \gamma_-^j(w)$ on $S_\e^j$ for any
  $j=1, \dots, k_1+k_2$.  We also simply write $v_0$,
  $\nabla v_0$ and $\nabla v_0 \cdot \nu^j$ when considering
    their traces  on $S_\e^j$.
\end{remark}

\section{Definition and properties of 
\texorpdfstring{$\mathcal E_\e$}{Ee}}\label{sec_capacity_torsion}

For some $n_0\in\N\setminus\{0\}$, let $u_0$ be an eigenfunction of
\eqref{prob_Aharonov-Bohm_0} associated to the eigenvalue
$\lambda_0=\lambda_{0,n_0}$ and $v_0$ be as in \eqref{def_vo_gauged},
so that $v_0$ is a non-zero weak solution of
\eqref{prob_eigenvlaue_guged0} with $\lambda=\lambda_0$. By Remark \ref{remark_va_real} it is
not restrictive to assume that $v_0$ is real-valued and
$\|u_0\|_{L^2(\Omega,\C)}=\norm{v_0}_{L^2(\Omega)}=1$.

Let $L_\e$ be the functional introduced in \eqref{def_Le}. We observe
that $L_\e$ is well-defined; indeed, for every $j=1,\dots,k_1+k_2$, we
have  $\nabla v_0\in L^p(S^j_\e)$ for all $p\in[1,2)$ in view of
\eqref{ineq_v_0_pointiweise_bound_odd} and
\eqref{ineq_v_0_pointiweise_bound_even}, whereas
$\gamma^j_+(w)\in L^q(S^j_\e)$ for all $w\in\mathcal H_1$ and
$q\in[2,+\infty)$ by \eqref{def_traces}.  We provide below an estimate
of the norm of $L_\e$ in $\mc{H}_1^*$, where $\mc{H}_1^*$ is the dual
space of $\mc{H}_1$.

\begin{proposition}\label{prop_Le}
  Let $m \in \mathbb{N}$ be as in Proposition
  \ref{prop_v0_asympotic_odd} for $v=v_0$, if $k$ is odd, or as in
  Proposition \ref{prop_v0_asympotic_even}, if $k$ is even.  For every
  $\e\in (0,1]$, the map $L_\e$ defined in \eqref{def_Le} belongs to
  $\mc{H}_1^*$ and, as $\e\to0^+$,
\begin{equation}\label{eq_norm_Le}
\norm{L_\e}_{\mc{H}_1^*}= 
\begin{cases}
O\big(\e^{\frac{m}{2}-1+\frac{1}{p}}\big), &\text{if $k$ is odd},\\
O\big(\e^{\frac{1}{p}}\big), &\text{if $k$ is even and } m=0,\\
O\big(\e^{m-1+\frac{1}{p}}\big), &\text{if $k$ is even and } m>0,
\end{cases}
\end{equation}
for every $p \in (1,2)$.  In particular, $L_\e\to0$ in $\mathcal
  H_1^*$ as $\e\to0^+$.
\end{proposition}
\begin{proof}
  If $k$ is odd, for every $p \in (1,2)$ and
  $w \in \mathcal H_1$, from the H\"older inequality,
  \eqref{def_traces} and \eqref{ineq_v_0_pointiweise_bound_odd} it
  follows that, letting $p'=\frac{p}{p-1}$,
\begin{equation*}
|L_\e(w)|\le 2\sum_{j=1}^{k_1+k_2} \norm{\nabla v_0}_{L^p(S^j_\e)}\|\gamma_+^j (w)\|_{L^{p'}(S^j_\e)}
\le C \e^{\frac{m}{2}-1+\frac{1}{p}}
\norm{w}_{\mathcal H_1},
\end{equation*}
for some constant $C>0$ independent of $\e$.
If $k$ is even, the proof is similar due to  \eqref{ineq_v_0_pointiweise_bound_even}.
\end{proof}
For every $\e\in (0,1]$, we now consider the functional $J_\e$ defined
in \eqref{def_Je} and,
recalling the definition of $\widetilde{\mc{H}}_\e$ in \eqref{def_tilde_H_e},
the minimization problem 
\begin{equation}\label{minimization_capacity_torsion}
  \inf\left\{J_\e(w): w\in \mathcal H_\e \text{ and }w -v_0  \in  \widetilde{\mc{H}}_\e \right\}.
\end{equation}
Note that, since $v_0 \in \widetilde{\mc{H}}_{0}$, the condition
$w -v_0 \in \widetilde{\mc{H}}_\e$ is equivalent to
\begin{equation}\label{eq:jump}
T^j(w)=
\begin{cases}
0, &\text{on } \Gamma^j_0\text{ for all } j=1,\dots,k_1,\\
2v_0, &\text{on } S^j_\e\text{ for all } j=1,\dots,k_1+k_2.
\end{cases}
\end{equation}

\begin{proposition}\label{prop_Ve_existence}
  The infimum in \eqref{minimization_capacity_torsion} is achieved by
  a unique $V_\e\in\mathcal H_\e$. Furthermore, $V_\e$ weakly solves
  the problem
\begin{equation}\label{prob_Ve}
\begin{cases}
  -\Delta V_\e=0,  &\text{in } \Omega \setminus \Gamma_\e,\\
  V_\e=0, &\text{on } \partial \Omega,\\
  T^j(V_\e-v_0)=0, &\text {on }\Gamma^j_\e \quad\text{for all } j=1,\dots,k_1,\\
  T^j(\nabla V_\e\cdot \nu^j-\nabla v_0 \cdot \nu^j)=0, &\text {on
  }\Gamma^j_\e
  \quad\text{for all } j=1,\dots,k_1,\\
  T^j(V_\e-v_0)=0, &\text {on }S^j_\e \quad\text{for all } j=k_1+1,\dots,k_1+k_2,\\
  T^j(\nabla V_\e\cdot \nu^j-\nabla v_0 \cdot \nu^j)=0, &\text {on
  }S^j_\e \quad \text{for all } j=k_1+1,\dots,k_1+k_2,
\end{cases}
\end{equation}
in the sense that $V_\e \in  \mc{H}_\e$, $V_\e -v_0 \in  \widetilde{\mc{H}}_\e$,   and 
\begin{equation}\label{eq_Ve}
\int_{\Omega \setminus \Gamma_\e} \nabla V_\e \cdot \nabla w\, dx = - L_\e(w)
\quad\text{for all }w \in \mc{\widetilde{H}}_\e.
\end{equation}
\end{proposition}

\begin{proof}
  In view of \eqref{def_Je}, the continuity of the linear operator
  $L_\e$, and Proposition \ref{prop_poincare}, we can easily verify
  that $J_\e$ is continuous and coercive on the set
  $v_0 + \widetilde{\mc{H}}_\e= \{w \in \mc{H}_\e: w -v_0 \in \widetilde{\mc{H}}_\e \}$, which is
  closed and convex.
  Moreover, $J_\e$ is convex. Therefore,
  the infimum in 
  \eqref{minimization_capacity_torsion} is achieved by some $V_\e$, which
  weakly solves \eqref{prob_Ve} in the sense of \eqref{eq_Ve}.
  If $V_{\e,1},V_{\e,2}\in v_0 + \widetilde{\mc{H}}_\e$ are weak solutions of \eqref{prob_Ve},
  then $V_{\e,1}-V_{\e,2}\in \widetilde{\mc{H}}_\e$ and 
\begin{equation}\label{prop_Ve_existence:2}
\int_{\Omega \setminus \Gamma_\e} (\nabla V_{\e,1}-\nabla V_{\e,2})
\cdot \nabla w\, dx =0\quad\text{for all }w \in \mc{\tilde{H}}_\e.
\end{equation}
Testing \eqref{prop_Ve_existence:2} with $w=V_{\e,1}-V_{\e,2}$ we
obtain  $\nabla (V_{\e,1}-V_{\e,2})=0$ and hence, by
Proposition~\ref{prop_poincare}, we conclude that $V_{\e,1}=V_{\e,2}$.
\end{proof}

For every $\e\in(0,1]$, let $J_\e$ and $V_\e$ be as \eqref{def_Je} and
Proposition \ref{prop_Ve_existence}, respectively. We consider the
quantity $\mc{E}_\e:=J_\e(V_\e)$ as in \eqref{eq:defEe}.  $\mc{E}_\e$
plays a significant role in the asymptotic expansion of the eigenvalue
variation $\la_{\e,n_0} -\la_{0,n_0}$, as the poles $a^j_\e$ move
towards the collision at~$0$.

To derive a first   upper and lower bound for $\mc{E}_\e$, we
consider, for every $r>0$, the radial cut-off function 
$\eta_r \in C^\infty_{\rm c}(\R^2)$ defined as 
\begin{equation}\label{eq:cut-off}
  \eta_r(x):=\eta\left(\frac xr\right)
\end{equation}
with $\eta$ as in \eqref{eq:def_eta}.

\begin{proposition}\label{prop_Ee_limit}
Let $m \in \mathbb{N}$ be as in Proposition
  \ref{prop_v0_asympotic_odd} for $v=v_0$, if $k$ is odd, or as in
  Proposition \ref{prop_v0_asympotic_even}, if $k$ is even.
Then there exists a constant $C_1>0$ such that, for all $\e\in(0,1]$, 
\begin{equation}\label{ineq_Ee}
\mc{E}_\e\le 
\begin{cases}
C_1\,\e^m, &\text{if $k$ is odd},\\
C_1\,\frac{1}{|\log\e|}, &\text{if $k$ is even and } m=0,\\
C_1\,\e^{2m}, &\text{if $k$ is even and } m>0.
\end{cases}
\end{equation}
Moreover, for every $p\in(1,2)$ there exists $C_2=C_2(p)>0$ such that 
  \begin{equation}\label{ineq_Ee-sotto}
\mc{E}_\e\geq 
\begin{cases}
-C_2\, \e^{m-2+\frac{2}{p}}, &\text{if $k$ is odd},\\
-C_2\, \e^{\frac{2}{p}}, &\text{if $k$ is even and } m=0,\\
-C_2\, \e^{2m-2+\frac{2}{p}}, &\text{if $k$ is even and } m>0.
\end{cases}
\end{equation}
In particular, $\mathcal E_\e\to0$ as $\e\to0^+$.
\end{proposition}
\begin{proof}
  If $k$ is odd, let
  $\eta_{\e} \in C^\infty_{\rm c}(\R^2)$ be a cut-off function as in
  \eqref{eq:cut-off} with $r=\e$.  From \eqref{def_Le}, \eqref{def_Je},
  \eqref{minimization_capacity_torsion}, \eqref{eq:defEe}, and
  \eqref{ineq_v_0_pointiweise_bound_odd} it follows that
\begin{multline*}
  J_\e(V_\e) \le J_\e(\eta_{\e} v_0)\le\frac{1}{2}\int_{\Omega
    \setminus \Gamma_\e} |\nabla (\eta_{\e} v_0) |^2 \, dx +
   2\sum_{j=1}^{k_1+k_2} \int_{S^j_\e} |\nabla v_0| \left|v_0\right|\, dS\\
  \le \int_{(\Omega\cap D_{2\e}(0)) \setminus \Gamma_\e} |\nabla
  v_0 |^2 \, dx +\int_{\Omega\cap D_{2\e}(0)} |\nabla \eta_\e
  |^2 v_0^2\, dx +2\sum_{j=1}^{k_1+k_2} \int_{S^j_\e}  |\nabla
  v_0|\left|v_0\right|\, dS\le C_1\e^m
\end{multline*}
for some constant $C_1>0$ independent of $\e$.  If $k$ is even and
$m \in \N\setminus\{0\}$, \eqref{ineq_Ee} can be proved arguing in a
similar way and using \eqref{ineq_v_0_pointiweise_bound_even} instead
of \eqref{ineq_v_0_pointiweise_bound_odd}.

If $k$ is even and $m=0$, for
every $\e\in(0,1]$ we consider the cut-off function
$\omega_{\e} \in W^{1,\infty}(\R^2)$ defined in
\eqref{def_eta_even}.  We have  $0\le\omega_{\e}\le 1$ and, thanks
to \eqref{def_Le}, \eqref{def_Je},
\eqref{minimization_capacity_torsion}, \eqref{eq:defEe}, and
\eqref{ineq_v_0_pointiweise_bound_even} with $m=0$,
\begin{multline*}
  J_\e(V_\e) \le J_\e(\omega_{\e} v_0)\le\frac{1}{2}\int_{\Omega
    \setminus \Gamma_\e} |\nabla (\omega_{\e} v_0) |^2 \, dx +2
  \sum_{j=1}^{k_1+k_2} \int_{S^j_\e} \left|\nabla v_0\right|\left|v_0\right|\, dS\\
  \le \int_{(\Omega\cap D_{\sqrt{\e}}(0)) \setminus \Gamma_\e}
  |\nabla v_0 |^2 \, dx +\int_{\Omega\cap D_{\sqrt{\e}}(0)}
  |\nabla \omega_{\e} |^2 v_0^2\, dx +2\sum_{j=1}^{k_1+k_2}
  \int_{S^j_\e} |\nabla v_0|\left|v_0\right|\, dS \le C_1
  \frac{1}{|\log \e|}
\end{multline*}
for some constant $C_1>0$ independent of $\e$. Estimate
\eqref{ineq_Ee} is thereby proved.

To prove \eqref{ineq_Ee-sotto}, we observe that
\begin{align*}
  \norm{V_\e}_{\mc{H}_1}^2&=\norm{V_\e}_{\mc{H}_\e}^2
=2\mathcal E_\e-2 L_\e(V_\e)
  \le2 \mc{E}_\e +2|L_\e(V_\e)| \\
  &\le 2 \mc{E}_\e +2\norm{L_\e}_{\mathcal
                            H_1^*}\norm{V_\e}_{\mc{H}_1}\le 2 \mc{E}_\e +2\norm{L_\e}_{\mathcal
                            H_1^*}^2+\frac12\norm{V_\e}_{\mc{H}_1}^2,
\end{align*}
and hence
\begin{equation}\label{eq:estoso}
  \mc{E}_\e +\norm{L_\e}_{\mathcal H_1^*}^2\geq \frac14 \norm{V_\e}_{\mc{H}_1}^2\geq0,
\end{equation}
which, together with \eqref{eq_norm_Le}, implies \eqref{ineq_Ee-sotto}.
\end{proof}

\begin{proposition} \label{prop_Ve_limit_0}
We have  $V_\e \to 0$ as $\e \to 0^+$ strongly  in  $\mc{H}_1$.
\end{proposition}
\begin{proof}
    From Proposition \ref{prop_Ee_limit} we have 
    $\lim_{\e\to0^+}\mathcal E_\e=0$, whereas Proposition
    \ref{prop_Le} implies that
    $\lim_{\e\to0^+}\norm{L_\e}_{\mc{H}_1^*}=0$. The conclusion then
    follows from \eqref{eq:estoso}.
\end{proof}

\begin{proposition}\label{prop_Ee_onormVe}
We have $\mc{E}_\e=o\left(\norm{V_\e}_{\mc{H}_\e}\right)$ as  $\e \to 0^+$.
\end{proposition}
\begin{proof}
Proceeding similarly to the previous proof, we have
\[
  \left|\mc{E}_\e \right| \leq \frac{\norm{V_\e}_{\mc{H}_1}^2}{2} +
  \norm{L_\e}_{\mathcal H_1^*}\norm{V_\e}_{\mc{H}_1}
\]
and we can conclude thanks to \eqref{eq_norm_Le} and Proposition \ref{prop_Ve_limit_0}.
\end{proof}

\begin{proposition}\label{prop_Ve_L2norm_oL2norm_nabla}
We have 
\begin{equation*}
\int_{\Omega}V_\e^2 \, dx=o\big(\norm{V_\e}_{\mc{H}_\e}^2\big) \quad \text{as } \e \to 0^+.
\end{equation*}
\end{proposition}
\begin{proof}
  Let us assume by contradiction that there exist a positive constant $C>0$ and
  a sequence $\{\e_n\}_{n \in \N}\subset(0,1)$ such that
  $\lim_{n\to\infty}\e_n=0$ and
\begin{equation}\label{proof_Ve_L2norm_oL2norm_nabla:1}
  \int_{\Omega}V_{\e_n}^2\, dx \ge C\int_{\Omega \setminus
    \Gamma_{\e_n}}|\nabla V_{\e_n}|^2\, dx  \quad\text{for all }n \in \N.
\end{equation}
For every $n \in \N$, we define
$W_n:=\frac{V_{\e_n}}{\norm{V_{\e_n}}_{L^2(\Omega)}}$. Then
$\norm{W_n}_{L^2(\Omega)}=1$ for every $n \in \N$ and
$\{W_n\}_{n \in \N}$ is bounded in $\mc{H}_1$ thanks to
\eqref{proof_Ve_L2norm_oL2norm_nabla:1}.  It follows that there exists
$W \in \mc{H}_1$ such that $W_n \rightharpoonup W$ weakly in
$\mc{H}_1$ as $n \to \infty$, up to a subsequence.  
Since
  $W_n\in \mathcal H_{\e_n}$ for every $n$, from Proposition
  \ref{p:mosco} we deduce that $W\in \mathcal H_0$, while Remark
  \ref{remark_embbeing_compact} ensures that
  \begin{equation}\label{eq:normW1}
\|W\|_{L^2(\Omega)}=1.
\end{equation}
Since
$W_n-\|V_{\e_n}\|_{L^2(\Omega)}^{-1}v_0\in \widetilde{\mathcal
  H}_{\e_n}$, we have $T^j(W_n)=0$ on $\Gamma^j_0$ for all
$j=1,\dots,k_1$, see \eqref{eq:jump}.
By continuity of the trace operator \eqref{eq:trTj},
we deduce that $T^j(W)=0$ on $\Gamma^j_0$ for all $j=1,\dots,k_1$,
hence $W\in \widetilde{\mathcal H}_0$.

Let $w\in \widetilde{\mathcal H}_{0,0}$, where
$\widetilde{\mathcal H}_{0,0}$ is defined in Lemma
\ref{l:approx-00}. For $n$
sufficiently large, $w\in \widetilde{\mathcal H}_{\e_n}$ and  $L_{\e_n}(w)=0$, hence we
can test \eqref{eq_Ve} with $w$, thus obtaining
\begin{equation*}
  \int_{\Omega \setminus \Gamma_1} \nabla W_n \cdot \nabla w\, dx=
  \int_{\Omega \setminus \Gamma_{\e_n}} \nabla W_n \cdot \nabla w\, dx = - L_{\e_n}(w)=0.
\end{equation*}
Letting $n\to\infty$ in the above identity, we obtain 
$\int_{\Omega \setminus \Gamma_0} \nabla W \cdot \nabla w\, dx=0$ for
all $w\in \widetilde{\mathcal H}_{0,0}$ and hence, by the density of
$\widetilde{\mathcal H}_{0,0}$ in $\widetilde{\mathcal H}_0$
established in Lemma \ref{l:approx-00},
\begin{equation}\label{eq:eqW}
  \int_{\Omega \setminus \Gamma_0} \nabla W \cdot \nabla w\,
  dx=0\quad\text{for all }w\in \widetilde{\mathcal H}_0.
\end{equation}
Choosing $w=W$ in \eqref{eq:eqW}, we conclude that $W=0$, thus
contradicting \eqref{eq:normW1}.
\end{proof}

\section{Asymptotic expansion of the eigenvalue variation}\label{sec_asymptotic_expansion}
For every $\e\in[0,1]$, we consider the bilinear form
$q_\e:\mc{\widetilde{H}}_\e\times\mc{\widetilde{H}}_\e\to \R$  defined as 
\begin{equation}\label{def_qe}
q_\e(w_1,w_2):=\int_{\Omega \setminus\Gamma_\e} \nabla w_1 \cdot \nabla w_2 \, dx,
\end{equation}
where $\mc{\widetilde{H}}_\e$ is as in \eqref{def_tilde_H_e}.  To
simplify notation,
we denote by $q_\e$ both the bilinear form defined above and the
associated quadratic form
\begin{equation*}
  q_\e(w)=\int_{\Omega \setminus\Gamma_\e} |\nabla w|^2 \,
  dx=\|w\|_{\mathcal H_\e}^2, \quad w\in\mc{\widetilde{H}}_\e.
\end{equation*}
The following preliminary result can be obtained in a standard way
from the compactness properties pointed out in Remark
\ref{remark_embbeing_compact} and abstract spectral theory, see for
example \cite[Theorems 6.16 and 6.21, Proposition 8.20]{Hspectral}.
\begin{proposition}\label{prop_spectral}
Let $\e\in[0,1]$ and 
$\mc{F}_\e:\mc{\widetilde{H}}_\e\to \mc{\widetilde{H}}_\e$ be the
linear operator defined as
\begin{equation}\label{def_Fe}
q_\e(\mc{F}_\e( w_1),w_2)=\ps{L^2(\Omega)}{w_1}{w_2}.
\end{equation}
Then 
\begin{itemize}
\item[(i)] $\mc{F}_\e$ is symmetric, non-negative and compact; in
  particular $0$ belongs to its spectrum $\sigma(\mc{F}_\e)$.
\item[(ii)]
  $\sigma(\mc{F}_\e)\setminus\{0\}=\{\mu_{n,\e}\}_{n \in \N \setminus
    \{0\}}$, where $\mu_{n,\e}:=1/ \la_{\e,n}$ for every
  $n \in \mathbb{N} \setminus \{0\}$.
\item[(iii)] For every $\mu \in \R$ and $w \in\mc{\widetilde{H}}_\e$,
\begin{equation*}
  \big(\mathop{\rm{dist}}(\mu,\sigma(\mc{F}_\e))\big)^2
  \le \frac{{q_\e}(\mc{F}_\e(w)-\mu w)}{q_\e(w)}.
\end{equation*}
\end{itemize}
\end{proposition}
Letting $n_0\in\N\setminus\{0\}$, $v_0$ and 
$\lambda_0=\lambda_{0,n_0}$ be as in Section
\ref{sec_capacity_torsion}, to prove an asymptotic expansion of
the eigenvalue variation we further assume that 
\begin{equation}\label{eq:simplicity}
  \lambda_0\text{ is simple as an eigenvalue of \eqref{prob_Aharonov-Bohm_0}},
\end{equation}
and, consequently, as an eigenvalue of
\eqref{prob_eigenvlaue_guged0}. Therefore, the continuity result of \cite[Theorem
1.2]{L2015}, see \eqref{limit_lae_la0}, implies that also $\lambda_{\e,n_0}$ is simple for $\e$
sufficiently small. From now on, we denote
\begin{equation*}
  \lambda_\e=\lambda_{\e,n_0}.
\end{equation*}
For $\e$ small, let $v_\e\in \mc{\widetilde{H}}_\e$ be the
  unique eigenfunction of \eqref{prob_eigenvalue_gauged_multipole} associated to the eigenvalue
  $\lambda_\e=~\!\!\lambda_{\e,n_0}$ satisfying
  \begin{equation}\label{eq:choice-of-v-eps}
    \int_\Omega v_\e^2\,dx=1\quad\text{and}\quad \int_\Omega v_\e v_0\,dx>0.
  \end{equation}
We denote as $\Pi_\e$ the projection onto  the one-dimensional space
spanned by $v_{\e}$, i.e.
\begin{align}\label{eq:projection}
\Pi_\e: \ & L^2(\Omega)\to \mc{\widetilde{H}}_\e,\\
\notag &w \mapsto \ps{L^2(\Omega)}{w}{v_{\e}} v_{\e}.
\end{align}
Theorem \ref{t:main1} is contained in the following result, the proof
of which is inspired by \cite[Appendix~A]{AFHL}.

\begin{theorem}\label{prop_eigenvlaues_with_CT}
Under assumption \eqref{eq:simplicity}, the following asymptotic expansion holds:
\begin{equation}\label{eq_asymptotic_eigenvlaues}
  \la_\e -\la_0=2\mc{E}_\e-2L_\e(v_0)
 +o\big(\|V_\e\|^2_{\mc{H}_\e}\big) \quad \text{as  }\e \to 0^+,
\end{equation}
where $V_\e$ is as Proposition \ref{prop_Ve_existence}.
  Furthermore,
  \begin{align}\label{eq:energy-eigenfunctions}
    &    \|v_0-V_\e-\Pi_\e(v_0-V_\e)\|_{\mathcal
      H_\e}=o\left(\|V_\e\|_{\mathcal H_\e}\right)
      \quad \text{as  }\e \to 0^+,\\
    \label{eq:energy-eigenfunctions2} &
                                        \|v_0-\Pi_\e(v_0-V_\e)\|_{L^2(\Omega)}
                                        =o\left(\|V_\e\|_{\mc{H}_\e}\right)
                                        \quad \text{as  }\e \to 0^+,\\
    &\label{eq:energy-eigenfunctions3}\norm{\Pi_\e(v_0-V_\e)}^2_{L^2(\Omega)}
      =1+o\left(\|V_\e\|_{\mc{H}_\e}\right)  \quad \text{as  }\e \to 0^+.
  \end{align}
\end{theorem}
\begin{proof}
  Let $\psi_\e:=v_0-V_\e$.  We recall that we are assuming that $v_0$
  is real-valued and $\|v_0\|_{L^2(\Omega)}=1$.  From \eqref{prob_Ve}
  and \eqref{prob_eigenvlaue_guged0} it follows that
  $\psi_\e \in\mc{\widetilde{H}}_\e$ is a weak solution of
  the problem
\begin{equation*}
\begin{cases}
  -\Delta \psi_\e=\la_{0} v_0,  &\text{in } \Omega \setminus \Gamma_\e,\\
  \psi_\e=0, &\text{on } \partial \Omega,\\
  T^j(\psi_\e)=0, &\text {on }\Gamma^j_\e \text{ for all } j=1,\dots,k_1,\\
  T^j(\nabla \psi_\e\cdot \nu^j)=0, &\text {on }\Gamma^j_\e \text{ for all } j=1,\dots,k_1,\\
  T^j(\psi_\e)=0, &\text {on }S^j_\e \text{ for all } j=k_1+1,\dots, k_1+k_2,\\
  T^j(\nabla\psi_\e\cdot \nu^j)=0, &\text {on }S^j_\e \text{ for all
  } j=k_1+1,\dots,k_1 + k_2,
\end{cases}
\end{equation*}
in the sense that, letting $q_{\e}$ be as in \eqref{def_qe},
\begin{equation}\label{eq_psia}
  q_{\e}(\psi_\e, w)=\la_0\ps{L^2(\Omega)}{v_0}{w}  \quad\text{for all } w
  \in \mc{\widetilde{H}}_\e.
\end{equation}
Let $v_{\e}$ be an eigenfunction of
\eqref{prob_eigenvalue_gauged_multipole} associated to $\la_\e$
chosen as in \eqref{eq:choice-of-v-eps}. Let $\Pi_\e$ be the
projection operator onto the one-dimensional space spanned by $v_{\e}$
defined in \eqref{eq:projection}.  Moreover, we define
\begin{equation}\label{def_hat_va}
\hat v_{\e}:=\frac{\Pi_\e(\psi_{\e})}{\norm{\Pi_\e(\psi_{\e})}_{L^2(\Omega)}}.
\end{equation}
From \eqref{eq_psia} we deduce that 
\begin{equation}\label{proof_eigenvlaues_with_CT:1}
  q_{\e}(\psi_\e, w)-\la_0\ps{L^2(\Omega)}{\psi_\e}{w}=
  \la_0\ps{L^2(\Omega)}{V_\e}{w} \quad\text{for all } w
  \in \mc{\widetilde{H}}_\e.
\end{equation}
Choosing $w=\hat v_{\e}$ in \eqref{proof_eigenvlaues_with_CT:1}, by
\eqref{eq_eigenvalue_gauged_multipole} and \eqref{def_hat_va} we
obtain 
\begin{equation}\label{proof_eigenvlaues_with_CT:2}
  (\la_\e-\la_0)\ps{L^2(\Omega)}{\psi_\e}{\hat v_{\e}}=
  \la_0\ps{L^2(\Omega)}{V_\e}{v_0}+\la_0 \ps{L^2(\Omega)}{V_\e}{\hat v_{\e}-v_0}. 
\end{equation}
We claim that 
\begin{equation}\label{eq_Ve_v0}
\la_0\int_{\Omega}V_\e v_0 \, dx = 2\mc{E}_\e -2L_\e(v_0).
\end{equation}
Indeed, an integration by parts
yields 
\begin{align}\label{proof_eigenvlaues_with_CT:3}
  \int_{\Omega \setminus \Gamma_\e} &\nabla v_0 \cdot \nabla V_{\e}\, dx
                                      -\la_0 \int_{\Omega} v_0 V_\e \, dx\\
  \notag&=\sum_{j=1}^{k_1}\int_{\Gamma^j_0}\Big(-\gamma_+^j(V_\e)\gamma_+^j(\nabla v_0
          \cdot \nu^j)+\gamma_-^j(V_\e)\gamma_-^j(\nabla v_0 \cdot \nu^j)\Big) \,dS \\
                                    &\qquad\notag+\sum_{j=1}^{k_1+k_2}
                                      \int_{S_\e^j}\Big(-\gamma_+^j(V_\e)
                                      \nabla v_0 \cdot \nu^j+\gamma_-^j(V_\e)\nabla v_0 \cdot \nu^j\Big) \,dS\\
                                    &\notag=-2\sum_{j=1}^{k_1+k_2}
                                      \int_{S_\e^j}\gamma_+^j(V_\e)\nabla
                                      v_0
                                      \cdot \nu^j\,dS
                                      +2\sum_{j=1}^{k_1+k_2} \int_{S_\e^j}v_0\nabla v_0 \cdot \nu^j\,dS,
\end{align}
thanks to \eqref{prob_Ve}. Testing \eqref{eq_psia} with $V_{\e}-v_0$
we obtain
\begin{equation*}
  \int_{\Omega \setminus \Gamma_\e}|\nabla(V_{\e}-v_0)|^2=-\la_{0} \int_{\Omega} v_0 (V_{\e}-v_0)\, dx,  
\end{equation*}
and hence, in view of \eqref{eq_eigenvalue_gauged0},
\begin{equation}\label{proof_eigenvlaues_with_CT:5}
  \int_{\Omega \setminus \Gamma_\e}\nabla V_{\e} \cdot \nabla v_0\, dx 
  =\frac{1}{2}\int_{\Omega \setminus \Gamma_\e}|\nabla V_{\e}|^2\, dx
  +\frac{\la_0}{2} \int_{\Omega} v_0 V_{\e}\, dx.  
\end{equation}
Combining \eqref{def_Le}, \eqref{def_Je}, \eqref{eq:defEe},
\eqref{proof_eigenvlaues_with_CT:3} and
\eqref{proof_eigenvlaues_with_CT:5}, we derive \eqref{eq_Ve_v0}.

From \eqref{proof_eigenvlaues_with_CT:2} and \eqref{eq_Ve_v0} we
deduce that, for all $\e\in(0,1)$, 
\begin{equation}\label{proof_eigenvlaues_with_CT:6}
(\la_\e-\la_0)\ps{L^2(\Omega)}{\psi_\e}{\hat v_{\e}}=2\mc{E}_\e -2L_\e(v_0)
+\la_0 \ps{L^2(\Omega)}{V_\e}{\hat v_{\e}-v_0}. 
\end{equation}
Now we study the asymptotics, as $\e \to 0^+$, of each term in
\eqref{proof_eigenvlaues_with_CT:6}. For the sake of clarity, we divide
the rest of the proof into several steps.

\smallskip\noindent\textbf{Step 1.} We claim that  
\begin{equation}\label{eq_step1}
|\la_\e-\la_0|=o\left(\norm{V_\e}_{\mc{H}_\e} \right) \quad \text{as  }\e \to 0^+.
\end{equation}
Letting $\mu_0:=\la_0^{-1}$ and $\mu_\e:=\lambda_\e^{-1}$, since $\la_0$ is simple
and $\la_\e \to \la_0$ by \eqref{limit_lae_la0}, we have 
\begin{equation}\label{proof_eigenvlaues_with_CT:7}
|\la_\e-\la_0| =\lambda_\e \lambda_0 |\mu_\e-\mu_0| \le 2
\lambda_0^2
\mathop{\rm{dist}}(\mu_0,\sigma(\mc{F}_\e))
\le 2 \lambda_0^2 \, \left(\frac{q_\e(\mc{F}_\e(\psi_\e)-\mu_0
    \psi_\e)}{q_\e(\psi_\e)}\right)^{\!1/2},
\end{equation}
where the last inequality is justified by Proposition \ref{prop_spectral}.
Since  $\norm{v_0}_{L^2(\Omega)}=1$, Proposition \ref{prop_Ve_limit_0}
and the Cauchy-Schwarz inequality imply that 
\begin{equation}\label{proof_eigenvlaues_with_CT:8}
q_\e(\psi_\e)=\la_0 +\int_{\Omega \setminus \Gamma_\e} |\nabla V_\e|^2 \, dx
-2\int_{\Omega \setminus \Gamma_\e}\nabla V_\e\cdot \nabla v_0 \, dx= \la_0 +o(1).
\end{equation}
Furthermore,  in view of \eqref{def_Fe} and \eqref{eq_psia}
tested with $\mc{F}_\e(\psi_\e)-\mu_0 \psi_\e$, 
\begin{align*}
  q_\e(\mc{F}_\e(\psi_\e)-\mu_0
  \psi_\e)&=-\ps{L^2(\Omega)}{V_\e}{\mc{F}_\e(\psi_\e)-\mu_0 \psi_\e}
  +\ps{L^2(\Omega)}{v_0}{\mc{F}_\e(\psi_\e)-\mu_0 \psi_\e}\\
  \notag&\quad-q_\e(\mu_0
  \psi_\e,\mc{F}_\e(\psi_\e)-\mu_0 \psi_\e)
  =-\ps{L^2(\Omega)}{V_\e}{\mc{F}_\e(\psi_\e)-\mu_0 \psi_\e}.
\end{align*}
Hence, by Proposition \ref{prop_poincare},
Proposition \ref{prop_Ve_L2norm_oL2norm_nabla} and the Cauchy-Schwarz
inequality we conclude that 
\begin{equation}\label{proof_eigenvlaues_with_CT:10}
  \big(q_\e(\mc{F}_\e(\psi_\e)-\mu_0
  \psi_\e)\big)^{1/2}=o\left(\norm{V_\e}_{\mc{H}_\e} \right)
  \quad\text{as }\e\to0^+.
\end{equation}
Claim \eqref{eq_step1} is proved by combining
\eqref{proof_eigenvlaues_with_CT:7},
\eqref{proof_eigenvlaues_with_CT:8}, and
\eqref{proof_eigenvlaues_with_CT:10}.

\smallskip\noindent\textbf{Step 2.} We claim that  
\begin{equation}\label{eq_step2}
  q_\e(\psi_\e-\Pi_\e\psi_\e)=o\big(\norm{V_\e}_{\mc{H}_\e}^2\big)
  \quad \text{as  }\e \to 0^+.
\end{equation}
Let 
\begin{equation}\label{proof_eigenvlaues_with_CT:11}
\chi_\e:=\psi_\e-\Pi_\e\psi_\e \quad \text{and}  \quad \xi_\e:=\mc{F}_\e(\chi_\e)-\mu_\e\chi_\e.
\end{equation}
By definition we have 
\begin{equation*}
\chi_\e \in N_\e:=\{w \in\mc{\widetilde{H}}_\e:\ps{L^2(\Omega)}{w}{v_{\e}}=0\}
\end{equation*}
and, since $v_{\e}$ is an eigenfunction of
\eqref{prob_eigenvalue_gauged_multipole}, from \eqref{def_Fe} it follows that
$\mc{F}_\e(w) \in N_\e$ for all $w \in  N_\e$. 
Hence the operator 
\begin{equation*}
\widetilde{\mc{F}}_\e:={\mc{F}_\e}\Big|_{N_\e}:N_\e\to N_\e
\end{equation*} 
is well-defined. Furthermore, it is easy to verify that
$\widetilde{\mc{F}}_\e$ satisfies properties (i)-(iii) of
Proposition \ref{prop_spectral} and
$\sigma(\widetilde{\mathcal
  F}_\e)=\sigma(\mc{F}_\e)\setminus\{\mu_\e\}$. In particular, there
exists a constant $K>0$, which does not depends on $\e$, such that
$\big(\mathop{\rm{dist}}(\mu_\e,\sigma(\widetilde{F}_\e))\big)^2 \ge K$. Then, by
\eqref{proof_eigenvlaues_with_CT:11},
\begin{equation}\label{eq:estpsi-pi}
  q_\e(\psi_\e-\Pi_\e\psi_\e)=q(\chi_\e)\le
  \frac{1}{K}\big(\mathop{\rm{dist}}
  (\mu_\e,\sigma(\widetilde{\mathcal F}_\e))\big)^2 q_\e(\chi_\e)
  \le \frac{1}{K}\,q_\e(\widetilde{\mathcal
    F}_\e(\chi_\e)-\mu_\e\chi_\e)
  =\frac{1}{K}\,q_\e(\xi_\e).
\end{equation}
To estimate $q_\e(\xi_\e)$ we use 
\eqref{proof_eigenvlaues_with_CT:1} and
\eqref{eq_eigenvalue_gauged_multipole} tested with $\xi_\e$, thus obtaining
\begin{equation}\label{proof_eigenvlaues_with_CT:16}
q_\e(\chi_\e,\xi_\e)-\la_\e\ps{L^2(\Omega)}{\chi_\e}{\xi_\e}=\la_{0}\ps{L^2(\Omega)}{V_\e}{\xi_\e}
+(\la_0-\la_\e)\ps{L^2(\Omega)}{\psi_\e}{\xi_\e}.
\end{equation}
From \eqref{def_Fe} and \eqref{proof_eigenvlaues_with_CT:16} we
deduce that
\begin{align*}
  q_\e(\xi_\e)&=q_\e(\mc{F}_\e(\chi_\e),\xi_\e)-\mu_\e q_\e(\chi_\e,\xi_\e)
                =-\mu_\e[ q_\e(\chi_\e,\xi_\e)-\la_\e q_\e(\mc{F}_\e(\chi_\e),\xi_\e)]\\
  &=-\frac{\la_{0}}{\la_\e}\ps{L^2(\Omega)}{V_\e}{\xi_\e}-\frac{(\la_0-\la_\e)}{\la_\e}
          \ps{L^2(\Omega)}{\psi_\e}{\xi_\e}.
\end{align*}
From the Cauchy-Schwarz inequality, Proposition \ref{prop_poincare}, and
\eqref{limit_lae_la0} it follows that
\begin{equation}\label{proof_eigenvlaues_with_CT:18}
(q_\e(\xi_\e))^{1/2} 
\le C\left(\norm{V_\e}_{L^2(\Omega)}+|\la_\e-\la_0|\norm{\psi_\e}_{L^2(\Omega)}\right)
\end{equation}
for some constant $C>0$ which does not depend on $\e$. Furthermore,
\eqref{proof_eigenvlaues_with_CT:1} tested with $\psi_\e$,
\eqref{proof_eigenvlaues_with_CT:8}, Proposition \ref{prop_poincare},
and Proposition \ref{prop_Ve_limit_0} yield
\begin{equation*}
  \norm{\psi_\e}_{L^2(\Omega)}^2-1=-\ps{L^2(\Omega)}{V_\e}{\psi_\e}+o(1)=o(1)
  \quad  \text{as } \e \to 0^+.
\end{equation*}
Then \eqref{eq_step2} follows from Proposition
\ref{prop_Ve_L2norm_oL2norm_nabla}, \eqref{eq_step1},
\eqref{eq:estpsi-pi}, and \eqref{proof_eigenvlaues_with_CT:18}.
Estimate \eqref{eq:energy-eigenfunctions} is thereby proved.

\smallskip\noindent\textbf{Step 3.} We claim that 
\begin{equation}\label{eq_step3}
  \norm{v_0-\hat{v}_\e}_{L^2(\Omega)}=o\left(\norm{V_\e}_{\mc{H}_\e}\right)\quad
  \text{ as  }\e \to 0^+.
\end{equation}
By \eqref{def_hat_va}
\begin{equation}\label{proof_eigenvlaues_with_CT:20}
  v_0-\hat{v}_\e=v_0 -\frac{\Pi_\e\psi_\e}{\norm{\Pi_\e\psi_\e}_{L^2(\Omega)}}
  =\frac{1}{\norm{\Pi_\e\psi_\e}_{L^2(\Omega)}}\left(\big(\norm{\Pi_\e\psi_\e}_{L^2(\Omega)}-1
    \big)v_0+v_0-\Pi_\e\psi_\e\right).
\end{equation}
Furthermore, from the definition of $\psi_\e$, Proposition
\ref{prop_Ve_L2norm_oL2norm_nabla}, Proposition \ref{prop_poincare}
and \eqref{eq_step2} it follows that
\begin{equation}\label{proof_eigenvlaues_with_CT:21}
  \norm{v_0-\Pi_\e\psi_\e}_{L^2(\Omega)}\le
  \norm{v_0-\psi_\e}_{L^2(\Omega)}+
  \norm{\psi_\e-\Pi_\e\psi_\e}_{L^2(\Omega)}=o\left(\norm{V_\e}_{\mc{H}_\e}\right)
  \quad \text{as  }\e \to 0^+,
\end{equation}
thus proving \eqref{eq:energy-eigenfunctions2}.
Since $\norm{v_0}_{L^2(\Omega)}=1$,
\eqref{proof_eigenvlaues_with_CT:21} and the Cauchy-Schwarz inequality imply that
\begin{equation}\label{proof_eigenvlaues_with_CT:22}
  \norm{\Pi_\e\psi_\e}^2_{L^2(\Omega)}=
  \norm{v_0-\Pi_\e\psi_\e}^2_{L^2(\Omega)}+
  \norm{v_0}_{L^2(\Omega)}^2-2 \ps{L^2(\Omega)}{v_0-\Pi_\e\psi_\e}{v_0}
  =1+o\left(\norm{V_\e}_{\mc{H}_\e}\right)
\end{equation}
as $\e \to 0^+$, thus proving estimate
  \eqref{eq:energy-eigenfunctions3}.
Combining \eqref{proof_eigenvlaues_with_CT:20},
\eqref{proof_eigenvlaues_with_CT:21} and
\eqref{proof_eigenvlaues_with_CT:22} we obtain \eqref{eq_step3}.

\smallskip\noindent\textbf{Step 4.} We claim that 
\begin{equation}\label{eq_step4}
  \ps{L^2(\Omega)}{\psi_\e}{\hat{v}_\e}=1+o\left(\norm{V_\e}_{\mc{H}_\e}\right)\quad
  \text{ as  }\e \to 0^+.
\end{equation}
Indeed, by \eqref{def_hat_va} we have 
\begin{equation*}
  \ps{L^2(\Omega)}{\psi_\e}{\hat{v}_\e}=\frac{\ps{L^2(\Omega)}{\psi_\e-\Pi_\e\psi_\e}{\Pi_\e\psi_\e}
    +\norm{\Pi_\e\psi_\e}^2_{L^2(\Omega)}}{\norm{\Pi_\e\psi_\e}_{L^2(\Omega)}}.
\end{equation*}
Hence claim \eqref{eq_step4} follows from \eqref{eq_step2} and
\eqref{proof_eigenvlaues_with_CT:22}.

\medskip Putting together Proposition \ref{prop_Ve_L2norm_oL2norm_nabla},
\eqref{proof_eigenvlaues_with_CT:6}, \eqref{eq_step3}, and
\eqref{eq_step4}, we finally obtain
\begin{align*}
  \la_\e-\la_0 
&=\left(1+o\left(\|V_\e\|_{\mc{H}_\e}\right)\right)
    \left(2\mc{E}_\e -2L_\e(v_0)+
    o\big(\|V_\e\|^2_{\mc{H}_\e}\big)\right)\\
  &=
    2\mc{E}_\e -2L_\e(v_0)+ o\big(\|V_\e\|^2_{\mc{H}_\e}\big)+ \big(\mc{E}_\e -L_\e(v_0)\big)
o\big(\|V_\e\|_{\mc{H}_\e}\big)\\
  &=2\mc{E}_\e -2L_\e(v_0)
    +o\big(\|V_\e\|^2_{\mc{H}_\e}\big)
    \quad \text{as  }\e \to 0^+,
\end{align*}
  having used in the last estimate the fact that
  \begin{equation*}
    \mc{E}_\e -L_\e(v_0)=o\big(\|V_\e\|_{\mc{H}_\e}\big)  \quad \text{as  }\e \to 0^+,
  \end{equation*}
due to \eqref{eq_Ve_v0} and  Proposition \ref{prop_Ve_L2norm_oL2norm_nabla}.
Expansion \eqref{eq_asymptotic_eigenvlaues} is thereby proved.
\end{proof}

\section{Blow-up Analysis for \texorpdfstring{$k$}{k} odd}\label{sec_blow_up}
In this section we assume that $k$, and consequently
  $k_1$, are odd and we perform a blow-up analysis for the solution
$V_\e$ of problem \eqref{prob_Ve}.  In order to characterize the
  functional space containing the limit profile, we first need a
Hardy-type inequality,  for the validity of which the assumption
  that $k$ is odd is crucial.

\subsection{A Hardy type inequality for functions jumping on an
    odd number of lines.}
  Let $\widetilde{\mathcal X}$ and $\widetilde{\mc{H}}$ be the
  functional spaces defined in \eqref{eq:defX} and
  \eqref{eq:deftildeH}, respectively.  To prove a Hardy-type
  inequality in $\R^2\setminus D_1$ for functions in
  $\widetilde{\mathcal X}$, we first need the following Hardy
  inequality on annuli for functions jumping on an odd number of
  lines.  For every $r>0$, we define
\begin{equation*}
  \widetilde{\mathcal X}_r:=\{w \in H^1((D_{2r}\setminus D_r)\setminus
  \Gamma_0):
  T^j(w)=0 \text{ on } \Gamma_0^j \text{ for all } j=1,\dots,k_1\}.
\end{equation*}
\begin{lemma}\label{l:hardy-anelli}
Let $k$ and $k_1$ be odd.
There exists a constant $C_H>0$ such that, for every $r>0$ and $w\in
  \widetilde{\mathcal X}_r$,
\begin{equation}\label{eq:hard-annuli-1}
  r^{-2}\int_{D_{2r}\setminus D_r} w^2 \, dx \le C_H
  \int_{(D_{2r}\setminus D_r)\setminus\Gamma_0}|\nabla w|^2 \, dx.
  \end{equation}
  and
  \begin{equation}\label{eq:hard-annuli-2}
    \int_{D_{2r}\setminus D_r} \frac{w^2}{|x|^2} \, dx \le C_H
    \int_{(D_{2r}\setminus D_r)\setminus\Gamma_0}|\nabla w|^2 \, dx.
  \end{equation}
\end{lemma}
\begin{proof}
Inequality \eqref{eq:hard-annuli-2} is a direct consequence of
  \eqref{eq:hard-annuli-1}.

  Let us first prove
  \eqref{eq:hard-annuli-1} for $r=1$.
We argue by contradiction and assume that there exists a sequence
$\{w_n\}_{n \in \mb{N}} \subset \widetilde{\mathcal X}_1$ such that,
for all $n\in\N$, 
\begin{equation}\label{proof_hardy:2}
  \int_{D_{2}\setminus D_{1}} w_n^2 \, dx=1 \quad\text{and}\quad
  \int_{(D_{2}\setminus D_{1})\setminus \Gamma_0)} |\nabla w_n|^2 \, dx < \frac{1}{n}.
\end{equation}
Hence $\{w_n\}_{n \in \mb{N}}$ is bounded in
$\widetilde{\mathcal X}_1$ and, up to a subsequence,
$w_n \rightharpoonup w$ weakly in $\widetilde{\mathcal X}_1$ for some
$w \in \widetilde{\mathcal X}_1$.  From \eqref{proof_hardy:2} and
  weak lower semi-continuity of the $L^2$-norm, we have 
$\nabla w\equiv 0$ in $(D_{2}\setminus D_{1})\setminus \Gamma_0$;
furthermore, reasoning as in Remark \ref{remark_embbeing_compact}, 
the natural embedding of
$H^1((D_{2}\setminus D_{1})\setminus \Gamma_0))$ into
$L^2(D_{2}\setminus D_{1})$ is compact, hence
$\|w\|_{L^2(D_{2}\setminus D_{1})}=1$.  It follows that $w$ is
constant on each connected component of
$(D_{2}\setminus D_{1})\setminus\Gamma_0$ and $w\not\equiv0$.  Since 
$(D_{2}\setminus D_{1})\setminus\Gamma_0$ has $k_1$ connected components
and $k_1$ is odd, a contradiction arises from the
condition $T^j(w)=0$, which is satisfied  on $\Gamma_0^j$ for all $j=1,\dots,k_1$.

For every $r>0$ and $w\in
  \widetilde{\mathcal X}_r$, it is enough to write the proved
  inequality for the scaled function
  $w(rx)$ to obtain \eqref{eq:hard-annuli-1}.
\end{proof}

We draw attention to the fact that the constant $C_H$ in Lemma
  \ref{l:hardy-anelli} does not depend on $r$. Hence, summing over
  annuli that fill $\R^2\setminus D_1$, we obtain the following result.
\begin{proposition}\label{prop_hardy}
  Let $k$ and $k_1$ be odd. Let
  $C_H>0$ be as in Lemma \ref{l:hardy-anelli}. Then, for every
  $w \in \widetilde{\mc{X}}$,
\begin{equation}\label{ineq_hardy}
  \int_{\R^2\setminus D_1}\frac{w^2}{|x|^2} \, dx \le C_H
  \int_{(\R^2\setminus D_1)\setminus \Gamma_1}|\nabla w|^2 \, dx. 
\end{equation}
Furthermore, there exists a constant $C_H'>0$ such that, for all
  $w \in \widetilde{\mc{X}}$,
\begin{equation}\label{ineq_hardy-inside}
  \int_{D_1}w^2\, dx \le C_H'
  \int_{\R^2\setminus \Gamma_1}|\nabla w|^2 \, dx. 
\end{equation}
\end{proposition}
\begin{proof}
If $w \in \widetilde{\mc{X}}$, then $w \in \widetilde{\mathcal X}_r$
for all $r>1$. Hence, by \eqref{eq:hard-annuli-2}, 
\begin{align*}
  \int_{\R^2\setminus D_1}\frac{w^2}{|x|^2} \, dx&=
  \sum_{h=0}^{\infty}\int_{D_{2^{h+1}}\setminus D_{2^h}} \frac{w^2 }{|x|^2}\, dx \\
  &\le C_H \sum_{h=0}^{\infty} \int_{(D_{2^{h+1}}\setminus D_{2^h})\setminus\Gamma_0}|\nabla
  w|^2 \, dx
 = C_H\int_{(\R^2\setminus D_1)\setminus \Gamma_1}|\nabla w|^2 \, dx,
\end{align*}
thus proving \eqref{ineq_hardy}.

By integrating the identity 
  $\mathop{\rm div}(u^2x)=2u\nabla u\cdot x+2u^2$ on each subset of
  $D_1$ obtained by cutting along the lines $\Sigma^j$,
  $j=1 \dots ,k_1+k_2$, and using the Divergence Theorem, we can prove
  that, for all  $w \in \widetilde{\mathcal X}$,
  \begin{equation*}
     \int_{D_1}w^2\, dx \le \int_{\partial D_1}w^2\,dS+ 
  \int_{D_1\setminus \Gamma_1}|\nabla w|^2 \, dx. 
  \end{equation*}
Then, by continuity of the trace operator from $H^1((D_{2}\setminus
D_{1})\setminus \Gamma_0)$ into $L^2(\partial D_1)$ and \eqref{ineq_hardy}, there exists a
positive constant $C>0$ such that
\begin{align*}
  \int_{D_1}w^2\, dx &\le C\left(\int_{D_2\setminus
                       D_1}w^2\,dx+\int_{(D_2\setminus D_1)\setminus \Gamma_1}|\nabla w|^2 \, dx\right)+ 
                       \int_{D_1\setminus \Gamma_1}|\nabla w|^2 \, dx\\
                     &\leq 4C \int_{D_2\setminus
                       D_1}\frac{w^2}{|x|^2}\,dx+(C+1)
                       \int_{\R^2\setminus\Gamma_1}|\nabla w|^2 \, dx
                       \leq (4CC_H+C+1)  \int_{\R^2\setminus \Gamma_1}|\nabla w|^2 \, dx,
\end{align*}
this proving \eqref{ineq_hardy-inside}.
\end{proof}

From Proposition \ref{prop_hardy} it follows that
\begin{equation}\label{def_norm_tilde_X}
\norm{w}_{\widetilde{\mathcal X}}:=\left(\int_{\R^2 \setminus \Gamma_1}|\nabla w|^2 \, dx\right)^{\!\!1/2}
\end{equation}
is a norm on $\widetilde{\mathcal X}$ and $\widetilde{\mathcal X}$ is
a Hilbert space with respect to the corresponding scalar
product.  Proposition \ref{prop_hardy} also ensures that the
  restriction operator
  \begin{equation}\label{eq:restriction}
    \widetilde{\mathcal X}\to H^1(D_\rho\setminus\Gamma_1)
  \end{equation}
  is continuous with respect to the norm defined in
  \eqref{def_norm_tilde_X} for every $\rho>0$. Hence,
for every $p\in~\!\![1,+\infty)$,  the trace operators 
\begin{equation}\label{def_traces_R2}
\gamma_+^j:\widetilde{\mathcal X} \to L^p(S^j_1)
\quad\text{and}
\quad \gamma_-^j: \widetilde{\mathcal X}\to L^p(S^j_1)
\end{equation} 
are well-defined and continuous with respect to the norm
$\|\cdot\|_{\widetilde{\mathcal X}}$. In particular, since
$\widetilde{\mc{H}}\subset \widetilde{\mathcal X}$,
\begin{equation}\label{ineq_sup}
  \sup_{w \in \widetilde{\mc{H}}\setminus\{0\}}
  \frac{\|\gamma_+^j(w)\|^2_{L^p(S^j_1)}}{\norm{w}_{\widetilde{\mc{X}}}^2}
  < +\infty \quad \text{ for every } p \in [1,+\infty)\text{ and }j=1,\dots,k_1+k_2.
\end{equation}
Using \eqref{ineq_hardy}, we prove now that functions in $\widetilde{\mathcal H}$ can be
approximated with functions with compact support. To this aim, we
define
\begin{equation*}
  \widetilde{\mc{H}}_c:=\{w \in \widetilde{\mc{H}}:
  \text{ there exists } r>0 \text{ such that } w\equiv0 \text{ on } \R^2 \setminus D_r\}.	
\end{equation*}
\begin{proposition}\label{prop_Hc_dense}
$\widetilde{\mc{H}}_c$ is dense in $\widetilde{\mc{H}}$.
\end{proposition}
\begin{proof}
  For every $r>1$, let $\eta_r$ be a cut-off function as in
    \eqref{eq:cut-off}. If $w \in \widetilde{\mc{H}}$, it is clear
  that $\{\eta_r w\}_{r>1}\subset \widetilde{\mc{H}}_c$; moreover, by
  \eqref{ineq_hardy} we have 
  $\frac{w}{|x|}\in L^2(\R^2\setminus D_1)$ and hence
\begin{equation*}
  \int_{\R^2 \setminus \Gamma_1} |\nabla \eta_r|^2 w^2 \, dx \le
  16\int_{D_{2r}\setminus D_r}\frac{w^2}{|x|^2} \, dx  \to 0^+\quad \text{as }r \to \infty.
\end{equation*}
This implies that $\nabla(\eta_r w)\to\nabla w$ in
  $L^2(\R^2\setminus\Gamma_1)$ and hence $\eta_r w\to  w$ in
  $\widetilde{\mathcal H}$.
\end{proof}

\subsection{Limit profile for blown-up potentials}

In this subsection, we introduce and characterize the function $\widetilde V$ 
appearing as limit profile in a blow-up analysis for the potentials $V_\e$.

\begin{proposition}\label{prop_Vtilde_existence}
  There exists a unique solution
  $\widetilde V\in \widetilde{\mathcal X}$ to the minimization problem
  \eqref{eq:minJ}. Furthermore, $\widetilde V$ satisfies
\begin{equation}\label{eq_tilde_V}
  \begin{cases}
\widetilde{V}-\eta \Psi_0\in \widetilde{\mathcal H}\\
\displaystyle{\int_{\R^2\setminus \Gamma_1} \nabla \widetilde{V} \cdot \nabla w \, dx
  =
 -2 \sum_{j=1}^{k_1+k_2} \int_{S^j_1}\nabla\Psi_0 \cdot \nu^j \gamma_+^j
 (w)\, dS}
\quad \text{for all  } w \in \widetilde{\mc{H}}.
\end{cases}
\end{equation}
\end{proposition}

\begin{proof}
Since $\nabla \Psi_0\in L^p(S^j_1)$ for all $p\in[1,2)$, by continuity of the trace 
operators in \eqref{def_traces_R2} we have that the  linear functional $L$ defined in
\eqref{def_L} is well-defined and continuous. Then the convex functional $J$ defined in 
\eqref{def_J}  is continuous and coercive on the closed and convex set
  $\eta\Psi_0+ \widetilde{\mc{H}}= \{w \in \widetilde{\mathcal X}: w -\eta\Psi_0 \in 
  \widetilde{\mc{H}}\}$.
  Therefore
  \eqref{eq:minJ} admits a solution $\widetilde{V}$, which
  satisfies \eqref{eq_tilde_V}.
  
  If $\widetilde{V}_1$ and $\widetilde{V}_2$ are solutions of
\eqref{eq_tilde_V}, then we may
take the difference between 
\eqref{eq_tilde_V} for $\widetilde{V}_1$ and \eqref{eq_tilde_V} for
$\widetilde{V}_2$,  both tested with
$\widetilde{V}_1-\widetilde{V}_2\in\widetilde{\mathcal H}$, and conclude that
$\widetilde{V}_1=\widetilde{V}_2$ thanks to \eqref{ineq_hardy}. Hence
$\widetilde{V}$ is the unique solution to \eqref{eq_tilde_V}.
\end{proof}

\subsection{An equivalent characterization of 
\texorpdfstring{$\mathcal E_\e$}{Ee}}\label{sec_characterization_capacity_torsion}
In this subsection, we obtain an equivalent characterization of the
energy $\mc{E}_\e$ introduced in \eqref{eq:defEe}, which will be used
to improve \eqref{ineq_Ee-sotto} and obtain an optimal estimate for
$|\mc{E}_\e|$ in the case $k$ odd.

\begin{proposition}\label{prop_characterization_Ee}
Let $ \eta_\e\in C^\infty_{\rm c}(\R^2)$ be  a cut-off function as in
  \eqref{eq:cut-off} with $r=\e$. Then, for every $\e \in (0,1]$,
\begin{equation}\label{eq_Ee_sup}
\mc{E}_\e=-\frac{1}{2} \sup_{w\in \widetilde{\mc{H}}_\e\setminus\{0\}} 
\frac{\left(\int_{\Omega \setminus \Gamma_\e}\nabla w \cdot \nabla(\eta_\e v_0)+L_\e(w)\right)^2}{\int_{\Omega \setminus \Gamma_\e}|\nabla w|^2 \, dx}
+\frac{1}{2}\int_{\Omega \setminus \Gamma_0}|\nabla (\eta_\e v_0)|^2 \, dx +L_\e(v_0).
\end{equation} 
\end{proposition}
\begin{proof}
Since $\mathcal E_\e$ is the infimum  in
\eqref{minimization_capacity_torsion} and $\varphi-v_0\in
\widetilde{\mathcal H}_\e$ if and only if $\varphi-\eta_\e v_0\in
\widetilde{\mathcal H}_\e$, we have 
\begin{equation}\label{proof_characterization_Ee:1}
  \mc{E}_\e=\inf_{w \in \widetilde{\mc{H}}_\e}J_\e(w+\eta_\e v_0)
  =\inf_{w \in \widetilde{\mc{H}}_\e\setminus\{0\}}
  \bigg(\inf_{t \in [0,+\infty)}J_\e(tw+\eta_\e v_0)\bigg).
\end{equation}
Moreover, by \eqref{def_Je}
\begin{align*}
  J_\e(tw+\eta_\e v_0)&=\frac{t^2}{2}\int_{\Omega \setminus
    \Gamma_\e}|\nabla w|^2 \, dx
  +t\left(\int_{\Omega \setminus \Gamma_\e}\nabla w \cdot
                        \nabla(\eta_\e v_0)\, dx
                        +L_\e(w)\right)\\
&\quad  +\frac{1}{2}\int_{\Omega \setminus \Gamma_0}|\nabla (\eta_\e v_0)|^2
  \, dx +L_\e( v_0).
\end{align*}
Hence, for every $w \in \widetilde{\mc{H}}_\e\setminus\{0\}$, 
\begin{equation*}
\inf_{t \in [0,+\infty)}J_\e(tw+\eta_\e
  v_0)=-\frac12\frac{\big(\int_{\Omega \setminus \Gamma_\e}
    \nabla w \cdot \nabla(\eta_\e v_0) \, dx+L_\e(w)\big)^2}{\int_{\Omega
      \setminus \Gamma_\e}
    |\nabla w|^2 \, dx}+\frac{1}{2}\int_{\Omega \setminus
      \Gamma_0}|\nabla (\eta_\e v_0)|^2 \, dx
    +L_\e(v_0),
\end{equation*}
which implies \eqref{eq_Ee_sup} in view of
\eqref{proof_characterization_Ee:1}.
\end{proof}

\begin{proposition}\label{prop_|Ee|}

Let $k$ and $k_1$ be odd and $m \in \mathbb{N}$ be as in Proposition
  \ref{prop_v0_asympotic_odd} for $v=v_0$. Then 
\begin{equation*}
\mc{E}_\e=O\left(\e^m \right) \quad \text{as } \e \to 0^+. 
\end{equation*}
\end{proposition}
\begin{proof}
  From Proposition \ref{prop_characterization_Ee} and the
  Cauchy-Schwarz inequality it follows that
\begin{align}\label{proof_|Ee|:2}
  |\mc{E}_\e| &\le \frac{1}{2} \sup_{w\in\widetilde{\mc{H}}_\e\setminus\{0\}} 
                \frac{\left(\int_{\Omega \setminus \Gamma_\e}\nabla w \cdot \nabla(\eta_\e v_0)+L_\e(w)\right)^2}{\int_{\Omega \setminus \Gamma_\e}|\nabla w|^2 \, dx}
  +\frac{1}{2}\int_{\Omega \setminus \Gamma_0}|\nabla (\eta_\e v_0)|^2 \, dx +|L_\e(v_0)|\\
 \notag &\le \sup_{w\in\widetilde{\mc{H}}_\e\setminus\{0\}} \frac{|L_\e(w)|^2}{\int_{\Omega \setminus \Gamma_\e}|\nabla w|^2 \, dx} 
  +\frac32\int_{\Omega \setminus \Gamma_0}|\nabla (\eta_\e
  v_0)|^2 \, dx +|L_\e(v_0)|.
\end{align}
From \eqref{ineq_v_0_pointiweise_bound_odd} and \eqref{def_Le} it follows that
\begin{equation}\label{proof_|Ee|:3}
  \int_{\Omega \setminus \Gamma_0}|\nabla (\eta_\e v_0)|^2 \, dx\leq
  2\int_{D_{2\e}}|\nabla \eta_\e|^2 v_0^2 \, dx+2 \int_{D_{2\e}
    \setminus \Gamma_0}  |\nabla v_0|^2 \, dx =O(\e^m)\quad\text{as
  }\e\to0^+
\end{equation}
and
\begin{equation}\label{proof_|Ee|:4}
|L_\e(v_0)| =O(\e^m)\quad\text{as
  }\e\to0^+.
\end{equation}
By \eqref{def_Le}, the H\"older inequality, and
\eqref{ineq_v_0_pointiweise_bound_odd}, for every
$p\in(1,2)$ and $p'=\frac p{p-1}$ we have 
\begin{align}\label{eq:stiLe}
  \sup_{w\in\widetilde{\mc{H}}_\e\setminus\{0\}}
  &\frac{|L_\e(w)|^2}{\int_{\Omega \setminus \Gamma_\e}|\nabla w|^2 \,
    dx} \leq 4(k_1+k_2)
    \sup_{w\in\widetilde{\mc{H}}_\e\setminus\{0\}}\frac{\sum_{j=1}^{k_1+k_2}
    \left(\int_{S^j_\e}|\nabla v_0|  |\gamma_+^j (w)|\,
    dS\right)^2}{\int_{\Omega \setminus \Gamma_\e}|\nabla w|^2 \,
    dx}\\
  \notag&\leq 4(k_1+k_2)\sum_{j=1}^{k_1+k_2}
    \left(\int_{S^j_\e}|\nabla v_0|^{p}\, dS\right)^{\!\!2/p}
    \sup_{w\in \widetilde{\mc{H}}_\e\setminus\{0\}}
    \frac{\|\gamma_+^j (w)\|^2_{L^{p'}(S_\e^j)}}{\int_{\Omega
    \setminus \Gamma_\e}|\nabla w|^2\, dx} \\
  \notag&=O\big(\e^{m-2+\frac2p}\big)\sum_{j=1}^{k_1+k_2}
    \sup_{w\in \widetilde{\mc{H}}_\e\setminus\{0\}}
    \frac{\|\gamma_+^j (w)\|^2_{L^{p'}(S_\e^j)}}{\int_{\Omega
    \setminus \Gamma_\e}|\nabla w|^2\, dx}.
\end{align}
A change of variables and \eqref{ineq_sup} yield 
\begin{equation*}
  \sup_{w\in \widetilde{\mc{H}}_\e\setminus\{0\}}
    \frac{\|\gamma_+^j (w)\|^2_{L^{p'}(S_\e^j)}}{\int_{\Omega
        \setminus \Gamma_\e}|\nabla w|^2\, dx}\leq
      \e^{2/p'}\sup_{v\in \widetilde{\mc{H}}\setminus\{0\}}
    \frac{\|\gamma_+^j
      (v)\|^2_{L^{p'}(S_1^j)}}{\norm{v}_{\widetilde{\mc{X}}}^2}
    =O(\e^{2/p'})\quad\text{as
    }\e\to0^+,
\end{equation*}
hence from  \eqref{eq:stiLe} we deduce that
\begin{equation}\label{proof_|Ee|:5}
  \sup_{w\in\widetilde{\mc{H}}_\e\setminus\{0\}}
  \frac{|L_\e(w)|^2}{\int_{\Omega \setminus \Gamma_\e}|\nabla w|^2 \,
    dx} =O(\e^m)\quad\text{as
    }\e\to0^+.
\end{equation}
The conclusion follows by combining estimates \eqref{proof_|Ee|:2},
\eqref{proof_|Ee|:3}, \eqref{proof_|Ee|:4}, and \eqref{proof_|Ee|:5}.
\end{proof}

\subsection{Blow-up analysis}\label{sec_blow_up_subsection}
Let $k$ and $k_1$ be odd and $m \in \mathbb{N}$ be as in Proposition
\ref{prop_v0_asympotic_odd} for $v=v_0$.
For every  $\e\in (0,1]$, letting $V_\e$ be as Proposition
\ref{prop_Ve_existence}, we define 
\begin{equation}\label{def_tilde_Ve_tilde_v_0e}
  \widetilde{V}_\e(x):=\e^{-\frac{m}{2}}V_\e(\e x) \quad \text{ and }
  \quad \widetilde{V}_{0,\e}(x):=\e^{-\frac{m}{2}}v_0(\e x).
\end{equation}
Extending trivially $\widetilde{V}_\e$ and $\widetilde{V}_{0,\e}$ in
  $\R^2\setminus\Omega$, we have 
$\widetilde{V}_\e, \widetilde{V}_{0,\e} \in \widetilde{\mathcal
  X}$.  Moreover
  \begin{equation}\label{eq:3}
  \widetilde{V}_\e-\widetilde{V}_{0,\e} \in \widetilde{\mathcal
    H}
\end{equation}
and,
by \eqref{eq_Ve} and Proposition \ref{prop_Hc_dense},
\begin{equation}\label{eq_tilde_Ve}
  \int_{\R^2\setminus \Gamma_1} \nabla \widetilde{V}_\e \cdot \nabla w
  \,dx
  = -2 \sum_{j=1}^{k_1+k_2} \int_{S^j_1}\nabla \widetilde{V}_{0,\e} \cdot
  \nu^j \gamma_+^j (w)\, dS
  \quad\text{for all } w \in \widetilde{\mc{H}}.
\end{equation}
Let $\Psi_0$ be as in \eqref{def_Psi}.
From \eqref{limit_v_0_odd} it follows that, for every $j=1,\dots,k_1+k_2$,
  \begin{equation}\label{eq:cd1}
    \nabla \widetilde{V}_{0,\e}(x) \cdot  \nu^j 
    \to  \nabla \Psi_0 (x)\cdot  \nu^j
  \end{equation}
  as $\e\to0^+$ for every $x\in S^j_1$, with
  \begin{align}\label{eq:nablaPsi0}
    &\nabla \Psi_0 (x) \cdot  \nu^j \\
    &=
    \begin{cases}
      \beta\frac m2|x|^{\frac m2-1} f(\alpha^j)
      \cos\left(\tfrac{m}{2}(\alpha^j-\alpha_0)\right),&\text{if
      }j=1,\dots,k_1,\\
      \beta\frac m2|x|^{\frac m2-1} f(\alpha^j)
      \cos\left(\tfrac{m}{2}(\alpha^j-\alpha_0)\right),&\text{if
      }x\in (S^j_1)', \ j=k_1+1,\dots,k_1+k_2,\\
      \notag-\beta\frac m2|x|^{\frac m2-1} f(\alpha^j+\pi)
      \cos\left(\tfrac{m}{2}(\alpha^j+\pi-\alpha_0)\right),&\text{if
      }x\in (S^j_1)'', \ j=k_1+1,\dots,k_1+k_2,
    \end{cases}
  \end{align}
  where, for every $j\in k_1,\dots,k_1+k_2$, 
  \begin{equation*}
    (S^j_1)':=\{t a^j:t\in [0,1]\},\quad (S^j_1)'':=\{t a^{j+k_2}:t\in [0,1]\}.
  \end{equation*}
  On the other hand, \eqref{ineq_v_0_pointiweise_bound_odd} implies
  that
  \begin{equation}\label{eq:cd2}
    |\nabla \widetilde{V}_{0,\e} (x)|\leq C|x|^{\frac m2-1}\quad\text{in }\R^2\setminus\Gamma_0.
  \end{equation}
  From \eqref{eq:cd1} and \eqref{eq:cd2} we deduce that, for every
  $j=1,\dots,k_1+k_2$ and $p\in[1,2)$,
  \begin{equation}\label{eq:conLp}
    \nabla \Psi_0 \cdot  \nu^j\in  L^p(S^j_1)\quad\text{and}\quad
    \nabla \widetilde{V}_{0,\e} \cdot  \nu^j \to  \nabla \Psi_0 \cdot  \nu^j
    \quad\text{in $L^p(S^j_1)$ as $\e\to0^+$}.
  \end{equation}
  Furthermore, by \eqref{eq:convH1} we know that
  \begin{equation}\label{eq:conH1-0}
    \widetilde{V}_{0,\e}\to \Psi_0 \quad\text{in }H^1(D_\rho\setminus\Gamma_0)\text{ for all }\rho>0.
  \end{equation}

\begin{proposition}\label{prop_blow_up}
Let $k$ and $k_1$ be odd and $m \in \mathbb{N}$ be as in Proposition
\ref{prop_v0_asympotic_odd} for $v=v_0$. 
For every  $\e\in (0,1]$, let $V_\e$ be as Proposition
\ref{prop_Ve_existence} and $\widetilde{V}_\e$ as in \eqref{def_tilde_Ve_tilde_v_0e}.
Then     
\begin{equation}\label{limit_blow_up}
\widetilde{V}_\e \to\widetilde{V} \quad \text{strongly in }
\widetilde{\mathcal X} \text{ as } \e \to 0^+,
\end{equation}
where $\widetilde{V} \in \widetilde{\mathcal X}$ is the unique solution to 
 the minimization problem \eqref{eq:minJ} (and then to 
\eqref{eq_tilde_V}, see Proposition \ref{prop_Vtilde_existence}). 
\end{proposition}
\begin{proof}
Taking into account \eqref{ineq_v_0_pointiweise_bound_odd},
\eqref{eq:defEe}, and 
  \eqref{def_tilde_Ve_tilde_v_0e},  a change of variables,
  \eqref{ineq_sup}, the H\"older inequality,
  and Proposition \ref{prop_|Ee|} imply that 
\begin{align}\label{eq:bound-tilde-Ve}
  \|\widetilde{V}_\e&\|_{\widetilde{\mathcal X}}^2=\int_{\R^2\setminus
    \Gamma_1} |\nabla \widetilde{V}_\e|^2 \,
  dx=\e^{-m}\norm{V_\e}_{\mc{H}_\e}^2
  =\e^{-m}(2\mc{E}_\e -2 L_\e(V_\e))\\
  \notag&\le O(1)+4 \e^{-m}\sum_{j=1}^{k_1+k_2} \int_{S^j_\e}|\nabla v_0|
  |\gamma_+^j (V_\e)|\, dS
  = O(1)+O(1) \sum_{j=1}^{k_1+k_2} \int_{S^j_1} |x|^{\frac{m}2-1}|\gamma_+^j (\widetilde{V}_\e)|\, dS\\
  \notag&= O(1) +O(1)\|\widetilde{V}_\e\|_{\widetilde{\mathcal X}}, \text{ as } \e \to 0^+.
\end{align}
Hence $\{\widetilde{V}_\e\}_{\e \in (0,1]}$ is bounded in
$\widetilde{\mathcal X}$. It follows that, for any sequence
$\{\e_n\}_n$ such that $\e_n\to 0$ as $n\to\infty$, there exist a
subsequence, still denoted by $\{\e_n\}_n$, and
$V \in \widetilde{\mathcal X}$ such that
$\widetilde{V}_{\e_n} \rightharpoonup V$ weakly in
$\widetilde{\mathcal X}$ as $n\to\infty$. Therefore, from
\eqref{eq_tilde_Ve}, \eqref{ineq_sup}, and \eqref{eq:conLp} we deduce
that $V$ solves the variational equation in \eqref{eq_tilde_V}.
Furthermore, by \eqref{eq:3} we have
$\widetilde{V}_\e-\eta\widetilde{V}_{0,\e} \in \widetilde{\mathcal
  H}$, hence \eqref{eq:conH1-0} ensures that $V$ satisfies the
condition $V-\eta \Psi_0\in \widetilde{\mathcal H}$. By the uniqueness
part of Proposition \ref{prop_Vtilde_existence} we conclude that
$V=\widetilde V$.

Since $\widetilde{V} -\eta \Psi_0 \in \widetilde{\mc{H}}$, we may test
\eqref{eq_tilde_V} with $\widetilde{V} -\eta \Psi_0$, thus obtaining
\begin{equation}\label{proof_blow_up:2}
\int_{\R^2\setminus \Gamma_1} |\nabla \widetilde{V}|^2 \,dx=
\int_{\R^2\setminus \Gamma_1} \nabla \widetilde{V}\cdot \nabla(\eta
\Psi_0) \, dx
-2 \sum_{j=1}^{k_1+k_2} \int_{S^j_1}\nabla\Psi_0 \cdot \nu^j \gamma_+^j (\widetilde{V}-\eta \Psi_0)\, dS.
\end{equation}
On the other hand, testing \eqref{eq_tilde_Ve} with
$\widetilde{V}_{\e_n}- \eta \widetilde{V}_{0,\e_n} \in \widetilde{\mc{H}}$
we obtain 
\begin{equation}\label{proof_blow_up:3}
  \int_{\R^2\setminus \Gamma_1} |\nabla \widetilde{V}_{\e_n}|^2\, dx
  =\int_{\R^2\setminus \Gamma_1}
  \nabla \widetilde{V}_{\e_n} \cdot \nabla(\eta \widetilde{V}_{0,\e_n}) \,dx 
  -2 \sum_{j=1}^{k_1+k_2} \int_{S^j_1}\nabla \widetilde{V}_{0,\e_n} \cdot
  \nu^j
  \gamma_+^j (\widetilde{V}_{\e_n}-\eta \widetilde{V}_{0,\e_n})\, dS.
\end{equation}
In view of the weak convergence
  $\widetilde{V}_{\e_n} \rightharpoonup \widetilde{V}$ in
  $\widetilde{\mathcal X}$, \eqref{eq:conH1-0}, \eqref{eq:conLp}, and
  the continuity of the trace operators \eqref{def_traces_R2}, the
  limit of the right hand side of \eqref{proof_blow_up:3} as
  $n \to \infty$ is equal to the right hand side of
  \eqref{proof_blow_up:2}, thus proving that
$\widetilde{V}_{\e_n} \to \widetilde{V}$ strongly in
$\widetilde{\mathcal X}$ as $n \to \infty$ by. Since $\widetilde{V}$
is the unique solution of
\eqref{eq_tilde_V}, 
\eqref{limit_blow_up} follows from the Urysohn Subsequence Principle.
\end{proof}

In view of the blow-up analysis performed above, we are in position
to prove Theorem \ref{t:exp-odd}.
\begin{proof}[Proof of Theorem \ref{t:exp-odd}]
    From \eqref{def_Je}, \eqref{eq:defEe}, \eqref{def_tilde_Ve_tilde_v_0e}, and
    a change of variables it follows that 
    \begin{equation}\label{eq:scalEe}
        \e^{-m}\mathcal E_\e= 
\frac{1}{2}\int_{\R^2\setminus{\Gamma_1}} |\nabla \widetilde{V_\e}|^2 \, dx +
2 \sum_{j=1}^{k_1+k_2} \int_{S^j_1}\nabla \widetilde V_{0,\e} \cdot \nu^j
  \gamma_+^j (\widetilde{V_\e})\, dS.        
    \end{equation}
The convergences \eqref{limit_blow_up} and \eqref{eq:conLp}, together with the continuity
of the trace operators in \eqref{def_traces_R2}, allow us to pass to the limit in the
right hand side of \eqref{eq:scalEe}, thus yielding
\begin{equation}\label{eq:asyEe}
    \lim_{\e\to 0^+}\e^{-m}\mathcal E_\e= 
\frac{1}{2}\int_{\R^2\setminus{\Gamma_1}} |\nabla \widetilde{V}|^2 \, dx +
2 \sum_{j=1}^{k_1+k_2} \int_{S^j_1}\nabla \Psi_0 \cdot \nu^j
  \gamma_+^j (\widetilde{V})\, dS=J(\widetilde{V})=\mathcal E    
\end{equation}
and proving claim (i). Furthermore, by \eqref{def_tilde_Ve_tilde_v_0e}, a change of 
variable, \eqref{eq:conLp}, and \eqref{eq:conH1-0}, we have  
\begin{align}\label{eq:asyL0e}
\e^{-m}L_\e(v_0)&=  2 \sum_{j=1}^{k_1+k_2} \int_{S^j_1}\nabla \widetilde V_{0,\e} \cdot \nu^j
  \gamma_+^j (\widetilde V_{0,\e})\, dS  \\
  \notag&=2\sum_{j=1}^{k_1+k_2} \int_{S^j_1}\nabla \Psi_0
  \cdot \nu^j\gamma_+^j (\Psi_0)\, dS+o(1)=L(\Psi_0)+o(1)\quad\text{as }\e\to0^+.
\end{align}
Claim (ii) follows from \eqref{eq_asymptotic_eigenvlaues-intro}, \eqref{eq:asyEe},
\eqref{eq:asyL0e}, and estimate \eqref{eq:bound-tilde-Ve}, which in particular ensures
that $\norm{V_\e}_{\mc{H}_\e}^2=O(\e^{m})$ as $\e\to0^+$.
\end{proof}

\subsection{Continuity of \texorpdfstring{$\mathcal E-L(\Psi_0)$}{E-L}  with respect to rotations of poles}

In this subsection we prove the continuity of $\mathcal E-L(\Psi_0)$
with respect to rotations of the configuration of poles. We fix a
configuration of poles $\{a^j\}$ as in \eqref{def_aj}. Then, for every
$\zeta \in [-\pi,\pi)$, we define $\Psi_0^{(\zeta)}$,
\colg{$L^{(\zeta)}(\Psi_0^{(\zeta)})$ and $\mc{E}^{(\zeta)}$} as in
\eqref{def_Psi}, \eqref{def_L}, and \eqref{def_E_limit}, respectively,
for a rotated configuration of poles $\{a^j_{\zeta}\}$, where
$a^j_{\zeta}$ are defined as in \eqref{def_aj} with angles
$\alpha^j+\zeta$ instead of $\alpha^j$, i.e.
\begin{equation*}
   a^j_{\zeta}=\mathcal R_\zeta (a^j), 
\end{equation*}
being $\mathcal R_\zeta:=R_{0,\zeta}$ with $R_{0,\zeta}$ as in \eqref{def_R_b}, see \Cref{fig:f3}.
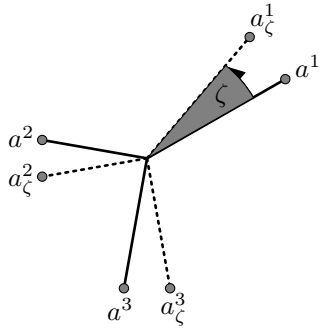
\begin{figure}[ht]
	\centering
	\begin{tikzpicture}[scale=0.7,line cap=round,line join=round,>=triangle 45,x=1.0cm,y=1.0cm]
  \clip(-3.7,-3.7) rectangle (3.5,3);
  \draw [line width=1pt,rotate=30] (0,0)-- (3,0);
\draw [fill=gray,rotate=30] (3,0) circle (2.5pt);
\draw[color=black,rotate=30] (3.5,0) node {$a^{1}$};
\draw [line width=1pt,rotate=170] (0,0)-- (2,0);
\draw [fill=gray,rotate=170] (2,0) circle (2.5pt);
\draw[color=black,rotate=170] (2.4,0) node {$a^{2}$};
\draw [line width=1pt,rotate=260] (0,0)-- (2.5,0);
\draw [fill=gray,rotate=260] (2.5,0) circle (2.5pt);
\draw[color=black,rotate=260] (2.9,0) node {$a^{3}$};
\draw [line width=1pt,dotted,rotate=50] (0,0)-- (3,0);
\draw [fill=gray,rotate=50] (3,0) circle (2.5pt);
\draw[color=black,rotate=50] (3.4,0) node {$a^{1}_\zeta$};
\draw [line width=1pt,dotted,rotate=190] (0,0)-- (2,0);
\draw [fill=gray,rotate=190] (2,0) circle (2.5pt);
\draw[color=black,rotate=190] (2.4,0) node {$a^{2}_\zeta$};
\draw [line width=1pt,dotted,rotate=280] (0,0)-- (2.5,0);
\draw [fill=gray,rotate=280] (2.5,0) circle (2.5pt);
\draw[color=black,rotate=280] (2.9,0) node {$a^{3}_\zeta$};
\draw[line width=0.5pt,->] (2,1.15) arc (30:50:2.2);
\filldraw[line width=0pt,fill opacity=0.2,fill=gray] (0,0) -- (2,1.15) arc (30:51:2.2cm) -- cycle;
\draw[color=black] (1.4,1.2) node {$\zeta$};
\end{tikzpicture}
\caption{The rotated configuration $\{a^j_{\zeta}\}$.}
\label{fig:f3}
\end{figure}

In the next theorem we prove that the function
$\zeta\mapsto \mc{E}^{(\zeta)}-L^{(\zeta)}(\Psi_0^{(\zeta)})$ is
continuous.
\begin{theorem}\label{theorem_continuity}
The function   
$G:[-\pi,\pi)\to \R$,  $G(\zeta):= \mc{E}^{(\zeta)}-L^{(\zeta)}(\Psi_0^{(\zeta)})$ 
is continuous.
\end{theorem}
\begin{proof}
  Through a rotation, the problem of continuity at any
  $\zeta\in[-\pi,\pi)$ can be reduced to the problem of continuity at
  $\zeta=0$. Hence, it is enough to prove that
  $\lim_{\zeta\to0}G(\zeta)=\mathcal E-L(\Psi_0)$.

We have 
    \begin{equation*}
        \Psi_0^{(\zeta)}(r\cos t,r\sin t)=f(t-\zeta)\phi_0(r\cos t,r\sin t),
    \end{equation*}
where $f$ is defined in \eqref{def_f} and 
\begin{equation*}
    \phi_0(r\cos t,r\sin t):=\beta  \, r^{\frac{m}{2}}\, 
\sin\left(\tfrac{m}{2}(t-\alpha_0)\right) .
\end{equation*}
With a slight abuse of notation, henceforth we denote  by $f$ also
the function $(r\cos t,r\sin t)\mapsto f(t)$ defined on 
$\R^2\setminus\{0\}$. 

A change of variables yields 
\begin{equation*}
    \mc{E}^{(\zeta)}=\min\Big\{I_\zeta(w): w\in \widetilde{\mathcal X}\text{ and }w-\eta f (\phi_0\circ\mathcal R_\zeta) \in \widetilde{\mathcal H}\Big\},
\end{equation*}
where $\eta\in C^\infty_{\rm c}(\R^2)$ is a radial cut-off function as
in \eqref{eq:def_eta} and
\begin{equation*}
    I_\zeta(w)=\frac12 \int_{\R^2\setminus\Gamma_1}|\nabla w|^2\,dx+
2 \sum_{j=1}^{k_1+k_2} \int_{S^j_1}f(\nabla \phi_0\circ \mathcal R_\zeta)M_\zeta \cdot \nu^j
  \gamma_+^j (w)\, dS,
    \end{equation*}
    being $M_\zeta$ the matrix defined in \eqref{eq:def_Mzeta}.
Moreover 
\begin{equation*}
    L^{(\zeta)}(\Psi_0^{(\zeta)})=
    2 \sum_{j=1}^{k_1+k_2} \int_{S^j_1}(\nabla \phi_0\circ \mathcal R_\zeta)M_\zeta \cdot \nu^j
  (\phi_0\circ \mathcal R_\zeta)\, dS.
\end{equation*}
Since, in a neighbourhood of $0$, 
\begin{equation}\label{eq:stiphi0}
    |\nabla \phi_0( \mathcal R_\zeta (x))|\leq C|x|^{\frac m2-1}\quad\text{and}\quad 
    |\phi_0( \mathcal R_\zeta (x))|\leq C |x|^{\frac m2}
\end{equation}
for some $C>0$ independent of $\zeta$, from the Dominated Convergence
Theorem we deduce that
\begin{equation}\label{eq:dim-cont1}
    \lim_{\zeta\to0} L^{(\zeta)}(\Psi_0^{(\zeta)})= L(\Psi_0).
\end{equation}
By Proposition \ref{prop_Vtilde_existence}, for every $\zeta$ there
exists a unique $\widetilde V_\zeta\in \widetilde{\mathcal X}$ such
that
$\widetilde{V}_\zeta-\eta f(\phi_0\circ \mathcal R_\zeta)\in
\widetilde{\mathcal H}$ and
$\mathcal E^{(\zeta)}=I_\zeta(\widetilde V_\zeta)$; furthermore,
$\widetilde{V}_\zeta$ satisfies
\begin{equation}\label{eq_tilde_V-zeta}
  \int_{\R^2\setminus \Gamma_1} \nabla \widetilde{V}_\zeta \cdot \nabla w \, dx
  =
  -2 \sum_{j=1}^{k_1+k_2} \int_{S^j_1}f(\nabla\phi_0\circ \mathcal
  R_\zeta)
  M_\zeta \cdot \nu^j \gamma_+^j
  (w)\, dS
  \quad \text{for all  } w \in \widetilde{\mc{H}}.
\end{equation}
Choosing $w=\widetilde{V}_\zeta-\eta f(\phi_0\circ \mathcal R_\zeta)$
in \eqref{eq_tilde_V-zeta} we obtain
\begin{align}\label{eq:vtizetatested}
    \int_{\R^2\setminus \Gamma_1} |\nabla \widetilde{V}_\zeta|^2\,dx
   &=\int_{\R^2\setminus \Gamma_1} \nabla \widetilde{V}_\zeta \cdot
     \nabla
     (\eta f(\phi_0\circ \mathcal R_\zeta))\, dx
   \\
\notag   &\quad-2 \sum_{j=1}^{k_1+k_2} \int_{S^j_1}f(\nabla\phi_0\circ
           \mathcal R_\zeta)M_\zeta
           \cdot \nu^j \gamma_+^j
 (\widetilde{V}_\zeta)\, dS\\
 \notag&\quad
+ 2 \sum_{j=1}^{k_1+k_2} \int_{S^j_1}f(\nabla\phi_0\circ \mathcal
         R_\zeta)M_\zeta \cdot
         \nu^j \gamma_+^j
 (f(\phi_0\circ \mathcal R_\zeta))\, dS.
\end{align}
Using Young's inequality, estimate \eqref{eq:stiphi0}, and the
continuity of the trace operators \eqref{def_traces_R2}, from the
above identity we deduce that
\begin{equation*}
    \|\widetilde{V}_\zeta\|_{\widetilde{\mathcal X}}\leq C
\end{equation*}
for some $C>0$ independent of $\zeta$.  It follows that every sequence
$\zeta_n\to0$ admits a subsequence $\{\zeta_{n_\ell}\}_\ell$ such that
$\widetilde{V}_{\zeta_{n_\ell}}\rightharpoonup W$ weakly in
$\widetilde{\mathcal X}$ as $\ell\to\infty$, for some
$W\in \widetilde{\mathcal X}$.  On account of \eqref{eq:stiphi0} and
\eqref{def_traces_R2}, the Dominated Convergence Theorem yields
\begin{equation*}
    \int_{S^j_1}f(\nabla\phi_0\circ \mathcal
    R_{\zeta_{n_\ell}})M_{\zeta_{n_\ell}}
    \cdot \nu^j \gamma_+^j
 (w)\, dS\to
\int_{S^j_1}f \nabla\phi_0 \cdot \nu^j \gamma_+^j
 (w)\, dS=\int_{S^j_1} \nabla\Psi_0 \cdot \nu^j \gamma_+^j
 (w)\, dS
\end{equation*}
as $\ell\to\infty$, for every $j=1,\dots k_1+k_2$ and
$w\in \widetilde{\mc{H}}$. By choosing $\zeta=\zeta_{n_\ell}$ in
\eqref{eq_tilde_V-zeta} and letting $\ell\to \infty$ we obtain that
\begin{equation}\label{eq_tilde_V-zeta-lim}
\int_{\R^2\setminus \Gamma_1} \nabla W \cdot \nabla w \, dx
  =
 -2 \sum_{j=1}^{k_1+k_2} \int_{S^j_1}\nabla\Psi_0 \cdot \nu^j \gamma_+^j  (w)\, dS
\quad \text{for all  } w \in \widetilde{\mc{H}}.
\end{equation}
Furthermore, since
$\widetilde{V}_\zeta-\eta f(\phi_0\circ \mathcal R_\zeta)\in
\widetilde{\mathcal H}$, $\widetilde{\mathcal H}$ is a closed subspace
of $\widetilde{\mathcal X}$, and
$\eta f(\phi_0\circ \mathcal R_\zeta)\to \eta\Psi_0$ as $\zeta\to0$ in
$\widetilde{\mathcal X}$ by the Dominated Convergence Theorem, we have
\begin{equation}\label{eq:W-etaPhi}
    W-\eta\Psi_0\in \widetilde{\mathcal H}.
\end{equation} 
From \eqref{eq_tilde_V-zeta-lim}-\eqref{eq:W-etaPhi} and the
uniqueness part of Proposition \ref{prop_Vtilde_existence} we deduce
that $W=\widetilde{V}$. Having uniquely identified the weak limit
independently of the subsequence, by the Urysohn subsequence principle
we conclude that
\begin{equation}\label{eq:weakzeta}
    \widetilde{V}_\zeta\rightharpoonup \widetilde{V}\quad\text{weakly
      in }
    \widetilde{\mathcal X} \quad\text{as }\zeta\to0.
\end{equation}
The weak convergence \eqref{eq:weakzeta} allows us to pass to the
limit as $\zeta\to0$ on the right hand side of
\eqref{eq:vtizetatested}, thus proving that
\begin{align}\label{eq:dim-cont2}
    \lim_{\zeta\to0} \int_{\R^2\setminus \Gamma_1} |\nabla \widetilde{V}_\zeta|^2\,dx&=
\int_{\R^2\setminus \Gamma_1} \nabla \widetilde{V} \cdot \nabla (\eta\Psi_0) \, dx
   -2 \sum_{j=1}^{k_1+k_2} \int_{S^j_1}\nabla\Psi_0 \cdot \nu^j \gamma_+^j (\widetilde{V}-\Psi_0)\, dS\\
\notag&=\int_{\R^2\setminus \Gamma_1} |\nabla \widetilde{V}|^2\,dx,
\end{align}
the last equality being a consequence of \eqref{eq_tilde_V} tested
with $w=\widetilde{V}-\eta\Psi_0$. From \eqref{eq:dim-cont2} it
follows that
$\lim_{\zeta \to0}\mathcal{E}^{(\zeta)}= \lim_{\zeta
  \to0}I_\zeta(\widetilde{V}_\zeta)=J(\widetilde{V})=\mathcal E$,
which, together with \eqref{eq:dim-cont1}, yields the conclusion.
\end{proof}

When $k_2=0$ and the poles $\{a^j\}_{j=1,\dots,k_1}$ are on the
tangents to nodal lines of $v_0$ (i.e. on the nodal set of $\Psi_0$),
we have $\Psi_0=0$ on $S^j_1$ for all $j=1,\dots,k_1$; on the other
hand, if the poles are on the bisectors between nodal lines, then
$\nabla\Psi_0\cdot \nu^j=0$ on $S^j_1$ for all $j=1,\dots,k_1$. This
leads to Proposition \ref{p:two-cases}, which determines, in these
particular cases, the sign of the dominant term in the asymptotic
expansion obtained in Theorem \ref{t:exp-odd}, and, consequently,
exploits the continuity result established in Theorem
\ref{theorem_continuity} to find configurations of poles for which the
eigenvalue variation is an infinitesimal of higher order.

\begin{proof}[Proof of Proposition \ref{p:two-cases}]
\begin{enumerate}[\rm (i)]\setlength\itemsep{0.8em}
        \item If  $\alpha^j\in \{\alpha_0+\ell\tfrac{2\pi}{m}:\ell=0,1,2,\dots,m-1\}$
        for all $j\in\{1,\dots,k_1\}$, then $\Psi_0=0$  on $S^j_1$ for all 
        $j\in\{1,\dots,k_1\}$, so that $L(\Psi_0)=0$ and $\eta\Psi_0+\widetilde
        {\mathcal H}=\widetilde {\mathcal H}$. It follows that 
        \begin{equation}\label{eq:231}
            \mathcal E-L(\Psi_0)=\mathcal E=\min_{\widetilde{\mathcal H}}J.
        \end{equation}
        Furthermore, $\nabla \Psi_0 \cdot \nu^j\not\equiv0$ on $S^j_1$
        for all $j\in\{1,\dots,k_1\}$, see \eqref{eq:nablaPsi0}, hence
        $L\not\equiv 0$ in $\widetilde{\mathcal H}$.  Fixing some
        $w\in \widetilde{\mathcal H}$ such that $L(w)\neq0$, we have
        then
        $J(tw)=\frac{t^2}2\int_{\R^2 \setminus \Gamma_1} |\nabla w|^2
        \, dx + tL(w)<0$ for some small $t$, thus implying that
        $\mathcal E=\min_{\widetilde{\mathcal H}}J<0$. Once we have
        established that $\mathcal E-L(\Psi_0)=\mathcal E<0$, from the
        asymptotic expansion of Theorem \ref{t:exp-odd}-(ii) we deduce
        that $\la_{\e, n_0} <\la_{0, n_0}$ for sufficiently small
        $\e>0$.
        \item If  $\alpha^j\in \{\alpha_0+(1+2\ell)
        \tfrac{\pi}{m}:\ell=0,1,2,\dots,m-1\}$ for all $j\in\{1,\dots,k_1\}$, then 
        $\nabla \Psi_0 \cdot \nu^j\equiv0$ on $S^j_1$ for all 
        $j\in\{1,\dots,k_1\}$, see \eqref{eq:nablaPsi0}. It follows that $L\equiv0$,
        and hence $J(w)=\frac12\|w\|_{\widetilde{\mathcal X}}^2$.
        Since, in this case, $\Psi_0\not\equiv0$  on $S^j_1$ for all 
        $j\in\{1,\dots,k_1\}$, we have  $w\not\equiv 0$ for every $w\in 
        \eta\Psi_0+\widetilde {\mathcal H}$. Therefore 
        \begin{equation}\label{eq:232}
       \mathcal E-L(\Psi_0)=\mathcal E=\min_{\eta\Psi_0+
        \widetilde{\mathcal H}}J=\frac12\, \min_{w\in \eta\Psi_0+\widetilde{\mathcal H}}
        \|w\|_{\widetilde{\mathcal X}}^2>0.     
        \end{equation}
        From the asymptotic expansion of Theorem 
        \ref{t:exp-odd}-(ii) we finally deduce that $\la_{\e, n_0} >\la_{0, n_0}$ for 
        sufficiently small $\e>0$. 
      \item Let us fix a configuration $\{a^j\}_{j=1}^k$ with
        $k=k_1\leq m$ odd and
        $\alpha^j\in \{\alpha_0+\ell\tfrac{2\pi}{m}:0\leq \ell\leq
        m-1\}$ for all $j\in\{1,\dots,k_1\}$ as in (i). Then the
        rotated configuration $\{a^j_{\pi/m}\}$ is as in (ii). By
        (i)-(ii) we have $G(0)<0$ and $G(\frac{\pi}{m})>0$. Since $G$
        is continuous by Theorem \ref{theorem_continuity}, Bolzano's
        Theorem ensures the existence of some
        $\zeta_0\in (0,\frac{\pi}{m})$ such that $G(\zeta_0)=0$, so
        that the angles $\{\alpha^j+\zeta_0:j=1,\dots,k\}$ are as we
        are looking for. \qedhere
    \end{enumerate}
\end{proof}

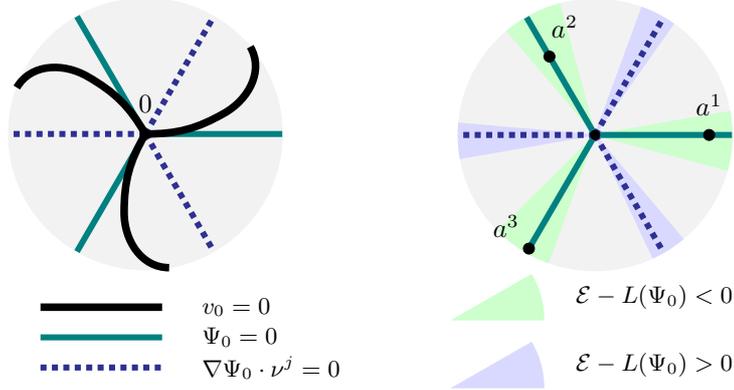
\begin{figure}[ht]

        \subfloat{
        \begin{tikzpicture}[scale=0.6]
        \filldraw [gray!10] (0,0) circle (3cm);

        \draw [line width=0.8mm,teal] (0,0) -- (3,0);
        \draw [line width=0.8mm,teal] (0,0) -- (120:3);
        \draw [line width=0.8mm,teal] (0,0) -- (240:3);
        \draw [line width=0.8mm,Blue,dotted] (0,0) -- (-3,0);
        \draw [line width=0.8mm,Blue,dotted] (0,0) -- (60:3);
        \draw [line width=0.8mm,Blue,dotted] (0,0) -- (300:3);
        \draw [line width=1mm, smooth] (0,0) to [out=120,in=330] (135:1.8) to [out=150,in=60] (160:3);
        \draw [line width=1mm, smooth] (0,0) to [out=0,in=210] (15:1.8) to [out=30,in=-60] (40:3);
        \draw [line width=1mm, smooth] (0,0) to [out=240,in=90] (255:1.8) to [out=270,in=180] (280:3);
        
        \filldraw [black] (0,0) circle (0.1cm) node [yshift=0.4cm] {$0$};

    \end{tikzpicture}}\hspace{2cm}
    \subfloat{
    \begin{tikzpicture}[scale=0.6]
        \filldraw [gray!10] (0,0) circle (3cm);
        
        \filldraw [green!20] (0,0) --  (-15:3) arc(-15:10:3) -- cycle;  
        \filldraw [green!20] (0,0) --  (105:3) arc(105:130:3) -- cycle;
        \filldraw [green!20] (0,0) --  (225:3) arc(225:250:3) -- cycle;
        \filldraw [blue!15] (0,0) --  (55:3) arc(55:70:3) -- cycle;
        \filldraw [blue!15] (0,0) --  (175:3) arc(175:190:3) -- cycle;
        \filldraw [blue!15] (0,0) --  (295:3) arc(295:310:3) -- cycle;

        \draw [line width=0.8mm,teal] (0,0) -- (3,0);
        \draw [line width=0.8mm,teal] (0,0) -- (120:3);
        \draw [line width=0.8mm,teal] (0,0) -- (240:3);
        \draw [line width=0.8mm,Blue,dotted] (0,0) -- (-3,0);
        \draw [line width=0.8mm,Blue,dotted] (0,0) -- (60:3);
        \draw [line width=0.8mm,Blue,dotted] (0,0) -- (300:3);

        \filldraw (0:2.5) circle (0.12cm) node [xshift=0cm, yshift=0.4cm] {\color{black}$a^1$};
        \filldraw (120:2) circle (0.12cm) node [xshift=0.2cm, yshift=0.4cm] {\color{black}$a^2$};
        \filldraw (240:2.9) circle (0.12cm) node [xshift=-0.3cm, yshift=0.3cm] {\color{black}$a^3$};

        \filldraw [black] (0,0) circle (0.1cm);        

    \end{tikzpicture}}

    \hspace{0.5cm}
    \subfloat{
    \begin{tikzpicture}[scale=0.8]
        \draw [line width=1mm] (0,0.5) -- (2,0.5);
        \node [right] at (2.5,0.5) {\small $v_0=0$};
        \draw [line width=0.8mm,teal] (0,0) -- (2,0);
        \node [right] at (2.5,0) {\small$\Psi_0=0$};
        \draw [line width=0.8mm,Blue,dotted] (0,-0.5) -- (2,-0.5);
        \node [right] at (2.5,-0.5) {\small$\nabla\Psi_0\cdot\nu^j=0$};
    \end{tikzpicture}
    }\hspace{1cm}
    \subfloat{
    \begin{tikzpicture}[scale=0.6]
        \filldraw [green!20] (0,0.5) -- (2,0.5) arc(0:30:2) -- cycle;
        \node [right] at (2.5,1) {\small $\mathcal{E}-L(\Psi_0)<0$};
        \filldraw [blue!15] (0,-1) -- (2,-1) arc(0:30:2) -- cycle;
        \node [right] at (2.5,-0.5) {\small  $\mathcal{E}-L(\Psi_0)>0$};
    \end{tikzpicture}
}

    \caption{Nodal set of $\Psi_0$ and sign of $\mathcal{E}-L(\Psi_0)$ ($m=k=3$, $\alpha_0=0$).}
    \label{fig:triple}

\end{figure}

\begin{figure}
    \centering
    
        \noindent\begin{minipage}{\textwidth}
        \begin{minipage}{0.15\textwidth}

        \begin{tikzpicture}[scale=0.5]
        \filldraw [gray!10] (0,0) circle (3cm);
        
        \filldraw [green!20] (0,0) --  (-15:3) arc(-15:10:3) -- cycle;  
        \filldraw [green!20] (0,0) --  (105:3) arc(105:130:3) -- cycle;
        \filldraw [green!20] (0,0) --  (225:3) arc(225:250:3) -- cycle;
        \filldraw [blue!15] (0,0) --  (55:3) arc(55:70:3) -- cycle;
        \filldraw [blue!15] (0,0) --  (175:3) arc(175:190:3) -- cycle;
        \filldraw [blue!15] (0,0) --  (295:3) arc(295:310:3) -- cycle;

        \draw [line width=0.4mm,teal] (0,0) -- (3,0);
        \draw [line width=0.4mm,teal] (0,0) -- (120:3);
        \draw [line width=0.4mm,teal] (0,0) -- (240:3);
        \draw [line width=0.4mm,Blue,dotted] (0,0) -- (-3,0);
        \draw [line width=0.4mm,Blue,dotted] (0,0) -- (60:3);
        \draw [line width=0.4mm,Blue,dotted] (0,0) -- (300:3);

        \draw[line width=0.5mm,Red,-stealth] (2.5,0) -- (1,0);
        \filldraw [Red] (0:2.5) circle (0.12cm) node [above] {\color{black}$a^1$};
        \draw[line width=0.5mm,Red,-stealth] (120:2) -- (120:0.5);
        \filldraw [Red] (120:2) circle (0.12cm) node [xshift=0.1cm,yshift=0.3cm] {\color{black}$a^2$};
        \draw[line width=0.5mm,Red,-stealth] (240:2.9) -- (240:1.4);
        \filldraw [Red] (240:2.9) circle (0.12cm) node [below] {\color{black}$a^3$};

        \filldraw [black] (0,0) circle (0.1cm);

    \end{tikzpicture}
    
    \begin{tikzpicture}[scale=0.5]
                \filldraw [gray!10] (0,0) circle (3cm);
        
        \filldraw [green!20] (0,0) --  (-15:3) arc(-15:10:3) -- cycle;  
        \filldraw [green!20] (0,0) --  (105:3) arc(105:130:3) -- cycle;
        \filldraw [green!20] (0,0) --  (225:3) arc(225:250:3) -- cycle;
        \filldraw [blue!15] (0,0) --  (55:3) arc(55:70:3) -- cycle;
        \filldraw [blue!15] (0,0) --  (175:3) arc(175:190:3) -- cycle;
        \filldraw [blue!15] (0,0) --  (295:3) arc(295:310:3) -- cycle;

        \filldraw[gray,opacity=0.25] (0,0) -- (0:2.5) arc(0:60:2.5) -- cycle;

        \draw [line width=0.4mm,teal] (0,0) -- (3,0);
        \draw [line width=0.4mm,teal] (0,0) -- (120:3);
        \draw [line width=0.4mm,teal] (0,0) -- (240:3);
        \draw [line width=0.4mm,Blue,dotted] (0,0) -- (-3,0);
        \draw [line width=0.4mm,Blue,dotted] (0,0) -- (60:3);
        \draw [line width=0.4mm,Blue,dotted] (0,0) -- (300:3);

        \draw[line width=0.5mm,blue,-stealth] (60:2.5) -- (60:1);
        \filldraw [blue] (60:2.5) circle (0.12cm) node [xshift=-0.5cm,yshift=0cm] {\color{black}$a^1_{\frac{\pi}{m}}$};
        \draw[line width=0.5mm,blue,-stealth] (180:2) -- (180:0.5);
        \filldraw [blue] (180:2) circle (0.12cm) node [above] {\color{black}$a^2_{\frac{\pi}{m}}$};
        \draw[line width=0.5mm,blue,-stealth] (300:2.9) -- (300:1.4);
        \filldraw [blue] (300:2.9) circle (0.12cm) node [below] {\color{black}$a^3_{\frac{\pi}{m}}$};

        \draw[line width=0.2mm,-stealth] (0:2.5) arc(0:60:2.5)  node [xshift=0.2cm,yshift=-0.6cm] {\footnotesize$\frac{\pi}{m}$};

        \filldraw [black] (0,0) circle (0.1cm);

    \end{tikzpicture}
    \end{minipage}\hfill
    \begin{minipage}{0.15\textwidth}
            \begin{tikzpicture}[scale=0.5]
                \filldraw [gray!10] (0,0) circle (3cm);
        
        \filldraw [green!20] (0,0) --  (-15:3) arc(-15:10:3) -- cycle;  
        \filldraw [green!20] (0,0) --  (105:3) arc(105:130:3) -- cycle;
        \filldraw [green!20] (0,0) --  (225:3) arc(225:250:3) -- cycle;
        \filldraw [blue!15] (0,0) --  (55:3) arc(55:70:3) -- cycle;
        \filldraw [blue!15] (0,0) --  (175:3) arc(175:190:3) -- cycle;
        \filldraw [blue!15] (0,0) --  (295:3) arc(295:310:3) -- cycle;

        \filldraw[gray,opacity=0.25] (0,0) -- (0:2.5) arc(0:30:2.5) -- cycle;
        
        \draw [line width=0.4mm,teal] (0,0) -- (3,0);
        \draw [line width=0.4mm,teal] (0,0) -- (120:3);
        \draw [line width=0.4mm,teal] (0,0) -- (240:3);
        \draw [line width=0.4mm,Blue,dotted] (0,0) -- (-3,0);
        \draw [line width=0.4mm,Blue,dotted] (0,0) -- (60:3);
        \draw [line width=0.4mm,Blue,dotted] (0,0) -- (300:3);

        \draw[line width=0.5mm,Yellow,-stealth] (30:2.5) -- (30:1);
        \filldraw [Yellow] (30:2.5) circle (0.12cm) node [xshift=0cm,yshift=0.4cm] {\color{black}$a^1_{\zeta_0}$};
        \draw[line width=0.5mm,Yellow,-stealth] (150:2) -- (150:0.5);
        \filldraw [Yellow] (150:2) circle (0.12cm) node [above] {\color{black}$a^2_{\zeta_0}$};
        \draw[line width=0.5mm,Yellow,-stealth] (270:2.9) -- (270:1.4);
        \filldraw [Yellow] (270:2.9) circle (0.12cm) node [below] {\color{black}$a^3_{\zeta_0}$};

        \draw[line width=0.2mm,-stealth] (0:2.5) arc(0:30:2.5) node [xshift=-0.2cm,yshift=-0.4cm] {\footnotesize$\zeta_0$};

        \filldraw [black] (0,0) circle (0.1cm);

    \end{tikzpicture}
    \end{minipage}\hfill
    \begin{minipage}{0.55\textwidth}
    \begin{tikzpicture}
        \begin{axis}[
        xmin=0,
        xmax=1.2,
        axis x line=bottom,
        axis y line=middle,
        xlabel=$\e$,
        xtick=\empty,
        xticklabel=\empty,
        ytick={2},
        yticklabels={\scriptsize$\lambda_{0,n_0}$},
        legend pos=south west,
        legend entries={\scriptsize$\mathcal{E}-L(\Psi_0)<0$,\scriptsize$\mathcal{E}-L(\Psi_0)>0$,\scriptsize$\mathcal{E}-L(\Psi_0)=0$}
        ]
            \addplot [domain=0:1, smooth,very thick, Red] { 2-2.5*x^3};
            \addplot [domain=0:1, smooth, very thick, blue] { 2+1.8*x^3};
            \addplot[domain=0:1, smooth,very thick, Yellow]{2+x^7*sin(50*deg(x)};
            \addplot [name path=A,domain=0:1, smooth,very thick, Yellow,opacity=0.2] { 2-x^7};
            \addplot [name path=B,domain=0:1, smooth,very thick, Yellow,opacity=0.2] { 2+x^7};
            \addplot[Yellow,opacity=0.2] fill between[of=A and B];

            \addplot [domain=0:1,smooth, thick, dashed] {2};  
            \node at (1.1,-0.2) {\color{Red}$\lambda_{\e,n_0}$};
            \node at (1.1,3.5) {\color{blue}$\lambda_{\e,n_0}$};
            \node at (1.1,1) {\color{Yellow}$\lambda_{\e,n_0}$};
        \end{axis}
    \end{tikzpicture}
    \end{minipage}%
\end{minipage}
    \caption{A visualization of Proposition 2.3.}
    \label{fig:3_cases}
\end{figure}
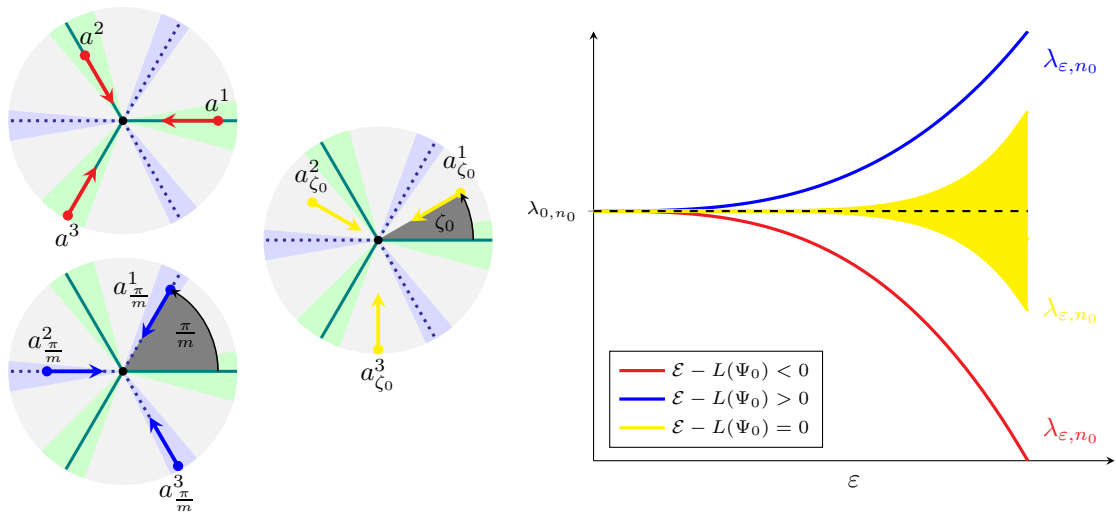

\begin{remark}\label{rmk:figures}
  \Cref{fig:triple} and \Cref{fig:3_cases} provide an example that
  helps to better visualize the result in \Cref{p:two-cases}. In
  \Cref{fig:triple} we zoom in near a point (the origin) where the
  limit eigenfunction $v_0$ vanishes of order $3/2$, namely
  \eqref{asy-v0} holds with $m=3$. We consider the case
  $\alpha_0=0$. The function $\Psi_0$ as in \eqref{def_Psi} is the
  $3/2$-homogeneous limit profile describing the local behavior of
  $v_0$. In the image on the left, the black lines are the nodal lines
  of $v_0$, which are tangent to the nodal lines of $\Psi_0$ (in
  green). The dotted lines denote the bisectors of the nodal lines of
  $\Psi_0$. In the image on the right, we fix an admissible
  configuration of poles $\{a^j\}_{j=1,2,3}$ with $k=3$ and
  $\alpha_j=2\pi(j-1)/3$ for $j=1,2,3$.  From \Cref{p:two-cases} we
  know that, if all the poles lie on the nodal set of $\Psi_0$, then
  the coefficient $\mathcal{E}-L(\Psi_0)$ of the leading term in the
  asymptotic expansion stated in \Cref{t:exp-odd} is strictly
  negative. On the other hand, if all the poles lie on the bisectors
  of the nodal lines, then the coefficient $\mathcal{E}-L(\Psi_0)$ is
  strictly positive. In \Cref{fig:3_cases} on the left, in the first
  picture (red arrows) we have our initial fixed configuration, which
  then provides a negative coefficient. In the second picture (blue
  arrows) we consider a rotation about the origin by an angle
  $\pi/m=\pi/3$: the rotated configuration ends up with all the poles
  lying on the bisectors, thus giving a positive coefficient
  $\mathcal{E}-L(\Psi_0)$. Furthermore, the continuity result in
  \Cref{theorem_continuity} ensures the existence of some
  $\zeta_0\in(0,\pi/3)$ such that, if we rotate the initial
  configuration by an angle $\zeta_0$, we find a configuration of
  poles for which $\mathcal{E}-L(\Psi_0)=0$: this is represented in
  the third picture on the left (yellow arrows). Finally, the right
  picture in \Cref{fig:3_cases} presents the behavior of the perturbed
  eigenvalue in the three cases previously described. We point out
  that, when $\mathcal{E}-L(\Psi_0)=0$ (yellow graph), it is currently
  not known what is the vanishing order of
  $\lambda_{\e,n_0}-\lambda_{0,n_0}$.
\end{remark}

\subsection{Blow-up and convergence rate for eigenfunctions}\label{sec:eigenfunctions}
From the blow-up analysis for the potential $V_\e$ performed
  in Subsection \ref{sec_blow_up_subsection} and the energy estimate given in
  \eqref{eq:energy-eigenfunctions}, we derive the following blow-up
  result for scaled eigenfunctions, together with a sharp estimate for
  their rate of convergence in the $\mathcal H_1$-norm.

\begin{proposition}\label{p:blow_up_eigenfunctions}
Under assumption \eqref{eq:simplicity}, let $k$ be odd and $v_0$ be an eigenfunction
of \eqref{prob_eigenvlaue_guged0} associated to the eigenvalue
$\lambda_0=\lambda_{0,n_0}$ with $\norm{v_0}_{L^2(\Omega)}=1$. For $\e>0$ small, let
$\la_\e=\lambda_{\e,n_0}$ and $v_{\e}$ be an eigenfunction of
\eqref{prob_eigenvalue_gauged_multipole} associated to $\la_\e$
and chosen as in \eqref{eq:choice-of-v-eps}.
Let $m \in \mathbb{N}$ be given in Proposition
\ref{prop_v0_asympotic_odd} for $v=v_0$. Then
\begin{equation}\label{eq:asy-veps}
  \e^{-\frac{m}{2}}v_\e(\e x)\to \Psi_0-\widetilde V\quad\text{as
    $\e\to0^+$ in
  }H^1(D_\rho\setminus\Gamma_1)\text{ for all }\rho>0,  
\end{equation}
where $\Psi_0$ is defined in \eqref{def_Psi} and $\widetilde V$ is the
unique solution to \eqref{eq_tilde_V}. Furthermore,
\begin{equation}\label{eq:rate-eig}
  \lim_{\e\to 0^+}\e^{-\frac{m}{2}}\|v_\e-v_0\|_{\mathcal H_1}= \|\widetilde
  V\|_{\widetilde{\mathcal X}}.
\end{equation}
\end{proposition}
\begin{proof}
  Using the same notation as in the proof of Theorem
  \ref{prop_eigenvlaues_with_CT}, let $\psi_\e=v_0-V_\e$, where $V_\e$
  is defined as in Proposition \ref{prop_Ve_existence}.
  From \eqref{eq:energy-eigenfunctions} it follows that
  \begin{equation*}
    \|\Pi_\e\psi_\e-\psi_\e\|_{\mathcal
      H_\e}^2=o\left(\|V_\e\|_{\mathcal H_\e}^2\right) \quad\text{as
    }\e\to0^+.
  \end{equation*}
  Therefore, defining
  \begin{equation*}
    W_\e(x):=\e^{-\frac{m}{2}}(\Pi_\e\psi_\e-\psi_\e)(\e x), \quad
    x\in \tfrac1\e\Omega,
  \end{equation*}
  and extending trivially $W_\e$ in $\R^2\setminus\frac1\e\Omega$, we
  have  $W_\e\in \widetilde{\mathcal H}$ and, in view of
  Proposition \ref{prop_blow_up},
  \begin{equation*}
    \|W_\e\|_{\widetilde{\mathcal X}}^2=\e^{-m}\|\Pi_\e\psi_\e-\psi_\e\|_{\mathcal H_\e}^2=
    \e^{-m}o\left(\|V_\e\|_{\mathcal
    H_\e}^2\right)=\|\widetilde{V_\e}\|_{\widetilde{\mathcal X}}^2\,o(1)=o(1)
\end{equation*}
  as $\e\to0^+$.
By continuity of the restriction operator in \eqref{eq:restriction} we
deduce that
\begin{equation}\label{eq:con-W-eps}
  W_\e\to 0 \quad\text{as $\e\to0^+$ in
  $H^1(D_\rho\setminus\Gamma_1)$ for all $\rho>0$}. 
\end{equation}
Let us define
\begin{equation}\label{eq:Ueps}
    U_\e(x):=\e^{-\frac{m}{2}}(\Pi_\e\psi_\e)(\e x), \quad
    x\in \tfrac1\e\Omega,
  \end{equation}
and extend trivially $U_\e$ in $\R^2\setminus\frac1\e\Omega$. We have
\begin{equation*}
  U_\e=\widetilde{V}_{0,\e}(x)-\widetilde V_\e+W_\e,
\end{equation*}
where $\widetilde{V}_{0,\e}$ and $\widetilde V_\e$ are defined in
\eqref{def_tilde_Ve_tilde_v_0e}.
Combining \eqref{eq:conH1-0}, \eqref{limit_blow_up}, and
\eqref{eq:con-W-eps}, we conclude that
\begin{equation}\label{eq:asyUeps}
  U_\e\to \Psi_0-\widetilde V\quad\text{as
    $\e\to0^+$ in
  }H^1(D_\rho\setminus\Gamma_1)\text{ for all }\rho>0. 
\end{equation}
From \eqref{eq:energy-eigenfunctions2} it follows that
\begin{equation*}
  \int_\Omega
  v_0\Pi_\e\psi_\e\,dx=1+o\left(\|V_\e\|_{\mc{H}_\e}\right)\quad\text{as
  }\e\to0^+,
\end{equation*}
and hence, for $\e>0$ small enough,
\begin{equation*}
  \int_\Omega \frac{\Pi_\e\psi_\e}{\|\Pi_\e\psi_\e\|_{L^2(\Omega)}}\,v_0\,dx>0.
\end{equation*}
Since $v_{\e}$ is the unique eigenfunction of
\eqref{prob_eigenvalue_gauged_multipole} associated to $\la_\e$
satisfying \eqref{eq:choice-of-v-eps}, we conclude that necessarily
\begin{equation}\label{eq:chara-veps}
  v_\e=\frac{\Pi_\e\psi_\e}{\|\Pi_\e\psi_\e\|_{L^2(\Omega)}}.
\end{equation}
The convergence stated in \eqref{eq:asy-veps} follows from
\eqref{eq:chara-veps}, \eqref{eq:Ueps}, \eqref{eq:asyUeps}, and
\eqref{eq:energy-eigenfunctions3}.

Moreover, \eqref{eq:energy-eigenfunctions3} implies that
\begin{align}\label{eq:eig1}
\|v_\e-\Pi_\e\psi_\e\|_{\mathcal
  H_1}&=\frac{|1-\|\Pi_\e\psi_\e\|_{L^2(\Omega)}|}
{\|\Pi_\e\psi_\e\|_{L^2(\Omega)}}\|\Pi_\e\psi_\e\|_{\mathcal
  H_1}\\
&\notag  =\big|1-\|\Pi_\e\psi_\e\|_{L^2(\Omega)}\big| \|v_\e\|_{\mathcal
  H_\e}    
  =o(\|V_\e\|_{\mathcal H_\e})\quad\text{as }\e\to0^+,
\end{align}
whereas \eqref{eq:energy-eigenfunctions} yields that
\begin{equation}\label{eq:eig2}
  \|\Pi_\e\psi_\e-v_0\|_{\mathcal
  H_1}^2=\|V_\e\|_{\mathcal
  H_\e}^2+\|\Pi_\e\psi_\e-\psi_\e\|_{\mathcal H_\e}^2-2(V_\e,
\Pi_\e\psi_\e-\psi_\e)_{\mathcal H_\e}=\|V_\e\|_{\mathcal
  H_\e}^2+o( \|V_\e\|_{\mathcal
  H_\e}^2) 
\end{equation}
as $\e\to0^+$. Combining \eqref{eq:eig1} and \eqref{eq:eig2} we deduce
that
\begin{equation}\label{eq:eig3}
  \|v_\e-v_0\|_{\mathcal H_1}^2=\|V_\e\|_{\mathcal
  H_\e}^2+o( \|V_\e\|_{\mathcal
  H_\e}^2) \quad\text{as }\e\to0^+.
\end{equation}
Letting $\widetilde{V}_\e$ be as in \eqref{def_tilde_Ve_tilde_v_0e},
from \eqref{eq:eig3} and \eqref{limit_blow_up} we deduce that 
\begin{equation*}
  \e^{-m}\|v_\e-v_0\|_{\mathcal H_1}^2=\|\widetilde V_\e\|^2_{\widetilde{\mathcal
  X}}(1+o(1))=\|\widetilde V\|^2_{\widetilde{\mathcal
  X}}(1+o(1)) \quad\text{as }\e\to0^+,
\end{equation*}
  thus proving \eqref{eq:rate-eig}.
\end{proof}
Going back to the eigenfunctions of the original magnetic problem via 
the inverse of transformation \eqref{eq:ph.transf}, we deduce 
Theorem \ref{t:autofunzioniAB} from Proposition \ref{p:blow_up_eigenfunctions}.

\begin{proof}[Proof of Theorem \ref{t:autofunzioniAB}] 
    If $u_0$ is an eigenfunction of \eqref{prob_Aharonov-Bohm_0} associated 
    to the eigenvalue $\lambda_{0,n_0}$ such that $\int_\Omega |u_0|^2\,dx=1$, 
    and $u_\e$ is the eigenfunction of
    \eqref{prob_Aharonov-Bohm_multipole} associated to 
    $\lambda_{n_0,\e}$ satisfying \eqref{eq:ipo-auto}, then 
    $v_\e:=e^{-i\Theta_\e}u_\e$ is an eigenfunction of 
    \eqref{prob_eigenvalue_gauged_multipole} associated to $\lambda_{n_0,\e}$ and 
    $v_0:=e^{-i\Theta_0}u_0$ is an eigenfunction of 
    \eqref{prob_eigenvlaue_guged0} associated to $\lambda_{n_0,0}$ such that
    condition \eqref{eq:choice-of-v-eps} is satisfied.  From Proposition 
    \ref{p:blow_up_eigenfunctions} it follows that  $v_\e$ satisfies 
    \eqref{eq:asy-veps} and \eqref{eq:rate-eig}, in which we replace $v_\e$ with 
    $e^{-i\Theta_\e}u_\e$ and $v_0$ with $e^{-i\Theta_0}u_0$ to get exactly 
    \eqref{eq:asy-ue1} and \eqref{eq:asy-ue2}, taking into account that 
    $\Theta_\e(\e x)=\Theta_1(x)$ for all $x\in \R^2\setminus
    \{a^j:j=1,\dots,k\}$.    
\end{proof}

\section{The case of two poles}\label{sec_two_poles}
The purpose of this section is to prove Theorems
\ref{theorem_2poles_nodal} and \ref{theorem_2poles_bisector}. We
consider the case $k_1=0$ and $k_2=1$, with the configuration of poles
as in assumption \eqref{eq:due-poli}, being $r_1\in(0,R)$ and
$\e\in(0,1]$.  For the sake of simplicity, let us denote
\begin{equation*}
	T:=T^1, \quad  	\gamma_+:=\gamma_+^1,
 \quad  	\gamma_-:=\gamma_-^1,\quad \text{and }  \quad \nu:=\nu^1=(0,1),
\end{equation*} 
see \eqref{eq:trTj}.  We first consider a linear functional
$L_{\e,h,\Lambda}$ more general than the one introduced in
\eqref{def_Le}, defined for a generic domain $\Lambda$ and with the
limit eigenfunction $v_0$ replaced by a generic function $h$; the
corresponding minimal energy $\mathcal E_{\e,h,\Lambda}$ thus
generalizes the energy $\mc{E}_\e$ defined in \eqref{eq:defEe}.  For
every simply connected open bounded domain $\Lambda\subset \R^2$ such
that $D_R\subseteq\Lambda$ and every
$h \in H^1_0(\Lambda)\cap C^{\infty}(\Lambda)$, let
\begin{equation*}
	L_{\e,h,\Lambda}: \mc{H}_{1,\Lambda} \to\R,\quad
        L_{\e,h,\Lambda}(w):= 2
        \int_{S_\e}\pd{h}{x_2}	\gamma_+ (w)\, dS
\end{equation*}
and 
\begin{equation*}
	J_{\e,h,\Lambda}: \mc{H}_{\e,\Lambda}\to\R,\quad
        J_{\e,h,\Lambda}(w):=\frac{1}{2}
        \int_{\Lambda \setminus S_\e} |\nabla w|^2 \, dx +L_{\e,h,\Lambda}(w),
\end{equation*}
where, for all $\e\in(0,1]$, $S_\e$ is defined in \eqref{def_S_e} and
the functional space $\mathcal H_{\e,\Lambda}$ is the closure of
\begin{equation*}
    \left\{w \in H^1(\Lambda \setminus S_\e):w=0 \text{ on a
    neighbourhood of } \partial \Lambda\right\}
\end{equation*}
with respect to the norm $\norm{w}_{H^1(\Omega \setminus
  S_\e)}$.
Then the minimization problem 
\begin{equation}\label{eq:Vepsin_2poles}
  \inf\left\{J_{\e,h,\Lambda}(w): w\in \mathcal H_{\e,\Lambda}
    \text{ and }w- h  \in  \widetilde{\mc{H}}_{\e,\Lambda} \right\}
\end{equation} 
with
$\widetilde{\mc{H}}_{\e,\Lambda}:= \left\{w \in \mc{H}_{\e,\Lambda}:
  T(w)=0\text { on } S_\e \right\}$, is uniquely achieved, as stated
in the following proposition. We omit the proof, being similar to the
one of Proposition \ref{prop_Ve_existence}.

\begin{proposition}\label{prop_Ve_existence_2poles}
  The infimum in \eqref{eq:Vepsin_2poles} is achieved by a unique
  $V_{\e,h,\Lambda} \in \mathcal H_{\e,\Lambda}$.  Furthermore,
  $V_{\e,h,\Lambda}$ weakly solves the problem
	\begin{equation}\label{prob_Ve_2poles}
		\begin{cases}
			-\Delta V_{\e,h,\Lambda}=0,  &\text{in } \Lambda \setminus S_\e,\\
			V_{\e,h,\Lambda}=0, &\text{on } \partial \Lambda,\\
			T(V_{\e,h,\Lambda}-h)=0, &\text {on }S_\e,\\
			T\left(\pd{V_{\e,h,\Lambda}}{x_2}-\pd{h}{x_2}\right)=0, &\text {on } S_\e,
		\end{cases}
	\end{equation}
	in the sense that $V_{\e,h,\Lambda} \in  \mc{H}_{\e,\Lambda}$, $V_{\e,h,\Lambda} -h \in  \widetilde{\mc{H}}_{\e,\Lambda}$,   and 
	\begin{equation}\label{eq_Ve_2poles}
		\int_{\Lambda \setminus S_\e} \nabla V_{\e,h,\Lambda}\cdot \nabla w\, dx = -L_{\e,h,\Lambda}(w)\quad\text{for all }w \in \mc{\widetilde{H}}_{\e,\Lambda}.
	\end{equation}
\end{proposition}

For every $\Lambda,h$ as above and $\e \in(0,1]$, let 
\begin{equation}\label{def_E_he}
	\mc{E}_{\e,h,\Lambda}:=J_{\e,h,\Lambda}(V_{\e,h,\Lambda}).
\end{equation}
For every $L>0$ and $\e >0$, let $E_\e(L)$ be the ellipse defined as
\begin{equation*}
	E_\e(L):=\left\{(x_1,x_2)\in \R^2:\frac{x_1^2}{L^2+r_1^2\e^2}+\frac{x^2_2}{L^2}<1\right\}.
\end{equation*}
We are going to compute $\mc{E}_{\e,P_m,E_\e(L)}$, where $P_m$ is a
homogeneous polynomial of degree $m\ge 1$.  We shall later apply such
estimate with $P_m$ being the Taylor polynomial of $u_0$ centered at
$0$ of order $m$, with $u_0$ and $m$ as in \Cref{sec:two-opposite}.

\begin{proposition}\label{prop_elliptic_coor}
  Let $m \in \mb{N}$, $m\ge 1$, and let $P_m$ be a homogeneous
  polynomial of degree $m$, i.e.
	\begin{equation}\label{def_P_m}
		P_m(x_1,x_2):=\sum_{j=0}^m \ell_jx_1^{m-j}x_2^j,
	\end{equation}
	for some $\ell_0,\ell_1,\dots,\ell_m \in \mb{R}$. Then,  for every $L>0$,  we have 
	\begin{equation}
          \label{eq_nablaV_elliptic}\int_{E_\e(L)\setminus S_\e} |\nabla V_{\e,P_m,E_\e(L)}|^2 \, dx = \pi (\e r_1)^{2m}\bigg(\ell_0^2\sum_{j=1}^m j |c_j|^2+\ell_1^2\sum_{j=1}^m\frac{|d_j|^2}{j}\bigg)+o(\e^{2m})
	\end{equation}
 as  $\e \to 0^+$, where 
	\begin{align}
		&c_j=\frac{1}{\pi}\int_0^{2\pi} (\cos\eta)^m\cos(j \eta) \, d\eta \quad \text{for every } j \in \mathbb{N},\label{def_c_j}\\
		&d_j=\frac{1}{\pi}\int_0^{2\pi} (\cos\eta)^{m-1}\sin \eta\sin(j \eta) \, d\eta \quad \text{for every } j \in \mathbb{N}\setminus\{0\}.\label{def_d_j}
	\end{align}
\end{proposition}

\begin{proof}
	We consider elliptic coordinates $(\xi, \eta)$ defined as 
	\begin{equation}\label{def_elliptic_coor}
		\begin{cases}
			x_1=\e r_1\cosh(\xi) \cos (\eta), \\
			x_2=\e r_1\sinh(\xi) \sin (\eta), \\
		\end{cases}
		\quad \xi\ge 0, \, \eta \in [0,2\pi),
	\end{equation} 
see e.g. \cite[Section 2.2]{AFHL}. 
	In this coordinates $S_\e$ is described by the conditions \begin{equation*}
 \xi=0,\quad \eta\in [0,2\pi),
 \end{equation*}
 whereas $E_\e(L)$ is described by  
\begin{equation*}
 \xi\in [0,\xi_\e),\quad \eta\in [0,2\pi),
 \end{equation*}
 where $\xi_\e$ is such that $r_1 \e \sinh(\xi_\e)=L$, that is 
	\begin{equation}\label{def_xi_e}
          \xi_\e=\mathop{\rm{arcsinh}}\left(\frac{L}{r_1\e}\right)=\log\left(\frac{L}{r_1\e}+\sqrt{1+\frac{L^2}{r_1^2\e^2}}\right).
	\end{equation}
	In particular $\partial E_\e(L)$ is described by the
        conditions
	\begin{equation*}
		\xi=\xi_\e,\quad \eta \in [0,2\pi).
	\end{equation*}
	The map 
	\begin{equation*}
		F_\e:[0,\xi_\e) \times [0,2\pi ) \to E_\e(L), \quad F_\e(\xi,\eta)=(x_1,x_2),
	\end{equation*}
	defined by \eqref{def_elliptic_coor},  has a Jacobian matrix of the form 
	\begin{equation*}
		J_{F_\e}(\xi,\eta)=\e r_1 \sqrt{\cosh^2\xi-\cos^2\eta}\, \,O(\xi,\eta)
	\end{equation*}
        for some orthogonal matrix $O(\xi,\eta)$, and
        $\det J_{F_\e}(\xi,\eta)=\e^2 r_1^2 (\cosh^2\xi-\cos^2\eta)$.
        In particular $F_\e$ is a conform mapping and
        $J_{F_\e}(\xi,\eta)$ is an invertible matrix if
        $(\xi,\eta)\neq (0,0)$ and $(\xi,\eta)\neq (0,\pi)$.

        Let $\widehat{V}_{\e,P_m,L}:=V_{\e,P_m,E_\e(L)}\circ F_\e$,
        where $V_{\e,P_m,E_\e(L)}$ is the solution of
        \eqref{prob_Ve_2poles} in the case $\Lambda=E_\e(L)$ and
        $h=P_m$.  We observe that, since $F_\e(\xi,\eta)\in \R^2_+$ if
        $\eta\in(0,\pi)$ and $F_\e(\xi,\eta)\in \R^2_-$ if
        $\eta\in(\pi,2\pi)$,
\begin{equation*}
  \widehat{V}_{\e,P_m,L}(0,\eta)=
  \begin{cases}
      \gamma_+(V_{\e,P_m,E_\e(L)})(\e r_1\cos\eta,0),
      &\text{if }\eta\in(0,\pi),\\
      \gamma_-(V_{\e,P_m,E_\e(L)})(\e r_1\cos\eta,0),
      &\text{if }\eta\in(\pi,2\pi).      
  \end{cases}
\end{equation*}
Furthermore,
\begin{equation*}
  \frac{\partial \widehat{V}_{\e,P_m,L}}{\partial\xi}(0,\eta)=
  \begin{cases}
    \e r_1(\sin\eta) \, \gamma_+\left(\dfrac{\partial
        V_{\e,P_m,E_\e(L)}}{\partial x_2}\right)(\e r_1\cos\eta,0),
    &\text{if }\eta\in(0,\pi),\\[10pt]
    \e r_1(\sin\eta)\, \gamma_-\left(\dfrac{\partial
        V_{\e,P_m,E_\e(L)}}{\partial x_2}\right)(\e r_1\cos\eta,0),
    &\text{if }\eta\in(\pi,2\pi).
  \end{cases}
\end{equation*}
We also note that, for every $\eta\in[0,2\pi)$, 
\begin{equation*}
    P_m(F_\e(0,\eta))=(\e r_1)^m \ell_0 
    (\cos\eta)^m\quad\text{and}\quad
\frac{\partial P_m}{\partial x_2}(F_\e(0,\eta))=\ell_1 (\e r_1)^{m-1}(\cos\eta)^{m-1.}
\end{equation*}
Therefore, $\widehat{V}_{\e,P_m,L}$ solves the problem 
	\begin{equation}\label{proof_elliptic_coor:2.1}
		\begin{cases}
                  -\Delta \widehat{V}_{\e,P_m,L}=0,  &\text{in } (0,\xi_\e)\times (0,2\pi),\\[2pt]
                  \widehat{V}_{\e,P_m,L}(\xi_\e,\eta)=0, &\text{for all } \eta \in [0,2\pi),\\[2pt]
                  \widehat{V}_{\e,P_m,L}(\xi,0)=\widehat{V}_{\e,P_m,L}(\xi,2 \pi), &\text{for all } \xi \in (0,\xi_\e),\\[2pt]
                  \widehat{V}_{\e,P_m,L}(0,\eta)+\widehat{V}_{\e,P_m,L}(0,2 \pi -\eta)=2\ell_0(\e r_1)^m(\cos\eta)^m, &\text{for all } \eta \in (0,\pi),\\[2pt]
                  \dfrac{\partial
                    \widehat{V}_{\e,P_m,L}}{\partial\xi}(0,\eta)-
                  \dfrac{\partial
                    \widehat{V}_{\e,P_m,L}}{\partial\xi}(0,2 \pi
                  -\eta)=2\ell_1(\e r_1)^m(\cos \eta)^{m-1}\sin\eta,
                  &\text{for all } \eta \in (0,\pi).
		\end{cases}
	\end{equation}
	Let us consider the Fourier expansion of $(\e r_1)^{-m}\widehat{V}_{\e,P_m,L}$ with respect to the variable $\eta$
	\begin{equation*}
          \frac{1}{(\e r_1)^m}\widehat{V}_{\e,P_m,L}(\xi,\eta)=\frac{a_{0,\e}(\xi)}{2}+\sum_{j=1}^\infty \Big(a_{j,\e}(\xi)\cos(j\eta)+b_{j,\e}(\xi)\sin(j \eta)\Big),
	\end{equation*}
	where 
	\begin{align*}
          &a_{j,\e}(\xi):=\frac{(\e r_1)^{-m}}{\pi}\int_0^{2\pi}\widehat{V}_{\e,P_m,L}(\xi,\eta) \cos(j \eta) \, d\eta \quad \text{for all } j \in \mathbb{N}, \\
          &b_{j,\e}(\xi):=\frac{(\e r_1)^{-m}}{\pi}\int_0^{2\pi}\widehat{V}_{\e,P_m,L}(\xi,\eta) \sin(j \eta) \, d\eta \quad \text{for all } j \in \mathbb{N}\setminus\{0\}.
    \end{align*}
	Since $\cos(2 \pi -\eta)=\cos\eta$ for any $\eta \in  (0,\pi)$, from \eqref{proof_elliptic_coor:2.1} 
  it follows that 
	\begin{equation*}
          a_{0,\e}(0)+2\sum_{j=1}^\infty a_{j,\e}(0) \cos(j\eta)=2\ell_0(\cos\eta)^m \quad\text{ for all } \eta \in  (0,2\pi),
	\end{equation*}
        hence $\{a_{j,\e}(0)\}_{j\in \mb{N}}$ are the Fourier
        coefficients of $\ell_0(\cos\eta)^m$ with respect to the
        orthonormal basis
        $\big\{\frac{1}{\sqrt{2\pi}},\frac{1}{\sqrt{\pi}}\cos(j\eta),
        \frac{1}{\sqrt{\pi}}\sin(j\eta)\big\}_{j \in
          \mathbb{N}\setminus \{0\}}$ of $L^2(0,2\pi)$, i.e.
	\begin{equation*}
    a_{j,\e}(0)=\ell_0 c_j\quad \text{for all } j \in \mathbb{N},
\end{equation*}
with $c_j$ as in \eqref{def_c_j}. 
	In particular 
 \begin{equation}\label{eq:cj0}
     a_{j,\e}(0)=\ell_0 c_j=0 \quad \text{if }j>m.
 \end{equation} 
 On the other hand, the last condition in
 \eqref{proof_elliptic_coor:2.1} reads as
	\begin{equation*}
		\sum_{j=1}^\infty b'_{j,\e}(0) \sin(j\eta)=\ell_1(\cos\eta)^{m-1} \sin\eta \quad \text{for all } \eta \in  (0,2\pi).
	\end{equation*}
	It follows that $b_{j,\e}'(0)$ are independent of $\e$ and
	\begin{equation*}
		b'_{j,\e}(0)=\ell_1d_j\quad \text{for all } j \in \mathbb{N}\setminus\{0\},
	\end{equation*}
 with $d_j$ as in \eqref{def_d_j};
	 hence 
  \begin{equation}\label{eq:dj0}
  b_{j,\e}'(0)=\ell_1 d_j=0 \quad\text{if }j>m.    
  \end{equation}
  From the equation in \eqref{proof_elliptic_coor:2.1} it follows 
 that 
\begin{align*}
    0&=\frac{1}{(\e r_1)^m}\Delta\widehat{V}_{\e,P_m,L}\\
    &=\frac{a''_{0,\e}(\xi)}{2}+\sum_{j=1}^\infty \Big((a_{j,\e}''(\xi)-j^2a_{j,\e}(\xi))\cos(j\eta)+(b_{j,\e}''(\xi)-j^2b_{j,\e}(\xi))\sin(j \eta)\Big),
	\end{align*}
	hence
	\begin{align}
		&a_{0,\e}(\xi)=-\frac{a_{0,\e}(0)}{\xi_\e} \xi +a_{0,\e}(0) 
  =-\frac{\ell_0 c_0}{\xi_\e} \xi +\ell_0c_0\quad \text{for all } \xi \in (0,\xi_\e),\label{proof_elliptic_coor:10}\\
		\notag &a_{j,\e}(\xi)= \ell_0c_j \left(\frac{e^{j\xi}}{1-e^{2j\xi_\e}}+\frac{e^{-j\xi}}{1-e^{-2j\xi_\e}}\right) \quad \text{for all } \xi \in (0,\xi_\e) 
		\text{ and } j \in \mathbb{N}\setminus\{0\},\\
		\notag &b_{j,\e}(\xi)=\frac{\ell_1 d_j}{j}\left(\frac{e^{j\xi}}{1+e^{2j\xi_\e}}-\frac{e^{-j\xi}}{1+e^{-2j\xi_\e}}\right) \quad \text{for all } \xi \in (0,\xi_\e) 
		\text{ and }j \in \mathbb{N}\setminus\{0\}, 	\end{align}
	with $\xi_\e$ as in \eqref{def_xi_e}.
	Then, by \eqref{eq:cj0} and \eqref{eq:dj0}, $a_{j,\e}\equiv b_{j,\e}\equiv0$ for all $j>m$, so that   
 \begin{equation*}
   \frac{1}{(\e r_1)^m}\widehat{V}_{\e,P_m,L}(\xi,\eta)=\frac{a_{0,\e}(\xi)}{2}+\sum_{j=1}^m \Big(a_{j,\e}(\xi)\cos(j\eta)+b_{j,\e}(\xi)\sin(j \eta)\Big).
	\end{equation*}
 By a change of variables and the Parseval identity,
\begin{align}\label{eq:parse}
  \int_{E_\e(L)\setminus S_\e} &|\nabla V_{\e,P_m,E_\e(L)}|^2 \, dx= \int_{0}^{\xi_\e}\int_{0}^{2\pi}|\nabla \widehat{V}_{\e,P_m,L}|^2  \, d\eta\, d\xi\\
  \notag &  =(\e r_1)^{2m}\frac{\pi}{2} \int_{0}^{\xi_\e}|a_{0,\e}'(\xi)|^2 \, d\xi\\
  \notag &\quad+(\e r_1)^{2m}\pi\sum_{j=1}^{m}\int_{0}^{\xi_\e}\left(
           |a'_{j,\e}(\xi)|^2+j^2|b_{j,\e}(\xi)|^2 +|b'_{j,\e}(\xi)|^2+j^2|a_{j,\e}(\xi)|^2 \right)\, d\xi.
	\end{align}
	Let us compute each integral in the above formula. In view of
        \eqref{proof_elliptic_coor:10} and \eqref{def_xi_e}, it is
        clear that
	\begin{equation*}
		\int_{0}^{\xi_\e}|a_{0,\e}'(\xi)|^2 \, d\eta=\frac{ \ell_0^2c_0^2 }{\xi_\e}=\frac{ \ell_0^2c_0^2 }{|\log\e|}+O\left(\frac{1}{|\log\e|^2}\right) \quad 
  \text{as } \e \to 0^+.
	\end{equation*}
	Furthermore, for every $j \in \mathbb{N}\setminus\{0\}$,
	\begin{align*}
	&	j^2\int_{0}^{\xi_\e}|b_{j,\e}(\xi)|^2 \, d\xi= \ell_1^2 d_j^2 \int_0^{\xi_\e} \left(\frac{e^{j\xi}}{1+e^{2j\xi_\e}}-\frac{e^{-j\xi}}{1+e^{-2j\xi_\e}}\right)^2 \, d\xi
		\\
  &=\frac{ \ell_1^2 d_j^2 }{(1+e^{2j\xi_\e})^2}\int_0^{\xi_\e}e^{2j\xi}\, d \xi +\frac{ \ell_1^2 d_j^2 }{(1+e^{-2j\xi_\e})^2}\int_0^{\xi_\e}e^{-2j\xi}\, d \xi 
		-\frac{2 \ell_1^2 d_j^2 \xi_\e}{2+e^{-2j\xi_\e}+e^{2j\xi_\e}}\\
  &=\frac{ \ell_1^2 d_j^2 }{2j}\left(\frac{1}{(1+e^{2j\xi_\e})^2}(e^{2j\xi_\e}-1)+\frac{1}{(1+e^{-2j\xi_\e})^2}(1-e^{-2j\xi_\e})\right)-\frac{2 \ell_1^2 d_j^2 \xi_\e}{2+e^{-2j\xi_\e}+e^{2j\xi_\e}}\\
	&	=\frac{ \ell_1^2 d_j^2 }{2j}(1+o(1))
 \quad \text{as } \e \to 0^+
	\end{align*}
	and similarly
	\begin{multline*}
          \int_{0}^{\xi_\e}|b'_{j,\e}(\xi)|^2 \, d\xi=\ell_1^2
          d_j^2\int_0^{\xi_\e}
          \left(\frac{e^{j\xi}}{1+e^{2j\xi_\e}}+\frac{e^{-j\xi}}{1+e^{-2j\xi_\e}}\right)^2
          d\xi =\frac{ \ell_1^2 d_j^2 }{2j}(1+o(1)) \quad\text{as } \e
          \to 0^+.
	\end{multline*}
	Finally, for every $j \in \mathbb{N}\setminus\{0\}$,
	\begin{align*}
          &\int_{0}^{\xi_\e}|a'_{j,\e}(\xi)|^2 \, d\xi=j^2\ell_0^2c_j^2\int_0^{\xi_\e} \left(\frac{e^{j\xi}}{1-e^{2j\xi_\e}}-\frac{e^{-j\xi}}{1-e^{-2j\xi_\e}}\right)^2 \, d\xi
            =\ell_0^2 c_j^2\frac{j}{2}(1+o(1)), \\
          &j^2\int_{0}^{\xi_\e}|a_{j,\e}(\xi)|^2 \, d\xi=j^2\ell_0^2c_j^2\int_0^{\xi_\e} \left(\frac{e^{j\xi}}{1-e^{2j\xi_\e}}+\frac{e^{-j\xi}}{1-e^{-2j\xi_\e}}\right)^2 \, d\xi
            =\ell_0^2 c_j^2\frac{j}{2}(1+o(1)), 
	\end{align*}
        as $\e \to 0^+$, as shown in the proof of \cite[Lemma
        2.3]{AFHL}.  Replacing the above estimates in \eqref{eq:parse}
        we obtain \eqref{eq_nablaV_elliptic}.
\end{proof}

\begin{proposition}\label{prop_sum_cj_dj}
  Let $m \in \mathbb{N}\setminus\{0\}$. For every
  $j \in \mathbb{N}\setminus\{0\}$, let $c_j$ and $d_j$ be as in
  \eqref{def_c_j} and \eqref{def_d_j}, respectively.  Then
 \begin{align}
		&\sum_{j=1}^m j |c_j|^2=\frac{m}{4^{m-1}}\binom{m-1}{\lfloor\frac{m-1}{2}\rfloor}^{\! 2},\label{eq_sum_cj}\\
		&\sum_{j=1}^m \frac{1}{j} |d_j|^2=\frac{1}{m4^{m-1}}\binom{m-1}{\lfloor\frac{m-1}{2}\rfloor}^{\! 2}\label{eq_sum_dj}.
	\end{align}
\end{proposition}
\begin{proof}
  For the proof of \eqref{eq_sum_cj} we refer to \cite[Proposition
  A.3]{AFL}. To prove \eqref{eq_sum_dj}, we observe that, in view of
  \eqref{def_c_j},
	\begin{equation*}
		(\cos\eta)^m=\frac{c_0}{2}+\sum_{j=1}^m c_j
                \cos(j\eta)
                \quad \text{for all } \eta \in [0,2\pi].
	\end{equation*}
	Deriving the previous identity with respect to $\eta$, we
        obtain
	\begin{equation*}
		(\cos\eta)^{m-1}\sin\eta=\frac{1}{m}\sum_{j=1}^m j c_j \sin(j\eta)=\sum_{j=1}^m d_j \sin(j\eta) \quad \text{for all } \eta \in [0,2\pi],
	\end{equation*}
	in view of \eqref{def_d_j}. It follows that $d_j=\frac{j}{m} c_j$ for 
 all $j=1,\dots,m$, hence \eqref{eq_sum_dj} follows from \eqref{eq_sum_cj}.
\end{proof}

\begin{remark}\label{remark_Pm}
  Let $m \in \mb{N}$, $m\ge 1$, and let $P_m$ be a homogeneous
  polynomial of degree $m$ as in \eqref{def_P_m}. Let
  $\Lambda\subset \R^2$ be a simply connected open bounded domain such
  that $D_R\subseteq\Lambda$.
\begin{enumerate}[\rm (i)]\setlength\itemsep{0.8em}
\item If the coefficient $\ell_0$ in \eqref{def_P_m} is zero, then
  $P_m\equiv 0$ on $S_\e$ for all $\e\in(0,1]$. Hence
  $V_{\e,P_m,\Lambda}\in \widetilde{\mathcal H}_{\e,\Lambda}$ and, in
  view of \eqref{eq_Ve_2poles},
  $\int_{\Lambda \setminus S_\e} |\nabla V_{\e,P_m,\Lambda}|^2\, dx =
  -L_{\e,P_m,\Lambda}(V_{\e,P_m,\Lambda})$, so that
\begin{equation*}
    \mc{E}_{\e,P_m,\Lambda}=-\frac12 \int_{\Lambda \setminus S_\e} |\nabla V_{\e,P_m,\Lambda}|^2\, dx. 
\end{equation*}
\item If the coefficient $\ell_1$ in \eqref{def_P_m} is zero, then
  $\frac{\partial P_m}{\partial x_2}\equiv 0$ on $S_\e$ for all
  $\e\in(0,1]$. Hence $L_{\e,P_m,\Lambda}\equiv 0$ and
\begin{equation*}
    \mc{E}_{\e,P_m,\Lambda}=\frac12 \int_{\Lambda \setminus S_\e} |\nabla V_{\e,P_m,\Lambda}|^2\, dx. 
\end{equation*}
\end{enumerate}
\end{remark}

\begin{proposition}\label{prop_asympotic_E_ePm}
  Let $\Omega \subset \R^2$ be a simply connected open bounded domain
  with $0 \in D_R\subseteq\Omega$. For every $\e\in(0,1)$, let $S_\e$
  be defined in \eqref{def_S_e}. Let $P_m$ be a homogeneous polynomial
  of degree $m$ as in \eqref{def_P_m} and $\mathcal E_{\e,P_m,\Omega}$
  be defined in \eqref{def_E_he} with $\Lambda=\Omega$ and
  $h=P_m$. Then, letting $\ell_0$ and $\ell_1$ be as in
  \eqref{def_P_m}, we have
 \begin{enumerate}[\rm (i)]\setlength\itemsep{0.8em}
     \item if $\ell_0=0$, then
    \begin{equation*}
        \mathcal E_{\e,P_m,\Omega}=-\frac{\pi}{2}r_1^{2m}\ell_1^2 \e^{2m}
        \frac{1}{m4^{m-1}}\binom{m-1}{\lfloor\frac{m-1}{2}\rfloor}^{\!2}+o(\e^{2m})
        \quad\text{as }\e\to0^+;
    \end{equation*} 
    \item if $\ell_1=0$, then
    \begin{equation*}
        \mathcal E_{\e,P_m,\Omega}=\frac{\pi}{2}r_1^{2m}\ell_0^2 \e^{2m}
        \frac{m}{4^{m-1}}\binom{m-1}{\lfloor\frac{m-1}{2}\rfloor}^{\!2}+o(\e^{2m})
        \quad\text{as }\e\to0^+.
    \end{equation*} 
 \end{enumerate}
\end{proposition}

\begin{proof}
  The set $\Omega$ is open and $0 \in \Omega$, hence there exist
  $L_1,L_2>0$ such that, for every $\e\in(0,1]$,
  $S_\e\subset E_\e(L_1)\subset\Omega\subset E_\e(L_2)$ (e.g. we can
  choose any $0<L_1<\sqrt{R^2-r_1^2}$ and
  $L_2=\mathop{\rm diam}\Omega$).  From \eqref{eq:Vepsin_2poles},
  \eqref{def_E_he}, \and the space inclusions
  $\mathcal H_{\e,E_\e(L_1)}\subset \mathcal H_{\e,\Omega}\subset
  \mathcal H_{\e,E_\e(L_2)}$,
  $\widetilde{\mathcal H}_{\e,E_\e(L_1)}\subset \widetilde{\mathcal
    H}_{\e,\Omega}\subset \widetilde{\mathcal H}_{\e,E_\e(L_2)}$
  obtained by trivial extension, we deduce that, for every
  $\e\in(0,1]$,
	\begin{equation}\label{proof_asympotic_E_eP_m}
		\mc{E}_{\e,P_m, E_\e(L_2)} \le \mc{E}_{\e,P_m,\Omega} \le \mc{E}_{\e,P_m,E_\e(L_1)}.
	\end{equation}
        If $\ell_0=0$, from Remark \ref{remark_Pm},
        \eqref{eq_nablaV_elliptic}, and \eqref{eq_sum_dj} it follows
        that, for $i=1,2$,
\begin{align*}
    \mathcal E_{\e,P_m,E_\e(L_i)}&=-\frac12 \int_{E_\e(L_i) \setminus
                                   S_\e}
                                   |\nabla V_{\e,P_m,E_\e(L_i)}|^2\, dx\\
    &=-\frac{\pi}{2} (\e
      r_1)^{2m}\ell_1^2\bigg(\sum_{j=1}^m\frac{|d_j|^2}{j}\bigg)
      +o(\e^{2m})=-\frac{\pi}{2}r_1^{2m}\ell_1^2 \e^{2m}
    \frac{1}{m4^{m-1}}\binom{m-1}{\lfloor\frac{m-1}{2}\rfloor}^{\!2}+o(\e^{2m})    
\end{align*}
as $\e\to0^+$,
thus proving (i) in view of \eqref{proof_asympotic_E_eP_m}.

On the other hand, if $\ell_1=0$, then Remark \ref{remark_Pm},
\eqref{eq_nablaV_elliptic}, and \eqref{eq_sum_cj} imply that, for
$i=1,2$,
\begin{align*}
    \mathcal E_{\e,P_m,E_\e(L_i)}&=\frac12 \int_{E_\e(L_i) \setminus S_\e} |\nabla V_{\e,P_m,E_\e(L_i)}|^2\, dx\\
    &=\frac{\pi}{2} (\e r_1)^{2m}\ell_0^2\bigg(\sum_{j=1}^m j|c_j|^2\bigg)+o(\e^{2m})
    =\frac{\pi}{2}r_1^{2m}\ell_0^2 \e^{2m}
\frac{m}{4^{m-1}}\binom{m-1}{\lfloor\frac{m-1}{2}\rfloor}^{\!2}+o(\e^{2m}) 
\end{align*}
as $\e\to0^+$, thus proving (ii) in view of \eqref{proof_asympotic_E_eP_m}.
\end{proof}

Let $u_0$ be as in \eqref{def_u_0} with $u_0(0)=0$ and
$m,\beta,\alpha_0$ be as in \eqref{asymptotic_u0_even}. Let $T_m$ be
the Taylor polynomial of $u_0$ centered at $0$ of order $m$ written in
\eqref{def_T_m}. In particular $T_m$ is of the form \eqref{def_P_m}
with
\begin{equation*}
    \ell_j=\frac{1}{(m-j)!j!}\pd{^mu_0}{x^{m-j}_1\partial x_2^j}(0).
\end{equation*}
If $\alpha_0=\frac{j\pi}{m}$ for some $j\in\{0,1,\dots,2m-1\}$, then,
by Remark \ref{remark_beta},
\begin{equation*}
    \ell_0=T_m(1,0)=0
\end{equation*}
and
\begin{equation*}
    \ell_1=\frac{\partial T_m}{\partial x_2}(1,0)=\nabla
    T_m(1,0)\cdot(0,1)
    =m\beta\cos(j\pi)=(-1)^jm\beta.
\end{equation*}
Hence, by Proposition \ref{prop_asympotic_E_ePm}, in this case
we have
\begin{equation}\label{eq:stf1}
        \mathcal E_{\e,T_m,\Omega}=-\frac{\pi}{2}r_1^{2m} \e^{2m}
        \frac{m\beta^2}{4^{m-1}}\binom{m-1}{\lfloor\frac{m-1}{2}\rfloor}^{\!2}+o(\e^{2m})
        \quad\text{as }\e\to0^+.
\end{equation}
On the other hand, if $\alpha_0=\frac{\pi}{2m}+\frac{j\pi}{m}$ for
some $j\in\{0,1,\dots,2m-1\}$, then, by Remark \ref{remark_beta},
\begin{equation*}
    \ell_0=T_m(1,0)=-\beta\sin\big(\tfrac{\pi}{2}+j\pi\big)=(-1)^{j+1}\beta
\end{equation*}
and
\begin{equation*}
    \ell_1=\frac{\partial T_m}{\partial x_2}(1,0)=m\beta\cos
    \big(\tfrac{\pi}{2}+j\pi\big)=0.
\end{equation*}
In this case, Proposition \ref{prop_asympotic_E_ePm} then 
provides the expansion 
\begin{equation}\label{eq:stf2}
       \mathcal E_{\e,T_m,\Omega}=\frac{\pi}{2}r_1^{2m}\e^{2m}
       \frac{m\beta^2}{4^{m-1}}\binom{m-1}{\lfloor\frac{m-1}{2}\rfloor}^{\!2}+o(\e^{2m})
       \quad\text{as }\e\to0^+.
\end{equation}
Let $g:=u_0-T_m$.  Since $u_0$ is smooth and $T_m$ is its Taylor
polynomial at $0$ of order $m$, then
\begin{equation}\label{ineq_f}
	g(x)=O(|x|^{m+1})   \quad \text{and} \quad |\nabla g(x)|
        =O(|x|^{m})
        \quad \text{as } x\to0.
\end{equation}

\begin{proposition}
  Let $m$ and $\alpha_0$ be as in \eqref{asymptotic_u0_even}.  For
  every $\e\in(0,1]$, let $V_{\e,T_m,\Omega}$ and
  $\mc{E}_{\e,T_m,\Omega}$ be as in \eqref{eq:Vepsin_2poles} and
  \eqref{def_E_he}, with $\Lambda=\Omega$ and $h=T_m$, and let
  $V_\e=V_{\e,u_0,\Omega}$ and
  $\mathcal E_\e=\mathcal E_{\e,u_0,\Omega}$ be as in
  \eqref{eq:Vepsin} and \eqref{eq:defEe}, respectively. Then
 \begin{equation}
  \label{eq_normVe_normVePm}
  \norm{V_\e-V_{\e,T_m,\Omega}}_{\mc{H_\e}}^2=O(\e^{2m+1})=o(\e^{2m})\quad \text{ as }\e \to 0^+
     \end{equation}
   and, if either $\alpha_0=\frac{j\pi}{m}$ or $\alpha_0=\frac{\pi}{2m}+\frac{j\pi}{m}$ for some $j\in\{0,1,\dots,2m-1\}$,  
     \begin{align}
\label{eq:stimaVe2p}
 &\|V_\e\|_{\mc{H_\e}}=O(\e^m) \quad \text{as }\e\to0^+,\\
&\label{eq_Ee_EePm} 
    \mc{E}_\e-\mc{E}_{\e,T_m,\Omega}=o(\e^{2m}) \quad \text{ as }\e \to 0^+. 
\end{align}
\end{proposition}
\begin{proof}
  Let $W_\e:=V_\e-V_{\e,T_m,\Omega}$. Then $W_\e$ satisfies
  \eqref{eq_Ve_2poles} with $h:=g$. Let $\eta_\e$ be as in
  \eqref{eq:cut-off}. Testing \eqref{eq_Ve_2poles} with
  $w=W_\e-\eta_\e g$, by Young's Inequality and \eqref{def_traces} we
  obtain
	\begin{multline*}
          \norm{W_\e}_{\mc{H_\e}}^2=\int_{\Omega \setminus S_\e
          }\eta_\e\nabla W_\e \cdot\nabla g \,dx +\int_{\Omega
            \setminus S_\e }g\nabla W_\e \cdot\nabla \eta_\e \,dx
          -2 \int_{S_\e} \pd{g}{x_2} \gamma^+(W_\e)\, d S +2\int_{S_\e} \pd{g}{x_2} g\, d S \\
          \le \frac{1}{2} \norm{W_\e}_{\mc{H_\e}}^2
          +C\left(\int_{\Omega}\eta_\e^2|\nabla g|^2
            \,dx+\int_{\Omega}g^2|\nabla \eta_\e|^2\, dx+\int_{S_\e}
            \left|\pd{g}{x_2}\right|^2dS\right) +2\int_{S_\e}
          \left|\pd{g}{x_2}\right| |g|\, d S,
	\end{multline*}
	for some positive constant $C>0$. Hence
        \eqref{eq_normVe_normVePm} follows from \eqref{eq:cut-off} and
        \eqref{ineq_f}.
	
We have 
\begin{equation}\label{eq:stimaav4}
  \mc{E}_\e-\mc{E}_{\e,T_m,\Omega}={\frac12}\left(\norm{V_\e}_{\mc{H_\e}}^2-\norm{V_{\e,T_m,\Omega}}_{\mc{H_\e}}^2\right) +2 \int_{S_\e} \left(\pd{u_0}{x_2} \gamma_+(V_\e)-\pd{T_m}{x_2}\gamma_+(V_{\e,T_m,\Omega})\right)  dS.
\end{equation}
By Remark \ref{remark_Pm} and Proposition \ref{prop_asympotic_E_ePm}
we have that, if either $\alpha_0=\frac{j\pi}{m}$ or
$\alpha_0=\frac{\pi}{2m}+\frac{j\pi}{m}$ for some
$j\in\{0,1,\dots,2m-1\}$, then
$\|V_{\e,T_m,\Omega}\|_{\mc{H_\e}}=\sqrt{2|\mc{E}_{\e,T_m,\Omega}|}=O(\e^m)$
as $\e\to0^+$. Then, \eqref{eq:stimaVe2p} follows from
\eqref{eq_normVe_normVePm}.  Using again \eqref{eq_normVe_normVePm} we
conclude that
\begin{equation}\label{eq:stimaav3}
  \norm{V_\e}_{\mc{H_\e}}^2-\norm{V_{\e,T_m,\Omega}}_{\mc{H_\e}}^2
  =(V_\e-V_{\e,T_m,\Omega},V_\e+V_{\e,T_m,\Omega})_{\mc{H_\e}}
=o(\e^{2m})\quad \text{ as }\e \to 0^+.
\end{equation}
 Furthermore, fixing some $p>2$ and letting $p'=\frac{p}{p-1}$, 
H\"older's inequality, \eqref{ineq_f}, and the continuity of the 
trace operators \eqref{def_traces} imply  that
\begin{align}\label{eq:stimaav2p}
  \bigg| \int_{S_\e} \bigg(&\pd{u_0}{x_2} \gamma_+(V_\e)-\pd{T_m}{x_2}\gamma_+(V_{\e,T_m,\Omega})\bigg)  \, d S\bigg| \\
  \notag&=
          \left| \int_{S_\e} \left(\pd{g}{x_2} \gamma_+(V_\e)+\pd{T_m}{x_2}(\gamma_+(V_\e)-\gamma_+(V_{\e,T_m,\Omega}))\right) \, d S\right|\\
  \notag&\leq \int_{S_\e} \left|\pd{g}{x_2}\right| |\gamma_+(V_\e)| \, dS+\int_{S_\e}\left|\pd{T_m}{x_2}\right||\gamma_+(W_\e)| \, d S\\
  \notag&\le {\rm{const}}\left(\e^{m+\frac{1}{p'}}\left(\int_{S_\e}|\gamma_+(V_\e)|^p\, \colg{d S} \right)^{\!\!1/p}+  
          \e^{m-1+\frac{1}{p'}}\left(\int_{S_\e}|\gamma_+(W_\e)|^p \, \colg{d S} \right)^{\!\!1/p}\right) \\
  \notag&\le {\rm{const}} \left(\e^{m+\frac{1}{p'}}\norm{V_\e}_{\mc{H_\e}} +\e^{m-1+\frac{1}{p'}}\norm{W_\e}_{\mc{H_\e}} \right)\\
  \notag&=O\Big(\e^{2m+\frac{1}{p'}}\Big)+O\Big(\e^{2m-\frac12+\frac{1}{p'}}\Big)=o\big(\e^{2m}\big)\quad\text{as }\e\to0^+,
\end{align}
where we used estimates \eqref{eq:stimaVe2p} and
\eqref{eq_normVe_normVePm}. Combining \eqref{eq:stimaav4},
\eqref{eq:stimaav3}, and \eqref{eq:stimaav2p} we finally
obtain~\eqref{eq_Ee_EePm}.
\end{proof}

\begin{proof}[\textbf{Proof of Theorems \ref{theorem_2poles_nodal} and  \ref{theorem_2poles_bisector}}]
  Since we are considering only two opposite poles on the same line,
  we have $v_0=e^{-i\Theta_0}u_0=u_0$.  Let
  $m \in \mb{N} \setminus\{0\}$ and $\alpha_0 \in [0,\frac{\pi}{m})$
  be as in \eqref{asymptotic_u0_even}.
	
  If $\alpha_0=\frac{j\pi}{m}$ or
  $\alpha_0=\frac{\pi}{2m}+\frac{j\pi}{m}$ for some
  $j\in\{0,1,\dots,2m-1\}$, then, by \eqref{asymptotic_u0_even} (see
  Remark \ref{remark_beta}), $u_0(x)=T_m(x)+O(|x|^{m+1})$ and
  $\frac{\partial u_0}{\partial x_2}(x)=\frac{\partial T_m}{\partial
    x_2}(x)+O(|x|^{m})$ as $x\to0$, so that
\begin{equation}\label{eq:stf3}
  L_\e(u_0)=2  \int_{S_\e}\frac{\partial T_m}{\partial x_2}
  T_m\, dS+O(\e^{2m+1})=O(\e^{2m+1})\quad\text{as }\e\to0^+
\end{equation}
since, in this case, either $T_m\big|_{S_\e}\equiv0$ or
$\frac{\partial T_m}{\partial x_2}\big|_{S_\e}\equiv0$.

 From Theorem \ref{t:main1}, 
 \eqref{eq:stf3}, \eqref{eq:stimaVe2p}, and 
 \eqref{eq_Ee_EePm}, it follows that 
	\begin{equation}\label{eq:stf5}
		\la_{\e, n_0} -\la_{0,
                  n_0}=2\mc{E}_\e-2L_\e(u_0)
    +o\big(\|V_\e\|^2_{\mc{H}_\e}\big)=2\mc{E}_{\e,T_m,\Omega}+o(\e^{2m})
                \quad \text{as } \e \to 0^+.
	\end{equation}
 Theorem \ref{theorem_2poles_nodal} follows from \eqref{eq:stf5}
and \eqref{eq:stf1}, while Theorem \ref{theorem_2poles_bisector}
is a consequence of \eqref{eq:stf5} and \eqref{eq:stf2}.
\end{proof}

\begin{remark}\label{rem:m=0}
  The case $m=0$ has been omitted in the present section as, for
  $u_0(0)\neq0$ the sharp expansion is already contained in
  \cite{AFLeigenvar} even without symmetry assumptions on the domain;
  however, the above argument could also apply in such a case,
  providing an alternative proof of the result of \cite{AFLeigenvar}.
\end{remark}

\section{Dealing with more general configurations of poles}\label{sec:general}

In this section, we give a hint on how our approach could be extended
to treat other possible configurations of poles, which are not covered
in the present paper for the sake of simplicity of exposition. By
\Cref{t:main1}, the quantity that sharply measures the eigenvalue
variation is $\mathcal{E}_\e-L_\e(v_0)$, where $\mathcal{E}_\e$ is as
in \eqref{eq:defEe}, $L_\e$ as in \eqref{def_Le} and $v_0$ is the
limit eigenfunction after a gauge transformation, thus solving
\eqref{prob_eigenvlaue_guged0}.  As explained in the introduction,
$\mathcal{E}_\e$ is essentially an intermediate quantity between a
capacity and a torsional rigidity, \emph{measuring} the set
$\cup_{j=1}^{k_1+k_2}S_\e^j$.  For the success of our method it is
important that the limit eigenfunction $v_0$ is regular on the sets
$S_\e^j$, while the perturbed eigenfunction $v_\e$ jumps on them,
together with $\nabla v_\e\cdot\nu^j$. Our approach can be applied to
all configurations of poles for which, after a gauge transformation as
in \Cref{sec_equivalent_Aharonov_Bohm}, the origin belongs to the
half-lines on which the perturbed eigenfunction $v_\e$ jumps.

We provide below some examples. Since the gauge transformation for a
configuration of poles is the composition of the gauge transformations
of the families of poles lying on the same straight line, we now focus
on a single set of $k$ collinear poles. Hence, for sake of simplicity,
we assume
\[
    \{a^j\}_{j=1,\dots,k}\subset D_R(0)\cap\{(x_1,0)\colon x_1\in\R\}\subset\Omega.
\]
More precisely, we assume that $k=n_1+n_2$, where
$n_1,n_2\in\mathbb{N}$ denote, respectively, the number of poles which
lie on the left and on the right side with respect to the origin
(either $n_1$ or $n_2$ might be zero). Namely,
\[
    a^j=\begin{cases}
        (-\delta_j,0), &\text{for }j=1,\dots,n_1, \\
        (\delta_j,0), &\text{for }j=n_1+1,\dots,n_1+n_2,
    \end{cases}
\]
where $\delta_j>0$ are such that
\begin{align*}
    -\delta_1<-\delta_2<\cdots<-\delta_{n_1}<0<\delta_{n_1+1}<\cdots<\delta_{n_1+n_2}.
\end{align*}
For the above configuration, we consider problem
\eqref{prob_Aharonov-Bohm_multipole}. One of the following cases
occurs:
\begin{enumerate}[\rm(i)]
    \item $n_1$ and $n_2$ are both even;
    \item $n_1$ and $n_2$ are both odd;
    \item $n_1$ is odd and $n_2$ is even (or vice versa).
\end{enumerate}
The procedure developed to prove our main result \Cref{t:main1} can be
reproduced in cases (i) and (ii), as well as in case (iii) if $n_2=0$.

Let us now briefly describe, in these cases, how problem
\eqref{prob_Aharonov-Bohm_multipole} becomes after a tailored gauge
transformation. Hereafter, we denote by $\Sigma:=\R\times\{0\}$ the
$x_1$ axis, by $T\colon H^1(\R^2\setminus\Sigma)\to L^p(\Sigma)$ the
jump trace operator defined as in \eqref{eq:trTj} with $\Sigma$
instead of $\Sigma^j$, and by $\nu:=(0,1)$.

\medskip\noindent {\bf Case (i):} even number of poles evenly
distributed, i.e. $n_1=2N$ and $n_2=2M$ for some $N,M\in\mathbb{N}$
(see \Cref{subfig:case_1}). In this case, reasoning as in
\Cref{sec_equivalent_Aharonov_Bohm}, it is possible to find a gauge
transformation such that problem \eqref{prob_Aharonov-Bohm_multipole}
is equivalent to
\[
    \begin{cases}
        -\Delta v=\lambda v, &\text{in }\Omega\setminus \bigcup_{j=1}^{N+M}S_\e^j, \\
        v=0, &\text{on }\partial\Omega, \\
        T(v)=T(\nabla v\cdot\nu)=0, &\text{on }\bigcup_{j=1}^{N+M}S_\e^j,
    \end{cases}
\]
where
\[
    S_\e^j:=\begin{cases}
        [-\e\delta_{2j-1},-\e\delta_{2j}]\times\{0\}, &\text{if }j=1,\dots,N, \\
        [\e\delta_{2j-1},\e\delta_{2j}]\times\{0\}, & \text{if }j=N+1,\dots,N+M.
    \end{cases}
\]
\\
{\bf Case (ii):} even number of poles oddly distributed,
i.e. $n_1=2N+1$ and $n_2=2M+1$ for some $N,M\in\mathbb{N}$ (see
\Cref{subfig:case_2}). Once again, reasoning as in
\Cref{sec_equivalent_Aharonov_Bohm}, one can find a gauge
transformation such that problem \eqref{prob_Aharonov-Bohm_multipole}
is equivalent to
\[
    \begin{cases}
        -\Delta v=\lambda v, &\text{in }\Omega\setminus \bigcup_{j=1}^{N+M+1}S_\e^j, \\
        v=0, &\text{on }\partial\Omega, \\
        T(v)=T(\nabla v\cdot\nu)=0, &\text{on }\bigcup_{j=1}^{N+M+1}S_\e^j,
    \end{cases}
\]
where
\[
    S_\e^j:=\begin{cases}
        [-\e\delta_{2j-1},-\e\delta_{2j}]\times\{0\}, &\text{for }j=1,\dots,N, \\
        [-\e\delta_{2N+1},\e\delta_{2N+2}]\times\{0\}, &\text{for }j=N+1, \\
        [\e\delta_{2j-1},\e\delta_{2j}]\times\{0\}, &\text{for }j=N+2,\dots,N+M+1.        
    \end{cases}
  \] {\bf (iii):} odd number of poles all on the same side,
  i.e. $n_1=2N+1$ and $n_2=0$ (see \Cref{subfig:case_3}). In this
  case, problem \eqref{prob_Aharonov-Bohm_multipole} is equivalent to
\[
    \begin{cases}
      -\Delta v=\lambda v, &\text{in }\Omega\setminus
      \left[\Gamma_0\cup\left(\bigcup_{j=1}^{N+1}S_\e^j\right)\right], \\
        v=0, &\text{on }\partial\Omega, \\
        T(v)=T(\nabla v\cdot\nu)=0, &\text{on }\Gamma_0, \\
        T(v)=T(\nabla v\cdot\nu)=0, &\text{on }\bigcup_{j=1}^{N+1}S_\e^j,
    \end{cases}
\]
where
\[
    S_\e^j
    :=\begin{cases}
    [-\e\delta_{2j-1},-\e\delta_{2j}]\times\{0\}, &\text{for  }j=1,\dots,N,\\
    [-\e\delta_{2N+1},0]\times\{0\}, &\text{for  }j=N+1.
    \end{cases}
\]
To conclude, the only case left open in the present work is case (iii)
with $n_2\neq0$. This requires non-trivial technical adaptations and
will be the object of future investigation.

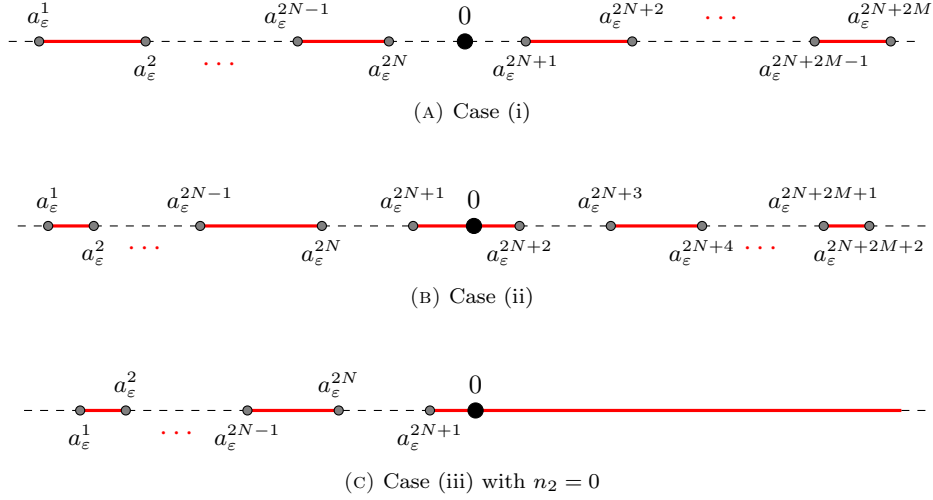
\begin{figure}[ht]
    \centering
    \subfloat[Case (i) \label{subfig:case_1}]{
    \begin{tikzpicture}[scale=2] 
        \draw [dashed] (-3,0) -- (-2.8,0);
        \draw [dashed] (-2.1,0) -- (-1.1,0);
        \draw [dashed] (-0.5,0) -- (0.4,0);
        \draw [dashed] (1.1,0) -- (2.3,0);
        \draw [dashed] (2.8,0) -- (3,0);
        
        \draw [red,very thick] (-2.8,0) -- (-2.1,0) node [yshift=-0.3cm,xshift=1cm] {\large$\cdots$};
        \draw [red,very thick] (-1.1,0) -- (-0.5,0);

        \draw [fill=gray] (-2.8,0) circle (0.03cm) node [yshift=0.35cm] {\small$a^1_\e$};
        \draw [fill=gray] (-2.1,0) circle (0.03cm) node [yshift=-0.35cm] {\small$a^2_\e$};
        \draw [fill=gray] (-1.1,0) circle (0.03cm) node [yshift=0.35cm] {\small$a^{2N-1}_\e$};
        \draw [fill=gray] (-0.5,0) circle (0.03cm) node [yshift=-0.35cm] {\small$a^{2N}_\e$};
       
        \draw [red,very thick] (2.8,0) -- (2.3,0);
        \draw [red,very thick] (1.1,0) node [yshift=0.3cm,xshift=1.2cm] {\large$\cdots$} -- (0.4,0) ;
        
        \draw [fill=gray] (2.8,0) circle (0.03cm) node [yshift=0.35cm] {\small$a^{2N+2M}_\e$};
        \draw [fill=gray] (2.3,0) circle (0.03cm) node [yshift=-0.35cm] {\small$a^{2N+2M-1}_\e$};
        \draw [fill=gray] (1.1,0) circle (0.03cm) node [yshift=0.35cm] {\small$a^{2N+2}_\e$};
        \draw [fill=gray] (0.4,0) circle (0.03cm) node [yshift=-0.35cm] {\small$a^{2N+1}_\e$};

        \filldraw [black] (0,0) circle (0.05cm) node [yshift=0.35cm] {$0$};
    \end{tikzpicture}
    }

    \vspace{0.3cm}
    
    \subfloat[Case (ii) \label{subfig:case_2}]{
        \begin{tikzpicture}[scale=2] 
        \draw [dashed] (-3,0) -- (-2.8,0);
        \draw [dashed] (-2.5,0) -- (-1.8,0);
        \draw [dashed] (-1,0) -- (-0.4,0);
        \draw [dashed] (0.3,0) -- (0.9,0);
        \draw [dashed] (1.5,0) -- (2.3,0);
        \draw [dashed] (2.6,0) -- (3,0);

        \draw [red,very thick] (-2.8,0) -- (-2.5,0) node [yshift=-0.3cm,xshift=0.7cm] {\large$\cdots$};
        \draw [red,very thick] (-1.8,0) -- (-1,0);
        \draw [red,very thick] (-0.4,0) -- (0.3,0);
        \filldraw [black] (0,0) circle (0.05cm) node [yshift=0.35cm] {$0$};

        \draw [fill=gray] (-2.8,0) circle (0.03cm) node [yshift=0.35cm] {\small$a^1_\e$};
        \draw [fill=gray] (-2.5,0) circle (0.03cm) node [yshift=-0.35cm] {\small$a^2_\e$};
        \draw [fill=gray] (-1.8,0) circle (0.03cm) node [yshift=0.35cm] {\small$a^{2N-1}_\e$};
        \draw [fill=gray] (-1,0) circle (0.03cm) node [yshift=-0.35cm] {\small$a^{2N}_\e$};
        \draw [fill=gray] (-0.4,0) circle (0.03cm) node [yshift=0.35cm] {\small$a^{2N+1}_\e$};
        \draw [fill=gray] (0.3,0) circle (0.03cm) node [yshift=-0.35cm] {\small$a^{2N+2}_\e$};
       
        \draw [red,very thick] (2.6,0) -- (2.3,0);
        \draw [red,very thick] (1.5,0) node [yshift=-0.3cm,xshift=0.8cm] {\large$\cdots$} -- (0.9,0) ;
        
        \draw [fill=gray] (2.6,0) circle (0.03cm) node [yshift=-0.35cm] {\small$a^{2N+2M+2}_\e$};
        \draw [fill=gray] (2.3,0) circle (0.03cm) node [yshift=0.35cm] {\small$a^{2N+2M+1}_\e$};
        \draw [fill=gray] (1.5,0) circle (0.03cm) node [yshift=-0.35cm] {\small$a^{2N+4}_\e$};
        \draw [fill=gray] (0.9,0) circle (0.03cm) node [yshift=0.35cm] {\small$a^{2N+3}_\e$};

    \end{tikzpicture}
    }

    \vspace{0.3cm}
    
    \subfloat[Case (iii) with $n_2=0$ \label{subfig:case_3}]{
        \begin{tikzpicture}[scale=2] 
        \draw [dashed] (-0.3,0) -- (-0.9,0);
        \draw [dashed] (-1.5,0) -- (-2.3,0);
        \draw [dashed] (-2.6,0) -- (-3,0);
        \draw [dashed] (2.8,0) -- (3,0);

        \draw [red,very thick] (-2.6,0) -- (-2.3,0) node [yshift=-0.3cm,xshift=0.7cm] {\large$\cdots$};
        \draw [red,very thick] (-1.5,0)  -- (-0.9,0) ;
        \draw [red,very thick] (-0.3,0) -- (2.8,0);

        \filldraw [black] (0,0) circle (0.05cm) node [yshift=0.35cm] {$0$};

        \draw [fill=gray] (-2.6,0) circle (0.03cm) node [yshift=-0.35cm] {\small$a^1_\e$};
        \draw [fill=gray] (-2.3,0) circle (0.03cm) node [yshift=0.35cm] {\small$a^{2}_\e$};
        \draw [fill=gray] (-1.5,0) circle (0.03cm) node [yshift=-0.35cm] {\small$a^{2N-1}_\e$};
        \draw [fill=gray] (-0.9,0) circle (0.03cm) node [yshift=0.35cm] {\small$a^{2N}_\e$};
        \draw [fill=gray] (-0.3,0) circle (0.03cm) node [yshift=-0.35cm] {\small$a^{2N+1}_\e$};

    \end{tikzpicture}
    }
 \caption{The jumping set after gauge transformation in cases (i), (ii), and (iii).}
 \label{fig:generalizzazione}
\end{figure}

\bigskip\noindent {\bf Acknowledgments.}  B. Noris has been supported
by the MUR grant Dipartimento di Eccellenza 2023-2027 and by the
INdAM-GNAMPA Project 2022 ``Studi asintotici in problemi parabolici ed
ellittici''. V. Felli, R. Ognibene, and G. Siclari are partially
supported by the INDAM-GNAMPA Project 2022 ``Questioni di esistenza e
unicit\`a per problemi non locali con potenziali di tipo
Hardy''. R. Ognibene is supported by the project ERC VAREG -
\emph{Variational approach to the regularity of the free boundaries}
(grant agreement No. 853404).

\bibliographystyle{acm}

\end{document}